\documentclass{article}
\usepackage{amsmath,amssymb,amsthm}
\newtheorem{theorem}{Theorem}

\newtheorem{lemma}{Lemma}

\newtheorem{remark}{Remark}
\newtheorem{question}{Question}
\newcommand{\hkr}{{H}}

\newcommand{\E}{{\mathbb{E}}}
\newcommand{\vol}{{\rm vol}}
\newcommand{\dist}{{\rm dist}}
\newcommand{\cut}{{\rm cut}}
\newcommand{\disc}{{\rm disc}}
\newcommand{\supp}{{\rm supp}}
\newcommand{\sign}{{\rm sign}}
\newcommand{\F}{{\mathcal F}}
\newcommand{\ignore}[1]{}

\begin{document}
\title{From quasirandom graphs to graph limits and graphlets }
 
\author{Fan Chung\thanks{Research supported in part by ONR MURI N000140810747, and AFSOR AF/SUB 552082.}\\
Department of Mathematics\\
University of California, San Diego\\
La Jolla, CA 92093\\
fan@ucsd.edu
}
\date{\empty}
\maketitle

\begin{abstract}
We generalize the notion of quasirandomness which concerns a class of equivalent properties that random graphs satisfy.
We  show that the convergence of a graph sequence
under the  spectral distance is equivalent to the convergence using the (normalized) cut distance.  The resulting graph limit is called graphlets.
We then consider  several   families of graphlets and, in particular,  we characterize quasirandom graphlets with low ranks for both dense and sparse graphs. For example, we show that a graph sequence $G_n$,  for $n = 1, 2,  \ldots $, converges to a graphlets of rank $2$, (i.e.,all normalized
eigenvalues $G_n$ converge to $0$ except for two  eigenvalues converging to $1$ and $\rho >0$)
 if and only if the graphlets is the union of $2$ quasirandom graphlets.

 \end{abstract}
\setlength{\parskip}{8pt}
\section{Introduction}\label{sec:1}
 The study of graph limits originated from quasi-randomness of
graphs which concerns  large equivalent families of graph properties that random graphs satisfy.
Lov\'asz and S\'os \cite{lsos}  first considered a generalized notion of  quasi-randomness as the limits of graph sequences.
Since then, there have been a great deal of 
 developments \cite{aldous, bs, bsr, borgs, borgs1, dj, dhj, elek, elek1, ha, hn,lo, lsos, lsz,lsz1, lsz2, lsz3, lsz4,sz} on the topic of graph limits. There are  two very distinct approaches.
The study of graph limits for dense graphs is entirely different from that for sparse graphs. By dense graphs, we mean  graphs on $n$ vertices with $c n^2$ edges for some constant $c$.  For a graph sequence of dense graphs, the graph limit is formed by  taking the limit of the adjacency matrices with entries  of each matrix 
 associated with squares of equal sizes which partition $[0,1]\times[0,1]$ (see \cite{lsz, lsz1}). Along this line of approach, the graph limit of a sparse graph sequence converges to
  zero.  Consequently,  very different approaches were developed for graph limits of very sparse graphs, mostly with vertex  degrees bounded above
  by a constant independent of the size of the graph \cite{ bs, br,  elek}. 
  
\ignore{  Although graphs are very different from Riemannian manifolds, there are apparent analogies  based on the fact that  a manifold can be viewed as the  limit of some increasing refined discretizations. Various types of discretizations can often be formulated as sparse graphs. For example, a graph sequence of 
  $n$-paths  converges to a  line which can be represented as an interval $[0,1]$ and this is a building block for a Riemannian manifold. Although we will have some examples
  of which the graph limits are indeed Riemannian manifolds, the graph limits we focus on here are not necessarily
  manifolds.  }
  
   To distinguish from earlier definitions for graph limits (called, graphons, graphines, etc.), we will call the graph limits in this paper  by the name of 
  {\it graphlets} to emphasize the spectral connection. In the subsequent sections, we will give a detailed definition for  graphlets as the graph limits of  given graph sequences.
  Although the terminology is sometimes similar to that in differential geometry, the definitions are along the line of spectral graph theory \cite{ch0}
  and mostly discrete.  In addition, the orthogonal basis of the graphlets of a graph sequence can be used, with additional scaling parameters,
  to provide a universal basis for all graphs in the domain (or the union of domains) that we consider. In this regard, graphlets 
  play a similar role as the wavelets do for affine spaces.

To study the convergence of a graph sequence, various different metrics come into play for comparing two graphs.   For two given graphs, there are many different ways to define some notion of  distance between them. Usually the  labeling map assigns consecutive integers to the vertices of a graph   which can then  be associated with  equal
intervals which  partition  $[0,1]$. As opposed to  the definitions in previous work, we will not use the
usual measure or metric on the interval $[0,1]$. Instead, our measure on $[0,1]$ will be determined by the graph sequences that we consider.
\ignore{With the labeling maps, we can proceed to compare two graphs  by considering the adjacency matrices (as well as the Laplacian) as operators acting on functions
defined on $[0,1]$. It should be noted that the metric of the graph limit $\Omega$ depends on the graph sequence and can be quite different from the usual Hausdorff metric
on $[0,1]$.  This will be discussed further in the Section \ref{geo} concerning the gradient and the geometry of the graphlets.}
Before we proceed to examine the distance between two graphs,
we remark that there is a great deal of work on  distances between manifolds \cite{bbg,g} via isometric embeddings. Although the details are obviously
different, there are similarities in the efforts for identifying the global structures of the objects of interest. We are using elements of $[0,1]$ as labels for the
(blow-up) vertices, similar to the  exchangeable probabilistic measures that were used in \cite{aldous, dj, dhj, defi, defi1,  hoover}.

Several metrics for defining distances between two graphs originated from  the quasi-random class of graphs \cite{ CG(d),CGW89}. One such example involves the {\it subgraph counts},  concerning  the number of induced (or not necessarily induced) subgraphs of $G$ that are isomorphic to a specified graph $F$. Another such metric is called
the {\it cut metric} which came from discrepancy inequalities for graphs.  The usual discrepancy
inequalities in a graph $G$ concern  approximating  the number of edges between two given subsets of vertices by the expected values  as in a random graph
 and therefore such discrepancy inequalities can be regarded as estimates for  the distance of a graph to a random graph.  For dense graphs, the equivalence of convergence under the subgraph-count metric and the cut metric among others are  well understood (see \cite{borgs, lsz}). 
The methods for dealing with dense graph limits have not been effective so far for dealing with sparse graphs.  A different separate set of metrics has been developed \cite{br,elek} using  local structures in the neighborhood of each vertex.
Instead of subgraph counts, the associated metric concerns counting trees  and local structures in the ``balls'' around each vertex. The problems of graph limits for sparse graphs are inherently harder as shown in \cite{br}. Nevertheless, most real world complex networks are sparse graphs and the study of  graph limits for sparse graphs can be useful for understanding the dynamics  of large information networks.

The paper is organized as follows: In Section \ref{def}, we first examine the convergence of degree distributions of graphs and  we consider
the convergence the discrete Laplace operators   under the  spectral norm. Then we give the definition for  graphlets in
Section \ref{lets}. 
\ignore{ In Section \ref{geo}, we proceed to examine the geometry of graphlets by considering the gradient operator and the heat kernel of a graph.}  In Section \ref{ex},
we give several families of examples, including  dense graphlets, 
\ignore{graphlets of paths, cycles, cartesian products of paths/cycles,}
 quasi-random graphlets, bipartite quasi-random
graphlets and graphlets of bounded rank. 
In Section \ref{2norms}, we consider the discrepancy distance between two graphs which can be viewed as a normalization of  the cut distance. Then we  prove the equivalence of the  spectral distance and the normalized cut distance
for both dense and sparse graphs. Note that our definition of the discrepancy distance  is  different from the cut distance as used in \cite{br} where a negative result about a similar equivalence was given.  In Sections \ref{quasi} and \ref{bquasi}, we further examine quasi-random graphlets and bipartite quasi-random graphlets for graph sequences with general degree distributions.
In Sections \ref{r2} and \ref{rk}, we  give a number of
equivalent properties for certain graphlets of rank $2$ and for general $k$.  
 In Section \ref{remarks}, we briefly discuss connections between the discrete
and continuous, further applications in finding communities in large graphs and possible future work that this paper might lead to.

 We remark that the work here is different from the spectral approach of graph limits  which focuses on the spectrum of the limit of
 the adjacency matrices in \cite{sz}. If the graph limit is derived from a graph sequence which consists  of dense and  almost regular graphs, the two spectra are essentially the same (differ  only by a scaling factor).
 However, a subgraph of a regular graph is not necessarily regular. All theorems in this paper hold for general graph sequences for both dense graphs
 and sparse graphs. Some of the methods here can be generalized to weighted directed graphs which will   not be discussed in this paper. 
 
\section{The spectral norm and spectral distance}
\label{def}
For a weighted graph $G=(V,E)$ with vertex set $V$ and edge set $E$, we denote the adjacency matrix by $A_G$ with rows and columns indexed by vertices in $V$. For an edge $\{u,v\} \in E$, the edge weight is denoted by $A_G(u,v)$.   
For a vertex $v$ in $V(G)$, the degree of $v$ is $d_G(v)=\sum_u A_G(u,v)$. We let $D_G$ denote the diagonal matrix with $D_G(v,v)=d_G(v)$.  Here we consider   graphs without isolated vertices. Therefore, we have $d_G(v) > 0$ for every $v$ and $D_G^{-1}$ is well defined. 

\ignore{
We consider a family of operators ${\mathcal W}$ consisting of  $W : [0,1] \times [0,1] \rightarrow [0,1] $ satisfying ,  $W(x,y)=W(y,x)$.
\ignore{
(ii) for each $x \in [0,1]$, the neighborhood $\Gamma_W(x)$ of $x$ in $W$ is defined by 
\[ \Gamma_W(x)=\{y : W(x,y) \not = 0\}.\]
The degree of $x$ in $W$, denoted by $d_W(x)$  satisfies
 \[
d_W(x) = \int_{\Gamma_W(x)} W(x,y) dy_x \not = 0. \]
} In this paper, we mainly concern operators $\mathcal W$ that are
{\it exchangeable} (see \cite{aldous, dj, dhj, defi, defi1,  hoover}). Namely, for any measure-preserving bijection $\tau: [0,1] \rightarrow [0,1],$ $W$ is said to be equivalent to a rearrangement  of $W$, denoted by $W_{\tau}$, which acts on
functions  defined on $[0,1]$  satisfying  $Wf(x) =W_\tau f (\tau(x))$ for $f : [0,1] \rightarrow {\mathbb R}$.
We 
write $W \sim W_{\tau}$. 
}

We consider   the family of operators $\mathcal W$ consisting of $W : [0,1] \times [0,1] \rightarrow [0,1] $ satisfying ,  $W(x,y)=W(y,x)$.
 $W$ is said to be
 of finite type if there is a finite partition $ (S_1,...,S_n) $ of $[0,1]$ such that $W$ is constant on each set $S_i \times S_j$ . 
 Given a graph $G_n$ on $ n$ vertices,  a special 	 finite-type  associated with $ G_n$ is defined by partitioning $[0,1]$ into $n$ intervals of length $1/n$ and, for a map $\eta : [0,1] \rightarrow V $,  the pre-image of each vertex $v$ corresponds to a interval $I_v = (j/n, (j+1)/n]$ for some $j$. We  can define $W_{G_n} \in \mathcal W$ by setting :
\begin{eqnarray}
\label{ww}
W_{G_n}(x,y)=A_{G_n}(u,v)\end{eqnarray} if $
 x \in I_{u}, $ and $y \in I_{v}$.

Suppose we have a sequence of graphs, $G_n,$ for $ n = 1, 2, \ldots$. Our goal is to describe the limit of a graph sequence provided it converges.
One typical way,  as seen in \cite{lsz}, is to take the limit of $W_{G_n}$ under the cut norm. For example, if $G_n$ is in the family of random graphs with edge density $1/2$, the limit of $W_{G_n}$  has all entries $1/2$. However, if we consider sparse graphs such as cycles, then the limit of $W_{G_n}$ converges  to  
the $0$ function.

 Instead, we will define the graph limit to be associated with a measure space $\Omega$  as the limit of measure spaces defined on 
 $G_n$ and the measure $\mu$ for $\Omega$
is the limit of the measures $\mu_n$
associated with $G_n$. Before we give the detailed definitions of $\Omega$ and $\mu$, there are a number of technical issues in need of clarification. The following remarks can be regarded as a companion for the definitions to be given in Sections 2.1 to 2.3 so that  possible misinterpretations could be avoided.

\begin{remark} 
{\rm
We label elements of $\Omega$  by $[0,1]$. However, the geometric structure of $\Omega$ can be quite different
from the interval $[0,1]$.
In general, $\Omega$ can be some  complicated compact space. For example, if the $G_n$ are square grids
(as cartesian products of two paths), then a natural choice  for $\Omega$ is a unit square. 
We will write $V(\Omega)=[0,1]$ to denote the set of ``labels''  for $\Omega$ while $\Omega$ can have  natural descriptions other than $[0,1]$.
}
\end{remark}
\begin{remark}
{\rm
In this paper, we mainly concern operators $ W$ that are 
{\it exchangeable} (see \cite{aldous, dj, dhj, defi, defi1,  hoover}). Namely, for a Lebesgue measure-preserving bijection $\tau: [0,1] \rightarrow [0,1],$ 
 a 
rearrangement of $W$, denoted by $W_{\tau}$,  acts on
functions $f$  defined on $[0,1]$  satisfying  
\begin{eqnarray}
\label{ex}
Wf(x) =W_\tau f (\tau(x)). 
\end{eqnarray}
We say $W$ is equivalent to $W_{\tau}$ and we
write $W \sim W_{\tau}$.  By an  exchangeable operator $W$, we mean the  equivalence class of  operators $W_{\tau}$ where $\tau$ ranges over all
measure-preserving bijections on $[0,1]$.
}
\end{remark}
\begin{remark}{\rm
We consider a family of exchangeable self-adjoint operators ${\mathcal W}^*$ which act on the space of functions $f : [0,1] \rightarrow {\mathbb R}$.
 Clearly, any exchangeable  $W : [0,1] \times [0,1] \rightarrow [0,1] $ with  $W(x,y)=W(y,x)$  is contained in ${\mathcal W}^*$.  The disadvantage of  using  such $W$ is
 the implicit requirement that $W(x,y)$ is supposed to be given as a specified value. For   some graph sequences $G_n$ which converge to a finite graph, it is quite
 straightforward to define the associated $W_n$ as in (\ref{ww}). However, in general,  it is quite possible that $W_n(x,y)$ as a function of $n$  approach $0$ as $n$ goes to 
 infinity. In such cases, it is better to treat  the limit as an operator. }
 \ignore{For example, for the graph sequence of paths $P_n$ on $n$ vertices, the associated
 measure $\mu_n$, as defined in (\ref{mu_n}),
 converge to the Lebesgue measure on $[0,1]$ and the operator $\Delta_n$, as defined in (\ref{Del}), converge to the the classical Laplace-Beltrami operator $\Delta$ (which
 is hard to be expressed in $\mathcal W$).}
\end{remark}
\begin{remark} 
{\rm
 Throughout the paper, 
$ \int F(y)dy$ denotes the usual integration of a function $F$ subject to the Lebesgue measure $\nu$.
We will impose the condition that the space of functions that we focus on are Lebesgue measurable and integrable so that all the inner products
involving integration  make sense. For some other measures, such as  $\mu_n$ and $\mu$ for a graph sequence, as defined in Section \ref{lap} and \ref{subdeg},
it can be easily checked that if a function $F$ is Lebesgue measurable and integrable then $F$ is also measurable and integrable subject to
$\mu_n$ and $\mu$.
}
\end{remark}
\subsection{
 The Laplace operator on a graph}\label{lap}
 For a weighted graph $G_n$ on $n$ vertices with edge weight $A_n(u,v)$ for vertices $u$ and $v$ , we define the Laplace operator $\Delta_n$ to be
 \begin{eqnarray}\label{Del}
  \Delta_n f (u)& =& \frac{1}{d_u}\sum_{v } (f(u)-f(v))A_n(u,v).\end{eqnarray}
  for $f : V \rightarrow {\mathbb R}$.
 It is easy to check  that
 \begin{eqnarray*}
 \Delta_n &=&I_n- D_n^{-1}A_n =D_n^{-1/2}{\mathcal L}_nD_n^{1/2}
 \end{eqnarray*}
 where  $I_n$ is the $n \times n$ identity matrix and ${\mathcal L}_n= I_n-D_n^{-1/2}A_n D_n^{-1/2}$ is the symmetric normalized Laplacian (see \cite{ch0}).

Let $\mu_n$ denote the measure defined by $\mu_n(v)=d_v/\vol(G)$ for $v$ in $G_n$ where $\vol(G_n)=\sum_v d_v$.   We define an inner product on functions $f, g: V \rightarrow {\mathbb R}$ by
 \[ \langle f, g \rangle_{\mu_n} = \sum_{v \in V} f(v)g(v) \mu_n(v). \]
 It is then straightforward to check that
 \begin{eqnarray*}
 \sum_{\{u,v\} \in E} \frac{ (f(u)-f(v))(g(u)-g(v) )A_n(u,v)}{\vol(G_n)}&=&
 \frac{\sum_u f(u) \sum_{v \sim u} (g(u)-g(v))A_n(u,v)}{\vol(G_n)}\\
 &=& \sum_u f(u) (\Delta_n g)(u)\frac{ d_u}{\vol(G_n)}\\
 &=& \langle f, \Delta_n g \rangle_{\mu_n}
 \end{eqnarray*}
 and 
 \[\langle f, \Delta_n g \rangle_{\mu_n}=\langle g, \Delta_n f \rangle_{\mu_n}. \]
 If $f$ and $g$ are complex-valued functions, then we have
 \[\langle f, \Delta_n g \rangle_{\mu_n}=\overline{\langle g, \Delta_n f \rangle_{\mu_n}} \]
 where $\bar x $ denotes the complex conjugate of $x$.

We note that 
$\langle f, \Delta_n {\mathbf 1} \rangle_{\mu_n}=\langle {\mathbf 1}, \Delta_n f \rangle_{\mu_n}=0$, where ${\mathbf 1}$ denotes  the constant function $1$.
Therefore,
$\Delta_n$ has an eigenvalue $0$ with an associated eigenfunction $\mathbf 1$, under the $\mu_n$-norm.  
The eigenfunctions $\phi_j$, for $j=0, \ldots, n-1$, form an orthogonal basis under the  $\mu_n$-norm for $G_n$. 
In other words, $D_n^{1/2}\phi_j$ form an orthogonal basis under the usual inner product as eigenvectors for the normalized Laplacian $I_n-D_n^{-1/2} A_n D_n^{-1/2}$. The $\phi_j$'s are previously called the
combinatorial eigenfunctions in \cite{ch0}.

.
\subsection{ The convergence of degree distributions}
\label{subdeg}
  Suppose we have a sequence of graphs.   For a graph $G_n$ on $n$ vertices,  the measure $\mu_n$, defined by  $ \mu_n(v) =\frac{d_{n}(v)}{\vol(G_n)}$, is also called
  the degree distribution of $G_n$ where
 $d_n(v)$ denotes the degree of $v$ in $G_n$ and $\vol(G_n)= \sum_v d_n(v).$ In general, for a subset $X$ of vertices in $G_n$,
$\vol_{G_n}(X) = \sum_{v \in X} d_n(v)$. 
In this paper, we focus on graph sequences with convergent degree distributions which we will describe.

For a graph $G_n$ with vertex set $V_n$ consisting of  $n$ vertices,  we let  $F_n$ 
denote the set of all bijections from $V_n  $ to $\{1,2, \ldots, n\}$. 
\begin{eqnarray}
\label{fn}
 F_n = \{ \eta : V_n \rightarrow \{1,2, \ldots, n\} \}. 
 \end{eqnarray}
For
each $\eta \in F_n$,  we let $\eta_n$  denote the associated partition map $\eta_n : [0,1] \rightarrow  V_n  $, defined by
 $\eta_n(x)=\eta(u)$
 if $x \in ((\eta(u)-1)/n, \eta(u)/n] =I_{\eta(u)}$.  We  write $I_{\eta(u)}=I_u$ if there is no confusion. In stead of  $F_n$, it is sometimes  convenient to
 consider
\begin{eqnarray}
\label{fn1}
{\mathcal F}_n=\{\eta_n : \eta \in F_n\} \subset \{ \varphi: [0,1] \rightarrow V_n\}
\end{eqnarray}
 Now, for any integrable functions $f,g : [0,1] \rightarrow {\mathbb R}$,  we  define 
 \begin{eqnarray} \label{mu_n}
 \langle f,  g\rangle_{\mu_n,\eta_n} = \int_0^1 f(x)g(x) \mu_n^{(\eta_n)}(x)
 \end{eqnarray}
 where  $\mu_n$ is defined by
   \begin{eqnarray}
\label{eta}
\int_0^1F(x) \mu_n^{(\eta_n)}(x) = \int_0^1F(x) n \mu_n(\eta_n( x) )dx
 \end{eqnarray} for integrable $F : [0,1] \rightarrow {\mathbb R}$.  
 We can then define the   associated norm:
 \begin{eqnarray} \label{mu_norm}
\|f\|_{\mu_n, \eta_n} &=& \sqrt{ \langle f,  f\rangle_{\mu_n,\eta_n} }
 \end{eqnarray}
 
As a measure on $[0,1]$,  $\mu^{(\eta_n)}_n$ satisfies
\begin{align}
\label{muu} \mu_n(u) = \int_{I_u} \mu^{(\eta_n)}_n(x) 
\end{align}
and 
\[ \int_0^1 \mu^{(\eta_n)}_n(x)= 1. \] 
For example, for a graph  $G_5$ with degree sequence $(2,2,3,3,4)$, and suppose the corresponding vertices are denoted by $v_1, \ldots, v_5$, then  $\mu_n(v_1)=\mu_n(v_2)= 1/7$ and $\mu_n(v_3)=\mu_n(v_4)=3/14$, etc.

 In particular, for a subset $S \subseteq [0,1]$, we consider the characteristic function $\chi_S(x)=1$ if $x \in S$ and $0$ otherwise.
 Then 
 for $f=g=\chi_S$, we have
 \begin{eqnarray}
 \langle \chi_S,  \chi_S\rangle_{\mu_n,\eta_n} &=& \mu_n^{(\eta_n)}(S)\nonumber\\
 &=& \int_S \mu_n^{(\eta_n)}(x). 
 \end{eqnarray}
Sometimes we suppress the labeling map $\eta_n$ and simply write $\mu_n$ as the associated measure on $[0,1]$ if there is no confusion.

For $\epsilon > 0$, we say 
two graphs $G_m$ and $G_n$ have $\epsilon$-similar degree distributions if
 \begin{eqnarray}
 \label{deg}
 \inf_{\theta \in \F_m, \eta \in \F_n } \int_0^1 |\mu_m^{(\theta)}(x) - \mu_n^{(\eta)}(x)|  < \epsilon.
 \end{eqnarray}

For a graph sequence $G_n, n=1,2, \dots$, we say the degree distribution $\mu_n$ is  Cauchy,  if for any $\epsilon>0$, there 
 exists $N=N(\epsilon)$ such that for any $m,n \geq N$, the degree distributions  of $G_m$ and $G_n$ are $\epsilon$-similar. 
 To see that the degree distributions
 converge, we use the following
 arguments:
\begin{lemma}
\label{muconv}
If the degree distribution of  the sequence $G_n$ is  Cauchy, then there are $\theta_n \in \F_n$ such that 
the sequence $\mu_n^{(\theta_n)}$ of $G_n$  converges to a limit, denoted by $\mu$. Furthermore $\mu$ is unique
up to a measure preserving map.
\end{lemma}
\proof
\ignore{
For each $n$, suppose we choose $\eta_n$ such that $\mu_n^{(\eta_n)}$ is a non-decreasing function on $[0,1]$.
By using  the fact that for $x_1 < x_2$ and $y_1<y_2$, we have $|x_1-y_1|+|x_2-y_2| \leq |x_1-y_2|+|x_2-y_1|$,
it follows that
 \begin{eqnarray*}
 \inf_{\theta \in F_m, \eta \in F_n } \int_0^1 |\mu_m^{(\theta)}(x) - \mu_n^{(\eta)}(x)| =
 \int_0^1 |\mu_m^{(\eta_m)}(x) - \mu_n^{(\eta_n)}(x)| 
 \end{eqnarray*} 
Thus the sequence $\mu_n^{(\eta_n)}$ is Cauchy and therefore  converges to  a limit $\mu$. }

For each positive integer $j$, we set  $\epsilon_j=2^{-j}$, and  let $ N(\epsilon_j)$ denote the least integer such that for $m,n \geq N(\epsilon_j)$, 
$G_m$ and $G_n$ have $\epsilon_j$-similar degree distributions.  To simplify the notation, we write $M(j)=N(\epsilon_j)$.

We first choose an arbitrary permutation $\eta_{M(1)}$ and then   by induction define  permutations $\eta_{M(j)}$'s, for $j>1$ using (\ref{deg})  so that 
\[ \int_0^1 |\mu_{M(j)}^{(\theta_{M(j)})}(x) - \mu_{M(j+1)}^{(\theta_{M(j+1)})}(x)| < \epsilon_j.
\]

For  each $n \in [M(j), M(j+1))$, we choose the  permutation $\eta_n$ such that
\[ \int_0^1 |\mu_{n}^{(\theta_{n})}(x) - \mu_{M(j)}^{(\theta_{M(j)})}(x)| < \epsilon_j. 
\]

\noindent
{\it Claim:} the sequence of  $\mu_n^{(\theta_n)}$, for $n =1, 2, \ldots$  is Cauchy.
\\
To prove the claim, we see that for any $m ,n  \geq M(j) $ satisfying $n \in [M(j), M(j+1)) $ and $m \in [M(k), M(k+1))$ with $j \leq k$, we have
\begin{eqnarray*}
&&
\int_0^1 \mid \mu_{n}^{(\theta_{n})}(x) - \mu_{m}^{(\theta_{m})}(x)\mid \\
&\leq& \int_0^1 \mid \mu_{n}^{(\theta_{n})}(x) - \mu_{M(j)}^{(\theta_{M(j)})}(x)\mid +\int_0^1 \mid \mu_{M(j)}^{(\theta_{M(j)})}(x) - \mu_{M(j+1)}^{(\theta_{M(j+1)})}(x)\mid + \ldots \\
&&+\int_0^1 \mid \mu_{M(k-1)}^{(\theta_{M(k-1)})}(x) - \mu_{M(k)}^{(\theta_{M(k)})}(x)\mid +\int_0^1 \mid \mu_{M(k)}^{(\theta_{M(k+1)})}(x) - \mu_{m}^{(\theta_{m})}(x)\mid \\
&\leq& 2\epsilon_j + \epsilon_{j+1} + \ldots + \epsilon_{k-1} + 2 \epsilon_k\\
&=& 3 \epsilon_j
 \end{eqnarray*}
and the Claim is proved.

To show that the sequence
  $ \mu_{n}^{(\theta_{n})} $ converges, we define $\mu(S)$ for any measurable subset $S \subseteq [0,1]$ as follows:
    \begin{eqnarray*}
  \mu(S) &=& \lim_{n \rightarrow \infty}\mu_{n}^{(\theta_{n})}(S)\\
  &=& \lim_{n \rightarrow \infty} \int_S \mu_{n}^{(\theta_{n})}(x).
    \end{eqnarray*}
  Since $\mu_{n}^{(\theta_{n})}$ is Cauchy, the above limit exists and $\mu(S)$ is well defined. Furthermore,  for any measure preserving map $\tau$,
  $\mu \circ \tau$ is  the limit of $\mu_{n}^{(\theta_{n}\circ \tau)} $. Thus, $\mu$ is unique up to a measure preserving map.

  \ignore{ are all bounded functions defined on $[0,1]$ and the sequence of functions $ \mu_{n}^{(\theta_{n})} $ is Cauchy,  they  converge to a  limit, denoted by $\mu$ on $[0,1]$. For different choice of  $\eta_{M(1)}$ we may arrive at different measure which, however, is equivalent to $\mu$
up to a measure preserving map.
}
 To see  that  $\mu$ is a probabilistic measure, we note that  for any $\epsilon > 0$, there is some $n$ such that
\begin{eqnarray*}
 \mid \int_0^1 \mu(x) -1\mid &= &\mid \int_0^1 \mu(x)-\int_0^1 \mu_{n}^{(\eta_n)} (x) \mid\\
 &= & \int_0^1\mid \mu(x)- \mu_n^{(\eta_n)}\mid\\
&\leq & \epsilon.
\end{eqnarray*}
 Lemma \ref{muconv} is proved.
\qed

\begin{remark}{\rm
Since we are dealing with exchangeable operators, the measures 
 $\mu$ can be regarded as the equivalence class of   probabilistic measures
where two measures $\varphi, \varphi'$ are said to be equivalent if there is a Lebesgue measure preserving bijection  $\tau$ on $ [0,1]$ such that
$\varphi =\varphi' \circ \tau$.
}
\end{remark}
\begin{remark}{\rm
An alternative proof for the convergence $\mu_n$ is due to Stephen Young \cite{young} which is simpler but the resulted limit $\mu$ is not
necessarily exchangeable.
For each $n$, suppose we choose $\eta_n$ such that $\mu_n^{(\eta_n)}$ is a non-decreasing function on $[0,1]$.
By using  the fact that for $x_1 < x_2$ and $y_1<y_2$, we have $|x_1-y_1|+|x_2-y_2| \leq |x_1-y_2|+|x_2-y_1|$,
it follows that
 \begin{eqnarray*}
 \inf_{\theta \in \F_m, \eta \in \F_n } \int_0^1 |\mu_m^{(\theta)}(x) - \mu_n^{(\eta)}(x)| =
 \int_0^1 |\mu_m^{(\eta_m)}(x) - \mu_n^{(\eta_n)}(x)| 
 \end{eqnarray*} 
Thus the sequence $\mu_n^{(\eta_n)}$ is Cauchy and therefore  converges to  a limit $\mu$. }
\end{remark}

 We note that two different graphs $G$ and $H$ both on $n$ vertices can have the same degree distribution  measure $\mu_n$  but $G$ and $H$ have
 different degree sequences. For example, $G$ is a $k$-regular graph and $H$ is a $k'$-regular graph where $k \not = k'$. In this case, $\mu_n(v)=1/n$
 for any vertex $v$ and $\mu_n(x)=1$ for  any $x \in [0,1]$.  To  define the convergence of graph sequences, we need to take into account the volume $\vol(G)=\sum_{v} d_v$ of $G$.

 \subsection{The  spectral distance }

 Suppose we consider two graphs $G_m$ and $G_n$ on $m$ and $n$ vertices, respectively. Their associated Laplace operators are denoted by
$\Delta_m$ and $\Delta_n$, respectively. If $m \not = n$, $\Delta_m$ and $\Delta_n$  have different sizes. In order to compare 
 two given matrices, we need some definitions.

In $G_n=(V_n,E_n)$, for $\eta_n \in \F_n$ (as described in (\ref{fn})), the operator $\Delta_n^{(\eta)} $ is acting
on  an integrable  function $f : [0,1] \rightarrow {\mathbb R}$ by
\begin{eqnarray}
\Delta_n^{(\eta)} f(x)
&=& \frac n {d_n(\eta_n(x))} \int_0^1\big(f(x)-f(y)\big) W_n^{(\eta_n)}(x,y)dy \nonumber \\
&=& f(x) -\frac n {d_n(\eta_n(x))} \int_0^1 W_n^{(\eta_n)}(x,y)f(y)dy \label{del}
\end{eqnarray}
where $W^{(\eta_n)}_n \in {\mathcal W}=[0,1]\times[0,1]$ is associated with the adjacency matrix $A_n$ by $W^{(\eta_n)} (x,y)= A_n(\eta_n(x), \eta_n(y))$.
Here we require $f$ to be Lebesgue measurable and therefore is also $\mu_n$-measurable. In the remainder of the paper, we deal with functions that are
Lebesgue integrable on $[0,1]$.
We note that for any two permutations $\theta, \eta \in F_n$, $W^{(\theta_n)}$ is equivalent to $W^{(\eta_n)}$ as exchangeable operators in ${\mathcal W}^*$, defined
in (\ref{ex}).

  For a Lebesque integrable function $f : [0,1] \rightarrow {\mathbb R}$, we consider
     \begin{eqnarray*}
\langle f, \Delta_n^{(\eta_n)} g \rangle_{\mu_n,\eta_n}&=& \int_0^1f(x) \big( \Delta_n^{(\eta_n)} g(x)\big) \mu_n^{(\eta_n)} (x)\\
&=& \int_0^1\int_0^1\frac n {d_n(\eta_n(x))} f(x)\big(g(x)- g(y)\big) W_n^{(\eta_n)}(x,y)~dy~ \mu_n^{(\eta_n)} (x)
\end{eqnarray*}
Using (\ref{eta}), we have
\begin{eqnarray*}
\langle f, \Delta_n^{(\eta_n)} g \rangle_{\mu_n,\eta_n}
&=& \frac {n^2}{\vol(G_n)} \int_0^1\int_0^1 f(x)\big(g(x)- g(y)\big) W_n^{(\eta_n)}(x,y) dxdy\\
&=& \frac {n^2}{2\vol(G_n)} \int_0^1\int_0^1 \big(f(x)-f(y)\big)\big(g(x)- g(y)\big) W_n^{(\eta_n)}(x,y) dxdy.
  \end{eqnarray*}
 In particular,
 \begin{align} \label{eq99}
\langle f, \Delta_n^{(\eta_n)}f \rangle_{\mu_n,\eta_n}
&= \frac {n^2}{2\vol(G_n)} \int_0^1\int_0^1 \big(f(x)-f(y)\big)^2 W_n^{(\eta_n)}(x,y) dxdy.\end{align}

\begin{remark}{\rm
The above inner products are invariant subject to any choice of  measure preserving maps $\tau$. Namely,
if we define $f \circ \tau(x)= f(\tau(x))$, then}
\begin{eqnarray}
\label{inv}
\langle f, \Delta_n^{(\eta_n)}f \rangle_{\mu_n,\eta_n}=\langle f \circ \tau, \Delta_n^{(\eta_n \circ \tau)}f \rangle_{\mu_{n}\circ \tau,\eta_n \circ \tau}.\end{eqnarray}
\end{remark}

For an operator $M$ acting on the space of integrable functions $f : [0,1] \rightarrow {\mathbb R}$, we say $M$ is {\it exchangeable} if for any measure preserving map $\tau$, we have
\[ \langle f, M g \rangle = \langle f \circ \tau, M_\tau (g \circ \tau) \rangle \]
where $M_\tau$ is defined by $M_\tau h(x,y) = M(h (\tau^{-1}(x), h(\tau^{-1}(y))) $. 
Clearly, $\Delta_n$ is an exchangeable operator.

For an integrable function $f : [0,1] \rightarrow {\mathbb R}$ and $\eta_n \in \F_n$, we define
$\tilde f_n:  [0,1] \rightarrow {\mathbb R}$, for $x \in I_u$, as follows:
 \begin{eqnarray}\label{eq33}
\tilde f_n(x)= 
\frac{\int_{I_u} f(y) \mu_n^{(\eta_n)}(y)dy}{\int_{I_u} \mu_n^{(\eta_n)}(y)dy}&=&\frac{\int_{I_u} f(y) \mu_n^{(\eta_n)}(y)dy}{\mu_n(u)} \nonumber\\
&=&\frac{\int_{0}^1 I_n^{(\eta_n)}(x,y) f(y) \mu_n^{(\eta_n)}(y)dy}{\mu_n(u)} 
\end{eqnarray}
where $I_n$ is the $n \times n$  identity matrix as defined in Section 2.1.
Note that $\tilde f_n(x)=\tilde f_n(z)$ if $\eta_n(x)=\eta_n(z)$. For $u$ in $V(G_n)$, we write $\tilde f_n(u)= \tilde f_n(x)$ where $x \in I_u$.

$f$ and $\tilde f$ are related as follows:
\begin{lemma} \mbox{~~}\\
(i) For $x \in I_u$, and $\eta_n \in \F_n$,
\begin{eqnarray*}
 \Delta_n^{(\eta_n)}f(x) &=& \frac n {d_n(\eta_n(x))} \sum_{v} \int_{y \in I_v} (f(x)-f(y))A_n(\eta_n(x),v)dy\\
  &=& \frac 1 {d_n(u)} \sum_v  ( f (u)-\tilde f_n(v)) A_n(u,v)\\
  &=& \Delta_n \tilde f_n(u) + ( f (x)-\tilde f_n(u)) .
  \end{eqnarray*}
  (ii)  For $f, g : [0,1] \rightarrow {\mathbb R}$,
  \begin{eqnarray}\label{appro}
  \langle f, \Delta_n^{(\eta_n)} f \rangle_{\mu_n, \eta_n} &=& \langle \tilde f_n, \Delta_n  \tilde f_n \rangle_{\mu_n} + \| f - \tilde f_n\|^2_{\mu_n, \eta_n}.
  \end{eqnarray}
  \end{lemma}
  The proof of (i) follows from   (\ref{Del}) and (\ref{eq33}). (ii) follows from (i) and (\ref{eq99}) by straightforward manipulation.

\begin{remark}
{\rm
In this paper, we define inner products and norms on the space of integrable functions defined on $[0,1]$, as seen in (\ref{mu_n}) and (\ref{mu_norm}). Consequently, the last term in (\ref{appro}) approaches $0$ as
$n$ goes to infinity. Namely,
\[ \| f - \tilde f_n\|^2_{\mu_n, \eta_n} \rightarrow 0 ~~\mbox{as $n \rightarrow \infty$} \]
if $f$ is integrable.
This implies that the graph Laplacian $\Delta_n$  for $G_n$ acting on the space of functions defined on $V_n$
can be approximated by $\Delta_n^{(\eta_n)}$ acting on the space of functions defined on $[0,1]$ 
with the exception for the function $f$ with  $\|\Delta_n f\|_{\mu_n, \eta_n} $ is too close to $0$, while $f$ is   orthogonal to the eigenfunction associated with eigenvalue $0$. The case of  a path $P_n$ is one such example and in fact, the graph sequence
of paths $P_n$ does not converge under the spectral distance that we shall define.  In order to make sure that 
 $\Delta_n^{(\eta_n)}$  closely approximates $\Delta_n$, there are two ways to proceed.
We can restrict (implicitly) ourselves to graph sequences 
$G_n$ with the least nontrivial eigenvalue $\lambda_1 $ of $\Delta_n$ greater than some absolute positive constant
(as done in this paper). An alternative way is to consider general labeling space $\Omega_0$ other than $[0,1]$ and 
impose further conditions on the 
space of functions defined on $\Omega_0$  (which will be treated in a subsequent paper).
}
\end{remark}

     For a graph sequence $G_n=(V_n,E_n)$, where $n = 1, 2 , \ldots$,  we say the sequence of the Laplace operators  $\Delta_n$ is { Cauchy} if for any $\epsilon > 0$ there exists $N$ such that for $m, n \geq N$, there exist $\theta_m \in \F_m$, $\theta_n \in \F_n$  such that the following holds:
     \\
     (i) The associated measures $\mu_m^{(\theta_m)}$  and $\mu_m^{(\theta_m)}$ satisfy
     \[\int_0^1 |\mu_m^{(\theta_m)}(x) - \mu_n^{(\theta_n)}(x)|  < \epsilon.\]
(ii)  The  Laplace operators associated with $G_m$ and $G_n$ satisfy 
\begin{eqnarray*}
  \left|
\frac{\langle f, \Delta_m^{(\theta_m)} g \rangle_{\mu_m,\theta_m}}{\|f\|_{\mu_m,\theta_m}\|g\|_{\mu_m,\theta_m}}
-\frac{\langle f, \Delta_n^{(\theta_n)} g \rangle_{\mu_n,\theta_n}}{\|f\|_{\mu_n,\theta_n}\|g\|_{\mu_n,\theta_n}}
\right| < \epsilon
\end{eqnarray*}
for  integrable $f, g$ defined on $ [0,1] $ and we write
\begin{eqnarray}
\label{cauchy}
d(\Delta_{m},\Delta_{n})_{\mu_m, \mu_n} < \epsilon \end{eqnarray}
where $\mu_m, \mu_n$ denote the degree distributions of $G_m, G_n$, respectively.

\begin{remark}We note that the spectral distance here
is invariant subject to  any choices of measure preserving maps. In fact, for any  measure preserving map $\tau$, it follows from the definition and 
that
$
d(\Delta_m,\Delta_n)_{\mu_m, \mu_n} < \epsilon$ if and only if $
d(\Delta_m,\Delta_n)_{\mu_m \circ \tau, \mu_n \circ \tau} <  \epsilon
$.
\end{remark}

 Suppose the  sequence of graphs $G_n=(V_n,E_n)$ have degree distributions $\mu_n$ converging to $\mu$ as above.
Then (\ref{cauchy}) can be simplified.
The inequality in (\ref{cauchy}) can be replaced by an equivalent condition
\[ d(\Delta_{m},\Delta_{n})_{\mu}  < \epsilon \]
which can be described by
there exists $N$ such that for $m, n \geq N$, there exist $\theta_m \in F_m$, $\theta_n \in F_n$ such that
the  Laplace operators associated with $G_m$ and $G_n$ satisfy 
\begin{eqnarray}
  \left|
\langle f, (\Delta_m^{(\theta_m)}-\Delta_n^{(\eta_n)}) g \rangle_{\mu}
\right| < \epsilon
\end{eqnarray}
for  integrable $f, g : [0,1] \rightarrow {\mathbb R}$ with $ \|f\|_{\mu}= \|g\|_{\mu}=1$ .

 For an operator $M$ on $[0,1]$ we can define spectral $\mu$-norm, defined by
\[ \| M\|^2_\mu= \sup_{f,g} \mid \langle f, M g \rangle_\mu \mid \]
where $f,g : [0,1] \rightarrow {\mathbb R}$ range over integrable functions satisfy $\|f\|_\mu=\|g\|_\mu=1$.
We are ready to examine the convergence of a graph sequence under the spectral distance.
\begin{theorem}
\label{lapconv}
For a graph sequence $G_n=(V_n,E_n)$, where $n = 1, 2 , \ldots$, suppose the sequence of the Laplace operators $\Delta_n$ is Cauchy,
 then  for each $n$, there are  permutations $\theta_n \in F_n$ such that  the sequence of $\Delta_n^{(\theta_n)} $ converges to an exchangeable operator $\Delta$ and
the measure $\mu_n^{(\theta_n)}$ of $G_n$'s  converge to $\mu$
where $\Delta$ satisfies
\begin{eqnarray}
\label{eq10}
\int_0^1f (x)\Delta g(x) \mu(x) = \lim_{n \rightarrow \infty} \langle f ,\Delta_n^{(\theta_n)} g \rangle_{ \mu_n}
\end{eqnarray}
for any two integrable functions $f, g : [0,1] \rightarrow {\mathbb R}$.
\end{theorem}
\proof

For each positive integer $j$, we set  $\epsilon_j=2^{-j}$, and  let $ N(\epsilon_j)$ denote the least integer such that for $m,n \geq N(\epsilon_j)$, 
(\ref{cauchy}) holds for $\epsilon_j$.  To simplify the notation, we write $M(j)=N(\epsilon_j)$.

We first choose an arbitrary permutation $\eta_{(1)} \in F_{M(j)}$  and then   by induction define  permutations $\theta_{(j)} \in F_{M(j)}$'s, for $j>1$, using (\ref{cauchy})  so that 
\[ d(\Delta_{M(j)}, \Delta_{M(j+1)})_{\mu_{M(j)}^{(\theta_{(j)})}, \mu_{M(j+1)}^{(\theta_{(j+1)})}} < \epsilon_j.
\]
We can assume the associated measure for $\theta_{(j)}$ is  non-decreasing since we can simply adjust by choosing measure preserving maps.

For  each $n \in [M(j), M(j+1))$, we choose the  permutation $\theta_n$ such that
\[ d(\Delta_{n}, \Delta_{M(j)})_{\mu_{n}^{(\theta_{n})}, \mu_{M(j)}^{(\theta_{(j)})}} < \epsilon_j.
\]

We will use a  similar method as  in Lemma \ref{muconv} to prove the following:

\noindent
{\it Claim 1:} The sequence of  $\Delta_n^{(\theta_n)}$, for $n =1, 2, \ldots$  is Cauchy.

To prove the claim, we see that for any $m ,n  \geq M(j) $ satisfying $n \in [M(j), M(j+1)) $ and $m \in [M(k), M(k+1))$ with $j \leq k$, we have
\begin{eqnarray*}
&&
d(\Delta_m, \Delta_n)_{ \mu_{n}^{(\theta_{n})},\mu_{m}^{(\theta_{m})}}\\
&\leq& d(\Delta_{n}, \Delta_{M(j)})_{\mu_{n}^{(\theta_{n})}, \mu_{M(j)}^{(\theta_{(j)})}}+ \ldots \\
&&+d(\Delta_{M(k-1)}, \Delta_{M(k)})_{\mu_{M(k-1)}^{(\theta_{(k-1)})}, \mu_{M(k)}^{(\theta_{(k)})}} +d(\Delta_{M(k)}, \Delta_{m})_{\mu_{M(k)}^{(\theta_{(k)})},\mu_m^{(\theta_n)}}\\
&\leq& 2\epsilon_j + \epsilon_{j+1} + \ldots + \epsilon_{k-1} + 2 \epsilon_k\\
&=& 3 \epsilon_j
 \end{eqnarray*}
and  Claim 1 is proved.

\noindent
{\it Claim 2:}  The sequence of $\mu_n^{(\theta_n)}$ is Cauchy and therefore converges to a limit $\mu$.\\
To prove Claim 2, we will first  show that  for any $\epsilon > 0$,  $m,n \geq N(\epsilon_j)$,  and  any subset $S \subset [0,1]$, we have $|\mu_m^{(\theta_m)}(S)-\mu_n^{(\theta_n)}(S)| \leq 6 \epsilon_j$.

From the proof of Claim 1, we know that $d(\Delta_m^{(\theta_m)},\Delta_n^{(\theta_n)}) \leq 3 \epsilon_j$, which implies,
by choosing    $f=\chi_S$ and $g=\mathbf 1$ in  (\ref{eq9}) and (\ref{eq12}), 
\begin{eqnarray*}
3 \epsilon_j &\geq& d(\Delta_m^{(\theta_m)},\Delta_n^{(\theta_n)})\\
 &\geq& \left|  \sqrt{\mu_m^{(\theta_m)}(S)} -\sqrt{\mu_n^{(\theta_n)}(S)}
 \right|\\
 &\geq&\frac { \left|{\mu_m^{(\theta_m)}(S)} -{\mu_n^{(\theta_n)}(S)}\right|}{ \sqrt{\mu_m^{(\theta_m)}(S)} +\sqrt{\mu_n^{(\theta_n)}(S)}}
 \\
 &\geq& \frac 1 2 \left|{\mu_m^{(\theta_m)}(S)} -{\mu_n^{(\theta_n)}(S)}\right|.
 \end{eqnarray*}
To show that $\mu_n^{(\theta_n)}$ is Cauchy, we set $S=\{x:\mu_n^{(\theta_n)}(x) > \mu_m^{(\theta_m)}(x) \}$. Then,
\begin{eqnarray*}
\int_0^1 \mid \mu_n^{(\theta_n)}(x) - \mu_m^{(\theta_m)}(x) |&=& 2 \int_S\mid \mu_n^{(\theta_n)}(x) - \mu_m^{(\theta_m)}(x) \mid
+\int_{\bar{S}}\mid \mu_n^{(\theta_n)}(x) - \mu_m^{(\theta_m)}(x) \mid \\
&=& 2 \mid \mu_n^{(\theta_n)}(S) - \mu_m^{(\theta_m)}(S) \mid\\
&\leq& 12 \epsilon_j.
\end{eqnarray*}
 Claim 2 is proved.

Now, we can define the operator $\Delta$:
\begin{eqnarray}
\langle f, \Delta g \rangle &=&
\int_0^1f (x)\Delta g(x) \mu(x)\\
 &=& \lim_{n \rightarrow \infty} \langle f ,\Delta_n^{(\theta_n)} g \rangle_{ \mu_n}
\end{eqnarray}
for any two integrable functions $f, g : [0,1] \rightarrow {\mathbb R}$.

Combining Claims 1 and 2,  the sequence $\Delta_n^{(\theta_n)}$  converges to a limit $\Delta$.
\qed

For a graph sequence $G_n,$ where $ n=1,2, ...$, the  Laplace operator $\Delta_n$ of $G_n$ and $W_{G_n} \in {\mathcal W}^*$ are related as follows:
For functions $f,g: [0,1] \rightarrow {\mathbb R}$,  by using (\ref{eta}) we have
\ \begin{eqnarray*}
\langle f, (I-\Delta_n^{(\eta_n)}) g \rangle_{\mu_n,\eta}
&=& \int_0^1f(x) \big((I- \Delta_n^{(\eta_n)}) g(x)\big) \mu_n^{(\eta_n)} (x)\\
&=& \int_0^1\int_0^1\frac n {d_n(\eta_n(x))} f(x) g(y)W_n^{(\eta_n)}(x,y)dy \mu_n^{(\eta_n)} (x)\\
&=& \frac {n^2}{\vol(G_n)} \int_0^1\int_0^1 f(x) g(y) W_n^{(\eta_n)}(x,y) dxdy
\end{eqnarray*}
although the existence of the limit of $W_{G_n}$ is not necessarily required.

There are similarities  between  $\Delta$  and  the previous definitions for graph limits (as defined in \cite{lsz}) but
 the scaling is different as seen below:
\begin{eqnarray}
\label{eq11}
\int_0^1 f(x)\big((I-\Delta) g(x)\big) \mu(x) &=&\lim_{n \rightarrow \infty} \int_0^1 f(x)\big((I-\Delta_n^{(\eta_n)}) g\big)(x) \mu_n^{(\eta_n)}(x) \nonumber\\
&=& \lim_{n \rightarrow \infty}  \langle f, \frac{n^2}{\vol(G_n)}W_ng \rangle.
\end{eqnarray}
Suppose the graph sequence have  volume $\vol(G_n)$ converging to a  function $\Phi$. Then we have
\begin{eqnarray}
\label{eq111}
\int_0^1 f(x)\big((I-\Delta) g(x)\big) \mu(x)
&=& \lim_{n \rightarrow \infty}  \langle f, \frac{n^2}{\vol(G_n)}W_{G_n}g \rangle
\end{eqnarray}
Thus, the Laplace operator $\Delta$  as a limit of  $\Delta_n$   is essentially the identity operator minus
a scaled multiple of the  limit $W$.
We  state here  the following useful fact which follows from Theorem \ref{lapconv}:
\begin{lemma}
\label{ma1}
For a sequence of graphs $G_n, $ for $ n = 1, 2, \ldots,$ with  degree distributions $\mu_n$ converging to $\mu$,
 the associated Laplace operators $\Delta_n $ converges to $\Delta$ satisfying
\begin{eqnarray}
\langle \chi_S, (I-\Delta) {\mathbf 1} \rangle_\mu = \mu(S)\geq 0, \label{ss}\\
\text{and}~~~\langle \chi_S, (I-\Delta) \chi_T \rangle_\mu \geq 0\label{tt}
\end{eqnarray}
for any integrable subsets $S,T \subseteq [0,1]$ where ${\mathbf 1}$ is the constant function assuming the value $1$.
\end{lemma}
\proof The proof of (\ref{ss}) follows from the fact that
\begin{eqnarray*}
\langle \chi_S, (I-\Delta) {\mathbf 1} \rangle_\mu &=&\lim_{n \rightarrow \infty} \langle \chi_S^{(n)}, (I-\Delta_n) {\mathbf 1} \rangle_{\mu_n}\\
&=& \lim_{n \rightarrow \infty} \langle \chi_S^{(n)},  {\mathbf 1} \rangle_{\mu_n}\\
&=& \lim_{n \rightarrow \infty} \mu_n(S^{(n)})\\
&=& \mu(S).
\end{eqnarray*}
To see (\ref{tt}), we note that for any two vertices $u, v$ in $G_n$,  $\langle \chi_u , (I-\Delta_n)\chi_v \rangle_\mu = A_n(u,v)/\vol(G_n) \geq 0$.
\qed

\subsection{Defining the graphlets}
\label{lets}

Using the convergence definitions in the previous subsections,  we  define graphlets as the limit of a graph sequence
\begin{eqnarray}
\label{conv}
 G_1, G_2, \dots, G_n, \dots \rightarrow \mathcal{G}(\Omega, \Delta) \end{eqnarray}
which satisfies the following conditions:
 \vspace{-.1in}
\begin{enumerate}
\item  The degree distributions  of $G_n$ introduce measures $\mu_n$ on $\Omega$ and $\mu_n$ converges to a measure $\mu$ for
 $\Omega$ as in (\ref{deg}).
\item
 The discrete Laplace operators $\Delta_n$  for $G_n$ converges to $\Delta$ as an operator on $\Omega$ under the spectral distance 
 using the $\mu$-norm as in (\ref{muu}) and (\ref{eq10}). 
 \item
The  volume $\vol(G_n)$ of $G_n$ is increasing in $n$.
\end{enumerate}

  Several examples of graphlets will be given in the next section.

\begin{remark}
{\rm
One advantage of the graphlets  $\mathcal{G}(\Omega, \Delta)$  is the fact that the eigenvectors of graphs in the graph sequences can be approximated by eigenvectors of $\Delta$. In other words, eigenvectors of $\Delta$ can be used as universal basis for all graphs in graph sequences in  the graphlets $\mathcal{G}(\Omega, \Delta)$. 
}
\end{remark}
\begin{remark}
{\rm
In the other direction, graphs in graph sequences in  graphlets  $\mathcal{G}(\Omega, \Delta)$  can be viewed as a scaling for
discretization of $\Omega$ and $\Delta$. If two different graph sequences converge to the same graphlets, they can be viewed as giving different scaling for discretization.
}
\end{remark}
\begin{remark}
{\rm
Another way to describe a graphlets  is to view $\mathcal{G}(\Omega, \Delta)$ as the limit of graphlets $\mathcal{G}(\Omega_n, \Delta_n)$. Here  $\Omega_n$ can be described as a measure space under a measure
$\mu_n$ as follows.
The elements in $\Omega_n$, (the same as that of $\Omega$, labelled by $[0,1]$) is the union of $n$ parts, denoted by $I_v$, indexed by vertices  $v$ of
$G_n$.  
The degree of $v$ satisfies}
\vspace{-.4in}
\end{remark}
\begin{eqnarray*} d_n(v) &\approx &\vol(G_n) \int_{I_v} \mu(x).\end{eqnarray*}
The Laplace operator $\Delta_n$ can be defined by using the adjacency entry $A_n(u,v)=W_n(x,y)$  for $x \in I_u$ and $y \in I_v$.
Namely, $\Delta_n(x,y)= I_n(x,y)-W_n(x,y)/d_x$. 
 The graphlets $\mathcal{G}(\Omega, \Delta)$ as the limit of $\mathcal{G}(\Omega_n, \Delta_n)$ specifies the incidence quantitiy between any  two integrable  subsets  $S$ and $T$ in $\Omega$. 
 For an integrable  $S \subseteq \Omega$, we let  $\chi_S$ denote the characteristic function of $S$, which assume the value $1$ on $S$, and $0$ otherwise. In $\Omega_n$, the incidence quantity between $S$ and $T$, denoted by ${\mathcal E}_n(S,T)$
 satisfies:
\begin{eqnarray}
\label{st}
 {\mathcal E}_n(S,T) = \vol(G_n)\int_0^1 \chi_S(x)\big((I-\Delta_n)\chi_T(x)\big) \mu_n(x). \end{eqnarray}
In particular, for $S=T$,
\begin{eqnarray*}
E_n(S,S)&\approx& \vol(G_n)\big(\mu(S)- \mu(\partial(S)) \big)
\end{eqnarray*}
where the boundary $\partial(S)$ of $S$ satisfies
\[ E_n(S, \bar{S})\approx \vol(G_n)\mu(\partial(S))=\vol(G_n) \int_0^1 \chi_S(x) \Delta\chi_S(x)\mu(x). \]
\ignore{\subsection{Similarities and differences of graphlets and manifolds}

 When we examine graphlets, the examples from manifolds often come to mind.
 To illustrate the difference between graphlets and manifolds, here we construct a graphlets from a manifold $\Omega$
 with a measure $\nu$.  Basically we construct a sequence of graphs $G_n$'s which are discretizations
 of $\Omega$ as follows:
 \vspace{-.1in}
 \begin{itemize}
 \item[(i)]
 First we partition $\Omega$ into $n$ parts, each part is associated with a vertex of $G_n$. Here we assume that $n$ is quite large  and the partition is {\it smooth} subject to the conditions (ii) and (iii) below.

\item[(ii)]  Suppose a vertex $v$ is associated with a part $I_v \subset \Omega$. 
We assume that the boundary of each part 
 is measurable under $\nu$.
We can then define a {\it new} measure $\mu_n(v)$ by
$$\mu_n(v) = c \cdot \nu(\partial (I_v)) $$ where the constant $c$  is chosen so
that $\sum_v \mu_n(v)=1$  and  $\partial(S)$ denotes the boundary of a subset $S$. In $\Omega_n$, all $x \in I_u$ shares
the same value,
\[ \mu(x)= \frac{ \mu_n(v)}{\int_{I_v} \nu(x)} \]
so that
\[ \int_{I_v}\mu(x)=\mu_n(v). \]
If for all $v$, $ \int_{I_v} \nu(x)$  are all equal, we have $\mu(x)=n \mu_n(v)$.
 \item[(iii)]
 We assume that the shared boundary of any two parts is measurable under $\nu$.
 Furthermore, we assume that 
\[ \nu(\partial(I_v) = \sum_u \nu(\partial(I_u)\cap \partial(I_v)). \]
The adjacency matrix $A_n$ of $G_n$ is defined by 
$$A_n(u,v)=\vol(G_n) \cdot c \cdot \nu(\partial(I_u\cap \partial(I_v))$$ and $\vol(G_n)$ is
  nondecreasing  in $n$.
\item[(iv)] From $A_n$, we can then define the degree distribution $\mu_n$ and the degree $d_v$ of a vertex in $G_n$. We see that
the degree $d_v$ of a vertex of $G_n$ is just $\Phi_n \mu_n(v)$. We can define the discrete Laplace operator $\Delta_n$ for $G_n$
and the associated $\Omega_n$ as in (1) and (ii) above.
  \end{itemize}
 
 From $\Omega_n$ we can construct the graphlets as the limit. The graphlets $\Omega(\mu, \Delta)$ could be  quite different
 from the original manifold $([0,1],\nu)$. The measure $\mu$ is certainly not necessarily equal to the  measure $\nu$.
$\Delta$ can be quite different from the usual Laplace operator of a manifold. 
Roughly speaking, the graphlets $\Omega(\mu, \Delta,)$ represent the quantitative  incidence  among elements and subsets of $\Omega$. Namely, 
for any discretization of $\Omega$, the quantity of incidence between the parts can be represented as in Equation (\ref{st}).
For example, for the case of a graph sequence of $n$-paths, the graphlets turns out to be the same as a unit interval since we have $\mu=\nu$ and $\Delta$
is  the same as   the continuous Laplace operator while we choose $\vol(G_n)=2n$.  However,  we also see  examples of dense graph sequences,
as illustrated in Section \ref{dense}, which  do not
converge to  manifolds in general.

 }
 
\ignore{
\section{The spectral distance and the discrepancy distance}

\label{2norms}

\subsection{The cut distance and the discrepancy distance}
In previous studies of graph limits,  a so-called cut metric that is often used  for which the distance of two graphs $G$ and $H$ which share the same set of vertices $V$ is measured by the following (see \cite{ borgs, fk}).
\begin{eqnarray}
\label{cut}
\cut(G , H)= \frac 1 {|V|^2}
\sup_{S, T \subseteq V} \left| E_{G}(S,T) - E_{G'}(S,T) \right|
\end{eqnarray}
where $E_G(S,T)$ denotes the number of ordered pairs $(u,v)$ where $u $ is in $ S$, $v$ is in $ T$ and $\{u,v\}$ is an edge in $G$.

We will   define a  discrepancy distance  which is similar to but different from the above cut distance.  For two graphs $G$ and $H$ on the same vertex
set $V$,
the discrepancy distance, denoted by $\disc(G,H)$ is defined
as follows:
\begin{eqnarray}
\label{disc}
\disc(G, H)
=\sup_{S, T \subseteq V}\left| \frac{E_{G}(S,T)}{\sqrt{\vol_G(S)\vol_G(T)}}-\frac{E_{H}(S,T)}{\sqrt{\vol_H(S)\vol_H(T)}} \right|. \end{eqnarray}
We remark that the only difference between the  the cut distance and the discrepancy distance is in the normalizing factor which will be useful
in the proof later.

For two graphs $G_m$ and $G_n$ with $m$ and $n$ vertices respectively, we use the labeling maps $\theta$ and $\eta$ to map the vertices of $G_m$ and $G_n$
to $[0,1]$, respectively.  We define the measures $\mu_m$ and $\mu_n$ on $[0,1]$ using the degree sequences of $G_m$ and $G_n$ repectively, as in
Section \ref{subdeg}.   From the definitions and substitutions, we can write:
\begin{eqnarray}
\label{enst}
 E_{G_n}(S,T)=\vol(G_n) \langle \chi_S, (I-\Delta_n) \chi_T \rangle_{\mu_n,\theta}.
 \end{eqnarray}
Therefore the  discrepancy distance in (\ref{disc}) can be written in the following general format:
\begin{align}
&\disc(G_m,G_n)\nonumber\\
&=\inf_{\theta \in {F}_m, \eta \in F_n} \sup_{S, T \subseteq [0,1]} \left| \frac{\langle \chi_{S}, (I-\Delta_m) \chi_{T} \rangle_{\mu_m,\theta}} 
{\sqrt{\mu_m(S)
 \mu_m(T)}
 }
-\frac{\langle \chi_{S}, (I-\Delta_n) \chi_{T} \rangle_{\mu_n,\eta}}{\sqrt{\mu_n(S)
 \mu_n(T)}
 } \right| \nonumber\\
\label{gm}
\end{align}
where $S, T$ range over all integrable subsets of $[0,1]$.
We can rewrite (\ref{enst}) as follows.
\begin{eqnarray}
\label{aa}
E_{G_n}(S,T)&=&\vol(G_n) \int_{x \in \Omega} \chi_S(x) \big((I-\Delta_n) \chi_T\big)(x) \mu_n(x).
\end{eqnarray}
Alternatively,  $E_{G_n}(S,T)$ was  previously expressed   (see   \cite{lsz}) as follows:
\begin{eqnarray}
\label{previous}
E_{G_n}(S,T)&=& n^2 \int_{x \in S}\int_{y \in T} { W}(x,y)~ ds~ dt
\end{eqnarray}

The two formulations (\ref{aa}) and (\ref{previous})  look quite different but are of the same form when the  graphs involved are regular.
However,  the format in (\ref{previous}) seems hard to extend to   general graph sequences with smaller edge density.

Although the above definition in (\ref{gm}) seems complicated, it can be simplified when the degree sequences converge. Then,  $\mu_m$ and $\mu_n$ are to be approximated
by the measure $\mu$ of the graph limit. In such cases, we define
\begin{align*}
\disc_\mu(G_m,G_n)
 &=\inf_{\theta \in {F}_m, \eta \in F_n} \sup_{S, T \subseteq [0,1]} \left| \frac{\langle \chi_{S}, \Delta_m \chi_{T} \rangle_{\mu,\theta}} 
{\sqrt{\mu(S)
 \mu(T)}
 }
-\frac{\langle \chi_{S}, \Delta_n \chi_{T} \rangle_{\mu,\eta}}{\sqrt{\mu(S)
 \mu(T)}
 } \right|\\
 &=\sup_{S, T \subseteq [0,1]} \frac 1 {\sqrt{\mu(S)\mu(T)}} \left|\langle \chi_{S}, (\Delta_m-\Delta_n) \chi_{T} \rangle_{\mu} \right|.
\end{align*}
where $S, T$ range over all integrable subsets of $[0,1]$ and we suppress the labelings $\theta,\eta$ which achieve the infininum.

We will show  that the convergence using the spectral distance defined under the $\mu$-norm  is equivalent to the convergence using  the discrepancy distance   in Section \ref{2norms}.  

\subsection{The equivalence of convergence using spectral distance and the discrepancy distance}

We will prove the following theorem which holds without any density restriction on the graph sequence. The proof applies  similar techniques in \cite{BL, BN, butler} to graphs of general
degree distributions.
\begin{theorem}
\label{2norm}
Suppose  the degree distributions $\mu_n$,  of a graph sequence $G_n,$ for $ n = 1, 2, \ldots$, converges to $\mu$.
The following statements are equivalent:

\noindent
(1) $G_n,$ for $  n = 1, 2, \ldots $, converges under the spectral distance.

\noindent
(2) $G_n, $ for $  n = 1, 2, \ldots $, converges under the $\disc$-distance.
\end{theorem}
\proof
Suppose that for a given $\epsilon > 0$, there exists an $N> 1/\epsilon$ such that for $n > N$, we have
\begin{eqnarray*}
 \| \mu_n-\mu \|_1 < \epsilon. \end{eqnarray*}
The proof for $ (1)$  
 $\Rightarrow$ 
 $(2)$ is rather straightforward and can be shown as follows:

Suppose (1) holds and we have, for $m, n > N$, $\|\Delta_m - \Delta_n\|_\mu < \epsilon$. Then,
\begin{align*}
\disc(G_m,G_n)
&=\sup_{S, T \subseteq [0,1]} \left| \frac{\langle \chi_{S}, (I-\Delta_m) \chi_{T} \rangle_{\mu_m}} 
{\sqrt{\mu_m(S)
 \mu_m(T)}
 }
-\frac{\langle \chi_{S}, (I-\Delta_n) \chi_{T} \rangle_{\mu_n}}{\sqrt{\mu_n(S)
 \mu_n(T)}
 } \right|\\
 &\leq
\sup_{S, T \subseteq [0,1]}\frac{1}{\sqrt{\mu(S)\mu(T)}} \left| \langle \chi_{S}, (\Delta_m-\Delta_n) \chi_{T} \rangle_{\mu}
 \right| + 2 \epsilon
\\
&=\sup_{S, T \subseteq [0,1]}\frac{1}{\|\chi_S\|_\mu\|\chi_T\|_\mu} \left| \langle \chi_{S}, (\Delta_m-\Delta_n) \chi_{T} \rangle_{\mu}\right|\\
&\leq \|\Delta_m-\Delta_n\|_\mu + 2 \epsilon\\
&\leq 3\epsilon.
\end{align*}

To prove (2) $\Rightarrow$ (1), if suffices to prove the Lemma \ref{ma22}  by substituting $A=\Delta_m$ and $B = \Delta_n$. 
The proof of Lemma \ref{ma22} will be given below.
\qed

\begin{lemma}
\label{ma22}

  Assume  an operator $M: [0,1] \times[0,1] \rightarrow \mathbb R$
  is the difference of two positive operators,
  $M=A-B$ where $A(x,y) \geq 0, B(x,y) \geq 0$ for all $x,y \in [0,1]$. Suppose   $A{\mathbf 1} = B{\mathbf 1} = \mathbf 1$ and suppose $M$ satisfies
  \begin{eqnarray}
  \label{assume}
\left|  \langle \chi_S, M\chi_T \rangle_\mu  \right| \leq
\gamma \sqrt{\mu(S) \mu(T)} 
\end{eqnarray}
for some  $\gamma >0$ for  any two integrable subsets  $S, T \subseteq [0,1]$. Then for any two integrable functions $f,g: [0,1] \rightarrow {\mathbb R}$, we have
\[ 
\langle f,  M g \rangle_\mu \leq 20 \gamma  \log (1/\gamma ) \| f\|_\mu \|g \|_\mu
\]
provided $\gamma < .02$.
\end{lemma}

To prove this, we  use ideas in  \cite{BL, BN, butler}. 
First we prove the following  fact:

\noindent
{\it Fact 1:}
For  an integrable function $f$ defined on ${[0,1]}$ with $\|f\|_\mu=1$, for any $\epsilon > 0$, there exists an $N(\epsilon)$ such that for any $n > N(\epsilon)$ there is a function $g$ defined on $[0,1]$ satisfying :\\
(1) $ \| g\|_\mu \leq 1$,\\
(2) $\|f-g\|_\mu \leq 1/4+\epsilon$,\\
(3) The  value $g(y)$ in the interval $( (j-1)/n,j/n]$ is a constant $g_j$ and $g_j$ is 
 of the form $(\frac 4 5)^j$ for integers $j$.\\
{\it Proof of Fact 1:}
Since $f$ is integrable, for a given $\epsilon$, we can approximate $\|f\|^2_\mu$ by
 a function $\bar f$, with $\bar f (x)=f_j$ in $((j-1)/n,j/n]$, such that 
 \[ \left| \int_0^1 f^2(x) \mu(x) -\sum_{j=1}^n \int_{(j-1)/n}^{j/n} f_j^2 \mu(x) \right| < \epsilon. \]
For   $\bar f=(f_j)_{1 \leq j \leq n}$, we define  $g=(g_j)_{1 \leq j \leq n} $ as follows. If $f_j=0$, we set $g_j=0$. Suppose $f_j \not = 0$, there is a unique integer $k$ so that $(4/5)^{k}< f_j \leq (4/5)^{k-1}$. We set $g_j=(\frac 4 5)^k $. Then 
\[ 0 < |f_j|-|g_j| \leq  (\frac 4 5)^{k-1} -(\frac 4 5)^k = \frac 1 4 (\frac 4 5 )^k < \frac 1 4 |f_j|,\]
which implies 
$\|f-g\|_\mu^2 \leq \epsilon + \sum_j \int_0^1|f_j-g_j|^2 {\mu}(x) \leq  \epsilon +\frac 1 {16} \sum_t |f_j|^2 \mu(x) = \frac 1 {16}+\epsilon$.
Fact 1 is proved.

Suppose there are  functions $f',g'$ satifying 
$
\|M\|_\mu$ $=|\langle f',Mg'\rangle_\mu|$ and $\|f'\|_\mu=\|g'\|_\mu=1$. If $f,g$ are functions such that $\|f\|_\mu, \|g\|_\mu \leq 1$ and $\|f'-g\|_\mu,\|f'-g\|_\mu \leq 1/4+\epsilon$, then 
\begin{eqnarray*}
\|M\|_\mu &=&|\langle f',Mg'\rangle_\mu|\\
&\leq&|\langle f,Mg\rangle_\mu|
+|\langle f'-f,My\rangle_\mu+\langle f',M(g'-g)\rangle_\mu|\\
&\leq& |\langle f,Mg\rangle_\mu|+\big(\frac 2 4+2 \epsilon\big) \|M\|_\mu.
\end{eqnarray*}
Thus, we have
\begin{eqnarray}
\label{eq3f}
\|M\|_\mu \leq (2+4 \epsilon) |\langle f, Mg \rangle_\mu|.
\end{eqnarray}

By applying Fact 1, we can choose $f,g$ of the following form: Namely, $
f =\sum_{t} (\frac 4 5)^t f^{(t)}$, where the  $f^{(t)}$ denotes the indicator function of  $x^{(t)}=\{x : \bar f(x)= (\frac 4 5 )^t\}$.   Similarly we write
 $g = \sum_t (\frac 4 5)^t g^{(t)}$, where the  $g^{(t)}$ denotes the indicator function of  $y^{(t)}=\{y : \bar g(y)= (\frac 4 5 )^t\}$. 
Now  we choose $\kappa= \log_{4/5} \gamma$ and we consider
\begin{align*}
\left| \langle f, M g \rangle_\mu
\right| &\leq \sum_{s,t} (\frac 4 5 )^{s+t} \left|\langle f^{(s)}, Mg^{(t)} \rangle_\mu \right|\\
&\leq \sum_{|s-t|\leq \kappa} (\frac 4 5 )^{s+t} \left|\langle f^{(s)}, Mg^{(t)} \rangle_\mu \right| \\
&~~~+ \sum_{s} (\frac 4 5 )^{2s+\kappa}\sum_t \left|\langle f^{(s)}, Mg^{(t)} \rangle_\mu \right| \\
&~~~+ \sum_{t} (\frac 4 5 )^{2t+\kappa}\sum_s \left|\langle f^{(s)}, Mg^{(t)} \rangle_\mu \right| \\
&= X + Y + Z.
\end{align*}
We now bound the three terms separately.  For a function $f$, we denote  $\mu(f)=\mu(\supp(f))$ to be the measure
 of the support of $f$.
 Using the assumption (\ref{assume}) for $(0,1)$-vectors and the fact that $f^{(s)}$'s are orthogonal (as well as the $g^{(t)}$'s), we have
\begin{align*}
X&=\sum_{|s-t|\leq \kappa} (\frac 4 5 )^{s+t} \left|\langle f^{(s)},Mg^{(t)} \rangle_\mu \right|\\
&\leq  \gamma  \sum_{|s-t|\leq \kappa} (\frac 4 5 )^{s+t}\sqrt{\mu(f^{(s)})\mu( g^{(t)})}\\
&\leq \frac{ \gamma } 2 \sum_{|s-t|\leq \kappa}\big( (\frac 4 5 )^{2s}\mu(f^{(s)})+(\frac 4 5 )^{2t}\mu( g^{(t)})\big)\\
&\leq \frac{ \gamma (2 \kappa+1)} 2 \big(\sum_{s} (\frac 4 5 )^{2s}\mu(f^{(s)})+\sum_{t} (\frac 4 5 )^{2t}\mu( g^{(t)})\big)\\
&\leq  \gamma (2 \kappa+1),
\end{align*}
since each term can appear at most $2 \kappa +1$ times.
For the second term we have, by using Lemmas \ref{ma1} and \ref{ma22}, the following:
\begin{align*}
Y&\leq \sum_{s} (\frac 4 5 )^{2s+\kappa}\sum_t \left|\langle f^{(s)}, Mg^{(t)} \rangle_\mu \right|\\
&\leq (\frac 4 5 )^{\kappa} \sum_{s} (\frac 4 5 )^{2s} \langle f^{(s)}, |(A-B)\sum_t g^{(t)}| \rangle_\mu \\
&\leq(\frac 4 5 )^{\kappa} \sum_{s} (\frac 4 5 )^{2s} \langle f^{(s)}, (A+B){\mathbf 1} \rangle_\mu \\
&\leq 2(\frac 4 5 )^{\kappa} \sum_{s} (\frac 4 5 )^{2s} \langle f^{(s)}, {\mathbf 1} \rangle_\mu \\
&\leq 2 (\frac 4 5 )^{\kappa}
\sum_s\mu(f^{(s)})\\
&\leq 2 (\frac 4 5 )^{\kappa}
\end{align*}
The third term can be bounded in a similar way.  Together, we have
\begin{align*}
\|M\|_\mu & \leq (2 +4\epsilon)\big( \gamma( 2 \kappa +1)+ 4 (\frac 4 5)^{\kappa})\\
& \leq (2+4\epsilon) \big( \gamma( 2\frac{ \log(1/\gamma)}{\log 5/4}+1)+ 4 \gamma \big)\\
&\leq \frac{4+8\epsilon }{\log(5/4)} \gamma \log(1/\gamma) + 8\gamma
\\
&\leq 20 \gamma \log (1/\gamma)
\end{align*}
as $n$ goes to infinity since $\frac {4} {\log 5/4} \approx 17.93$ and $\gamma < .02$.
This completes the proof of the theorem. 
\qed
}
\ignore{
\section{The geometry of graphlets}
\label{geo}
 First we need some definitions for related concepts
on a graph before we proceed to consider  the graph limit $\Omega$.  We will define the gradient operator on a graph and 
also define a metric by  using the Laplace operator
$\Delta$ and  the discrete heat kernel. 
\subsection{The gradient operator on a graph $G$}
Let $G$ denote an undirected weighted graph $G=(V, E)$ with the weighted adjacency matrix $A$.
We choose an arbitrary fixed   orientation of edges in $G$, denoted by $\hat{G}=(V, \hat{E})$ the gradient operator $\nabla$ is a $|E| \times |V|$ matrix  defined as follows: For $ u \in V$ and $e \in \hat{E}$, 
\[ \nabla  (e,u)= \begin{cases}
1 & \text{if $e=(u,w) \in \hat{E} $ for  some $w \in V$}, \\
-1 & \text{if $e=(w',u) \in \hat{E} $ for  some $w' \in V$}, \\
0& \text{otherwise}.
\end{cases}. \]
For any $f : V \rightarrow {\mathbb R}$ and an oriented edge $e=(u,v)$, we have
\[ \nabla f (e)= 
f(u)-f(v). \]

For  the   measure $\mu$ on $V $  with $\mu(u)= \frac{d_u}{\vol(G)}$, we define an extension of $\mu$ on $E$ such that  
for  an edge $\{u,v\} \in E$, we denote $\mu_e = \mu(u,v) = \frac{A(u,v)}{\vol(G)}$, satisfying
\begin{eqnarray*}
\langle \nabla f, \nabla f \rangle_\mu&=& \sum_{e} (\nabla f ~(e))^2 \mu(e)\\
&=& \langle f , \Delta f \rangle_\mu.
\end{eqnarray*}
This is the analog of the continuous version:
\begin{eqnarray}
\int_\Omega \| \nabla f \|^2 = \int_\Omega f \Delta f.
\end{eqnarray}
We remark that both the discrete and the continuous versions have their advantages. The discrete version is explicit but somewhat messy. The continuous is 
succinct but can be ambiguous with possible hidden dangers such as   taking limits, interchanging double integrals and dealing with infinity.

For a vertex $u$, the gradient  operator $\nabla$ can be viewed as a {\it frame} which consists of vectors from $u$ to each of the   neighbors of $u$.  A path $P$ 
that joins  $u$ to any 
other  node of the graph must travel through one
of its neighbors, say $v$, and the length of $P$ is just $|\nabla f(u,v)|$ plus the length of the rest
of the path. Although the term``frame" has been used for describing the  gradient fields  on manifolds,  the frame here is different in the sense that our frame does not span a vector field, for example. 
\subsection{The gradient operator for  graph limits}
Now suppose a graph sequence $G_n$ converges to  graphlets  $\Omega=\Omega(\mu,\Delta)$.
We will define the gradient $\nabla$ at an element $x$ of $\Omega$. 
By analogy with manifolds, the gradient indicates the rate of change when taking a `small' step along each direction.
For the graph limit $\Omega$,   the scaling factor (as the choice of  some   `small' step) is  obtained  by selecting $n$.  If $n$ is sufficiently
large, $G_n$ approaches $\Omega$.  After we pick (the scaling factor) $n$, we can then use the frames at vertices of $G_n$. 
Thus, the frame at a point $x$ in $\Omega$ is just   a sequence of frames at vertices $u_n$  each from  graph $G_n$ in  the graph sequence and $x$ is the limit of $u_n$.
 Let 
 $\eta_n$ denote the mapping from $V_n$ to  $\{1, 2, \ldots, n\}$.
 For $f : \Omega \rightarrow {\mathbb R}$,  the gradient at scale $n$, denoted by $\nabla_n$, satisfies
\begin{eqnarray*}
| \nabla_n f(x,y)|=\begin{cases}|f(x)-f(y)| &\text{if $A_n(x,y) \not = 0$}\\
0 &\text{otherwise,}
\end{cases}
  \end{eqnarray*}
where $A_n(x,y)= A_n(\eta_n(x), \eta_n(y))$.
  An illuminating example  is the graph sequence of   $n$-cycles $C_n$, which  converge to
$\Omega$, a circle of circumference $1$.  If we consider the graph sequence of cartesian products $C_n \times C_n$, the
gradient at each point at $\Omega$ has a frame consisting of four `directions'. 

If the degrees of $G_n$ are bounded above by some constant independent of $n$, then the frame at each point in $\Omega$ also has 
a bounded number of `directions'.  
If the degrees of $G_n$ are not bounded above by some constant, the frame can still be defined.
For a vertex $u_n$ in $G_n$,  let $N_n(u_n)$ denote the set of neighbors of $u$. For each point $x$ in $\Omega$, the frame
at $x$ is represented by the sequence  $\eta_n(N_n(u_n))$.  For example, for a dense graph sequence the graph limit $\Omega$ is finitely generated and
therefore 
the frame at $x $ in $\Omega$  consists of all the `directions' from $x$ to  some intervals of length  bounded below by a constant independent of $n$ as seen in Section  \ref{dense}.

In order to fully describe the geometry of the gradient of a general graph, we will need to consider graph embeddings into Hilbert spaces, in particular, the embeddings 
by using the heat kernels which we will discuss in the next  section.  With some appropriate universal embeddings $\psi$ and additional smoothness conditions, the `directions'  of a frame at $x$ as mentioned above can be represented as 
the limits of  sequences of vectors $\big(\psi(u_n)-\psi(v_n)\big) /\|\psi(u_n)-\psi(v_n)\|$  where the sequence $u_n$ converges to $x$ and $v_n$ are neighbors of $u_n$ in $G_n$.

Frames play crucial roles in the geometry of the graphlets. In addition to finding geodesic paths, the frames are pivotal in connection with
the   component analysis and dimension reduction
methods in data analysis, for example.

  \subsection{The heat kernel of a graph}
  The heat kernel  $H_t$ of
a graph $G$ is the fundamental solution for the heat equation for $t \geq 0$ (see \cite{cy}):
\begin{eqnarray}
\label{heat} \frac{\partial}{\partial t} H_t = - \Delta H_t. \end{eqnarray}
Consequently, $H_t$ can be written as:
\begin{eqnarray*}
H_t &=& e^{-t \Delta}
\\
&=& I - t \Delta + \frac {t^2} 2 \Delta^2 + \ldots + (-1)^k \frac{t^k}{k!} \Delta^k + \ldots\\
&=& \sum_{j=0}^{n-1} e^{-t \lambda_j} P_j
\end{eqnarray*}
where $0=\lambda_0\leq \lambda_1 \leq \ldots \leq \lambda_{n-1}$ are eigenvalues of ${\mathcal L}=D^{1/2} \Delta D^{-1/2}$,
$\phi_j$'s denote the associated eigenvectors of  $\mathcal L$ forming  an orthogonal basis, and $P_j$ are the $j$th projection. If we treat
$\phi_j$ as a row vector,
then we can write $P_j=  D^{-1/2} \phi_j^*\cdot  \phi_j D^{1/2}$ where $f^*$ denotes the transpose of $f$. In particular $P_0 = {\mathbf 1}^* \cdot \pi$.
It can be easily checked (see \cite{cy}) that 
\begin{eqnarray}
\label{h1} H_t(u,v) \geq 0~~~\text{ and }\sum_v H_t(u,v)=1. 
\end{eqnarray}
\ignore{
For a graph limit $\Omega$ and $t \geq 0$, we can define the $t$-distance to be
\begin{eqnarray*}
(\dist_t(x,y))^2&= & -t \log H_t(u,v).
\end{eqnarray*}
The disadvantage of this definition is its dependency on $t$. However, as seen in the example of graphs limits of $n$-cycles, the distances
in each graph of a graph sequence depends on the number of vertices. The relations among the distances in each graphs and the distance in the graph limits can be quite different for different examples of graph limits.
The advantage of this
 definition is the fact that this is an adaption  of the formula for evaluating the heat kernel of a  finite dimensional Riemannian manifolds with upper bounded diameter
and lower bounded Ricci curvature. This allows some of the methods and concepts in differential geometry to be used for graphs (see \cite{cgy}).
A general graph can be quite different from a discretization of such Riemannian manifolds.
}
For a given vertex $u$, we consider the function  $H_{t,u}$ defined by $H_{t,u}(v)=H_t(u,v)$ for all vertices $v$. 
Here are some useful facts concerning $H_{t,u}$.
\begin{lemma}
\label{m1} For  vertices $u,v$ in $G$, we have\\
\begin{align*}
(i)&~~~ H_{t,u}(v) =d_u \sum_{x } \frac{ \hkr_{t/2, u}(x)}{d_x} \sum_{y} \frac{ \hkr_{t/2, v}(y) }{d_y}, \\
(ii)&~~~\frac{\partial}{\partial t} \hkr_{t, u} (v) = -d_v \sum_{x
\sim y} \bigg(\frac{\hkr_{t/2, u}(x)}{d_x} - \frac{\hkr_{t/2,
u}(y)}{d_y}\bigg)\bigg(\frac{\hkr_{t/2, v}(x)}{d_x} - \frac{\hkr_{t/2,
v}(y)}{d_y}\bigg) 
\\
& \text{where the sum is over all unordered pairs of vertices $\{x,y\}$ as edges in
$G$.  In particular,}\\
&~~~~\frac{\partial}{\partial t} \hkr_{t, u} (u) = -d_u \sum_{x
\sim y} \bigg(\frac{\hkr_{t/2, u}(x)}{d_x} - \frac{\hkr_{t/2,
u}(y)}{d_y}\bigg)^2 \\
(iii)&~~~ H_{t,u}(u)-\pi(u) \geq \sum_{x } \frac{ \big(\hkr_{t/2, u}(x)-\pi(x) \big)^2}{d_x} \geq 0
\end{align*}
\end{lemma}
\proof 
Let $\chi_u$ denote the characteristic function of $u$.  For (i), we see that
\begin{eqnarray*}
H_{t,u}(v) &=& \chi_u H_{t/2} D^{-1} H^*_{t/2} D \chi_v^*\\
&=&
d_v \sum_{x } \frac{\hkr_{t/2, u}(x) }{d_x} \sum_{y } \frac{\hkr_{t/2, v}(y) }{d_y} 
\end{eqnarray*}

To prove (ii), we use the heat equation  (\ref{heat}) and we have
\begin{eqnarray*}
\frac{\partial}{\partial t} \hkr_{t,u}(v) &=&\chi_u
\frac{\partial}{\partial t} H_t \chi_v^*
\\
&=& -\chi_u(I-D^{-1}A)H_t \chi_v^*\\
&=& -\chi_u H_{t/2}(I-D^{-1}A)D^{-1}H^*_{t/2} D \chi_v^*\\
&=& -\chi_u H_{t/2} D^{-1} (D-A) D^{-1} H^*_{t/2}\chi_v^* d_v\\
&=&-d_v \sum_{x \sim y}  \bigg(\frac{\hkr_{t/2, u}(x)}{d_x} - \frac{\hkr_{t/2,
u}(y)}{d_y}\bigg)\bigg(\frac{\hkr_{t/2, v}(x)}{d_x} - \frac{\hkr_{t/2,
v}(y)}{d_y}\bigg) .
\end{eqnarray*}
Here we use the fact  (see \cite{ch0}) that for any $f, g ~:~ V
\rightarrow {\mathbb R}$, we have
\[ f(D-A)g^* = \sum_{x \sim y} (f(x)-f(y))(g(x)-g(y)). \]

For (iii), we have
\begin{eqnarray*}
 \hkr_{t,u}(u) -\pi(u) &=&\chi_u
 (H_t -P_0)\chi_u^*
\\
&=& \chi_u (H_{t/2}-P_0)D^{-1}(H^*_{t/2}-P^*_0) D \chi_i^*\\
&=& \chi_u H_{t/2}  D^{-1} H^*_{t/2}\chi_u^* d_u\\
&=&d_u \sum_{x } \frac{ \big(\hkr_{t/2, u}(x)-\pi(x) \big)^2}{d_x}\\
&\geq& 0
\end{eqnarray*}
\qed

For any function $f$ defined on $V$, we write  $H_{t,f}=\sum_{v \in V} f(v) H_{t,v}$.
 For a subset $S$ in $G$  we define a function $f_S$ by
\begin{eqnarray*}
 f_S(u)&=& \begin{cases} \frac{d_u}{\vol(S)}&  \text{if $u \in S$,}\\
0 & \text{otherwise.} \end{cases}
\end{eqnarray*} 
 For a subset $S$ of vertices in $G$, the {\it edge boundary} of $S$, denoted by $\partial S$
 is defined by
 \[ \partial S = \{ \{u,v\} \in E~:~ u \in S~ \mbox{and} ~v \not \in
 S \}. \]
 Let $\bar{S}=V \setminus S$ denote the complement of $S$. Clearly,
 $\partial S = \partial \bar{S}$. The {\it Cheeger ratio } of $S$,
 denoted by $h_S$, is defined by
 \begin{eqnarray}
 \label{chee}
  h(S) = \frac{|\partial S|}{\min \{ \vol(S), \vol(\bar{S})\}} 
  \end{eqnarray}
 and the {\it Cheeger constant} of a graph $G$ is defined by
 $ h_G = \min_{S \subseteq V} h(S).$

\begin{lemma}
For a subset $S$ in $G$, we have
\label{m2}
\begin{align*}
(i)&~~~\frac{\partial}{\partial t} \hkr_{t, f_S} (S) = -\vol(S) \sum_{x
\sim y} \bigg(\frac{\hkr_{t/2, f_S}(x)}{d_x} - \frac{\hkr_{t/2,
{f_S}}(y)}{d_y}\bigg)^2\\
& \text{where the sum is over all unordered pairs of vertices $\{x,y\}$ as edges in
$G$.}\\
(ii)&~~~ \frac{\partial}{\partial t} \hkr_{t, f_S} (S)  \geq -h_S.
\end{align*}
\end{lemma}
\proof 
We use the heat equation and Lemma \ref{m1} (ii)  as follows:
\begin{align*}
\frac{\partial}{\partial t} \hkr_{t, f_S} (S) &= \frac{\partial}{\partial t} \sum_{u \in S}f_S(u) \hkr_{t, u} (S) = \frac{\partial}{\partial t} \sum_{u,v \in S}f_S(u) \hkr_{t, u} (v) \\
&=-
 \sum_{x
\sim y}\sum_{u \in S}  f_S(u) \bigg(\frac{\hkr_{t/2, u}(x)}{d_x} - \frac{\hkr_{t/2,
u}(y)}{d_y}\bigg)\bigg(\sum_{v\in S} d_v\big(\frac{\hkr_{t/2, v}(x)}{d_x} - \frac{\hkr_{t/2,
v}(y)}{d_y}\big)\bigg) \\
&=-
 \sum_{x
\sim y}\bigg(\frac{\hkr_{t/2, f_S}(x)}{d_x} - \frac{\hkr_{t/2,
f_S}(y)}{d_y}\bigg)\bigg(\sum_{v\in S} d_v\big(\frac{\hkr_{t/2, v}(x)}{d_x} - \frac{\hkr_{t/2,
v}(y)}{d_y}\big)\bigg) \\
&=-
\vol(S)\sum_{x
\sim y}\bigg(\frac{\hkr_{t/2, f_S}(x)}{d_x} - \frac{\hkr_{t/2,
f_S}(y)}{d_y}\bigg)\bigg(\sum_{v\in S} f_S(v)\big(\frac{\hkr_{t/2, v}(x)}{d_x} - \frac{\hkr_{t/2,
v}(y)}{d_y}\big)\bigg) \\
&=-
\vol(S)\sum_{x
\sim y}\bigg(\frac{\hkr_{t/2, f_S}(x)}{d_x} - \frac{\hkr_{t/2,
f_S}(y)}{d_y}\bigg)^2
\end{align*}
To prove ({\it ii}),  we checked that
$\frac{\partial^2}{\partial t^2} \hkr_{t, f_S}(S) \geq 0$, which implies
\begin{eqnarray*}
\frac{\partial}{\partial t} \hkr_{t, f_S}(S)& \geq&
\frac{\partial}{\partial t} \hkr_{t, f_S}(S) ~{\bigg |}_{t=0}\\
&=& -\frac{|\partial S|}{\vol(S)} = -h_S
\end{eqnarray*}
\qed
\begin{theorem}
\label{t3}
For  a subset $S$ in $G$ with $h(S)=h$, we have
 \begin{eqnarray}
 \label{46}
 H_{t,f_S} (S) \geq e^{-th}.\end{eqnarray}
 \end{theorem}

\proof
We consider $F(t)=  -\log H_{t,f_S}(S)$. From (i) in Lemma \ref{m2}, we have
\begin{eqnarray*}
\frac{\partial}{\partial t} F(t) &=& \frac{-
\frac{\partial}{\partial t} \hkr_{t, f_S}(S)
}{H_{t,f_S}(S)}
\leq h.
\end{eqnarray*}
This implies
\[ \log H_{t, f_S}(S) \geq- \int_0^t F(s)ds = -F(t) \geq
- ht,\]
using the fact that $H_{0,f_S}(S)=1$.
Therefore, we have
\[ H_{t,f_S}(S) \geq  e^{-ht}.\]
We have proved (\ref{46}).
\qed

 \begin{theorem}
 \label{t33}
 For  a subset $S$ in $G$ with $h(S)=h$,
and for any given $\epsilon, \delta > 0$,  there is a subset $R$ with $\vol(R) \geq (1-\delta) \vol(S)$ such that
for any $u$ in $R$, we have
\begin{eqnarray}
\label{m333xx}H_{t,u}(S) \geq 1-\epsilon \end{eqnarray}
provided $t \leq \epsilon \delta/h$.

In particular,  there is a subset $T$ with $\vol(T) \geq  \vol(S)/2$ such that
for any $u$ in $T$, we have
\begin{eqnarray}
\label{m33}H_{t,u}(S) \geq e^{-2th}.
 \end{eqnarray}

\end{theorem}

\proof
We consider $R$ consisting of all $u$ satisfying (\ref{m33}).  
Suppose $\vol(R) \leq (1-\delta) \vol(S)$. From Theorem \ref{t3}, we have 
\begin{eqnarray*}
1-e^{-th} &\geq& H_{t, f_S}(\bar{S})\\
&=& \sum_{v \in S} \frac{d_v}{\vol(S)} H_{t, v}(\bar{S})\\
&\geq& \epsilon \frac 1 {\vol(S)} \sum_{u \not \in R} d_u\\
&\geq & \epsilon\frac{\vol(\bar{R})}{\vol(S)}) \geq \epsilon \delta.
\end{eqnarray*}
which contradicts the assumption that $t < \epsilon \delta/h$. From  (\ref{h1}), we have $H_{t,f_S}(S)=1-H_{t,f_S}(\bar S)$. Thus we conclude that $\vol(R) \geq (1-\delta) \vol(S)$ as desired.

For the special case of $\epsilon = 1-e^{-2th}$, we can choose $\delta < 1/2$ and therefore (\ref{m33}) holds.
\qed

For a finite-dimensional Riemannian manifold with appropriate conditions on the growth rate of  neighborhoods, the heat kernel estimate is usually of the following form (see \cite{ly}):
\begin{eqnarray*}
H_t(x,y) \sim e^{-\dist(x,y)^2/t}\end{eqnarray*}
A natural question is to derive  similar estimates for the heat kernels for   graphs and  graphlets. However, there are a number of obstacles.
A general graph can be quite different from a discretization of  a $d$-dimensional Riemannian manifold. The graph distance increases  when the discretizations
approach the limit. The rate of changes for graph distances can be quite different for various graphlets. Nevertheless,
as we see from the above theorem, the heat kernel can be related to the Cheeger ratio  in  an analogous exponential format.

\subsection{The heat kernel for graphets}
Suppose a graph sequence $G_n,$ for $  n = 1, 2, \ldots $,  converges to  the graphlet $(\Omega, \Delta)$. We can use the heat kernel to define a family of embeddings with a parameter $t\geq 0$.
For each $t \geq 0$ and $ for $  n = 1, 2, \ldots $,  the embedding for a graph $G_n$  with vertex set $V_n$
 is obtained  by mapping  each vertex $u$ in  $V_n$  to the heat kernel function  $H^{(n)}_{t,u}  :  V_n \rightarrow [0,1] $ as defined in
the previous section. The function $H^{(n)}_{t, u}$ can be extended to the domain $ [0,1]$ by  defining for $y \in [0,1]$,
$H^{(n)}_{t,u}(y)=nH^{(n)}_{t,u} (v) $ if $y $ is in the interval $ I_v$ associated with vertex $v$ in $V_n$. In other words, we have for $H^{(n)}_{t, u}:
[0,1] \rightarrow [0,1]$:
\begin{eqnarray*}
 \int_{I_v} H^{(n)}_{t,u}(y) dy & =&H^{(n)}_{t,u}(v) \\
\text{and} ~~~~~~~~ \int_0^1 H^{(n)}_{t,u}(y) dy & =&\sum_{v \in V_n} H^{(n)}_{t,u}(v)=1 .
\end{eqnarray*}
Hence, $H^{(n)}_{t,u}$ can be viewed as  a measure defined on $[0,1]$. For $x $ in $\Omega$ labelled by $[0,1]$, we have
\begin{eqnarray*}
H_{t,x}(y) = \lim_{n \rightarrow \infty} H^{(n)}_{t,u_n}(y)
\end{eqnarray*}
where $ x \in I_{u_n} $ with $u_n \in V_n$ and therefore ${u_n}$ converges to $x$.

 For   $G_n$,   the family of $H^{(n)}_{t,u}, u \in V_n$ can  in general  be of high dimensions and for $x $ in $\Omega$, the family of $H_{t,x}$ can be more complex.
 Nevertheless,   there are several possibilities for such embedding families to have 
 succinct representations. Examples include the cases that (i)~   $\Omega$ is finitely generated; or (ii)~ $\Delta$ can be approximated by using the first $k$ eigenfunctions; or (iii) 
 with the appropriate choice of $t$, the embeddings $H_{t,x}$ can be locally focused (i.e., with almost all support) on 
an induced subgraph of a subset  $S$ of vertices  with small Cheeger ratio. This will be further discussed in Section \ref{cheeger}.

We remark that in spectral geometry, a main method  to compare  two different Riemannian manifolds is to first embed them into the same Hilbert space (see Gromov \cite{g} and Ber\'ard \cite{bbg}). Usually, such  embeddings  are required to preserve distances. However, for graphlets, the setting is
quite different. For example, the graph sequence consisting of paths with increasing and diverging diameters converge to $[0,1]$ of diameter $1$. Therefore, it can not be expected
that any notion of  graph distance in a graph sequence can be  amenable to preserves distance in such graph embeddings.  

Instead, we will show that the embeddings using the heat kernel of the graphlets preserve   Cheeger ratios in
a similar way as the Cheeger inequalities using eigenfunctions.  We note that the general isoperimetric problem of finding a subset with the smallest Cheeger ratio is a difficult 
problem since
there are exponentially many subsets in a graph. Previously, the spectral partitioning algorithms use eigenfunctions to restrict the choices of subsets to
a linear number  while the best Cheeger ratio among the restricted subsets is still within a quadratic factor of the optimal (see \cite{ ch0}). Here we will use the heat kernel instead of
the eigenfunctions in a way that the performance guarantee remains the same. The additional advantage of using the heat kernel is the fact that the parameter $t$ can be chosen so that the heat kernel can be approximated by
a truncated version with small  support depending only on $t$.  A local partitioning algorithm was previously given by using the heat kernel instead of the usual
PageRank (which is used to quantitatively rank vertices in a graph  \cite{hk, hk1}).  

For a subset $S$ of  $\Omega$ with $\mu(S) \leq 1/2$,  the Cheeger ratio of $S$ can be defined as
\begin{eqnarray}
 h(S) &= & \frac{\mu(\partial(S))}{\mu(S)}\\
&=&\frac{\langle \chi_S, \Delta \chi_S \rangle_\mu}{\mu(S)} 
\end{eqnarray}
which is consistent with the Cheeger ratios defined in the graphs.

From Theorems \ref{t3} and  \ref{t33},  using the fact that $\Omega$ is the graph limit of the graph sequence, we  have the following result.

 \begin{theorem}
 \label{t333}
 For  a subset $S$ in $\Omega$ with $h(S)=h$,
the following holds:

\begin{description}
\item[(a)]
\begin{eqnarray*}
\int_S\frac{ \mu(x)}{\mu(S)} \int_S  H_{t,x}(y) \geq e^{-th}
\end{eqnarray*}
(Note that here $H_{t,x}$ is taken as a measure.)
\item[(b)]
For given $\epsilon, \delta > 0$,  there is a subset $R$ with $\vol(R) \geq (1-\delta) \mu(S)$ such that
for any $x$ in $R$, we have
\begin{eqnarray}
\label{m333xxx}
 \int_{S} H_{t,x}(y) \geq 1-\epsilon \end{eqnarray}
provided $t \leq \epsilon \delta/h$.
\item[(c)] There is a subset $T$ with $\vol(T) \geq  \mu(S)/2$ such that
for any $x$ in $T$, we have
\begin{eqnarray}
\label{m333x}
\int_S H_{t,x}(y) \geq e^{-2th}.
 \end{eqnarray}
\end{description}
\end{theorem}

To establish  Cheeger-type inequalities  we will need to derive upper  bounds for $H_{t,x} (y)$'s
which  will be further discussed in Section \ref{cheeger}.

\ignore{
\subsection{Graph  embeddings preseving the Cheeger ratio using the heat kernel }
In a graphlets $(\Omega, \Delta)$,  each $x$ is associated with its heat kernel function $H_{t,x}$ which can be viewed as an embedding of $\Omega$ to $[0,1]^{[0,1]}$.
The functions $H_{t,x}$ has many strong properties, in particular, in preserving the Cheeger ratios.

For  a funciton  $f~:~\Omega \rightarrow {\mathbb R}$ and  real value $r \in [0,1]$, the {\it segment}  $S_{r}$  of $f$ denotes a subset  of $\Omega$
with $\mu(S_r)=r$ so that $f(x)/\mu(x) \geq f(y)/\mu(y)$ for any $x \in S_{r}$ and $y \not \in S_{r,f}$.  In other words, if we arrange elements  $x$ of $\Omega$ in the 
decreasing order of $f(x)/\mu(x)$ and select  elements having large values to form a set $S_{r}$  with  $\mu(S_{r})=r$.   For a given $s$, we consider
 the least Cheeger
ratio $h(S_{r})$ over all  segments $S_r$, $r \leq s$, denoted by $\kappa_{r,f}$, called  the $r$-local Cheeger
ratio determined by  $f$.

We will  examine local Cheeger ratio determined by   the heat kernel     functions. The idea of the proof is similar to the proof of the Cheeger inequality in \cite{ch0}. A weaker  discrete version of (\ref{echee1}) was given
in \cite{hk1} with a different proof.
\begin{theorem}
\label{tchee1}
For a subset $S \subset \Omega $ with the Cheeger ratio $h(S)=h$ and $\mu(S)=s \leq 1/4$, there is a subset $T$ of $S$ with $\mu(S) \geq s/2$ such that for $x
\in T$,  the heat kernel   $H_{2t,x}$ satisfies:
\begin{eqnarray}
\label{echee1}
-   \frac { \frac{\partial}{\partial t} H_{2t,x}(x)}{H_{2t,x}(x)} \geq 
\frac {(1-\delta)} 2 \kappa_{t,x,2s}^2 
\end{eqnarray}
provided $\delta = 1-e^{-4th} \leq 1/4$ where $\kappa_{t,x,2s} $ is the($2s$)-local Cheeger ratio  determined by $H_{t,x}$.
\end{theorem}

\proof
For $x$ in $T$, we define 
\[  f_t(y)=\frac{H_t(x,y)}{\mu(y)}. \]

\noindent
{\it Claim A:}
For $\epsilon=1-e^{-2th}$,  the heat kernel  rank function  $f_t$ satisfies the following:
\begin{align}
(a)~~&-\frac{\partial}{\partial t} f_{2t} (x) = 2 \int_{\Omega} |\nabla f_t(e)|^2 \mu_e, \label{m4a}\\
(b)~~&~~~~~~~f_{2t}(x) = \int_{\Omega} f_t^2(y) \mu(y), \label{m4b}\\
(c)~~&\text{For $r \geq s$,} ~~\int_{ S_{r}} f_t(y)\mu(y)  = 1-\epsilon, \label{m4c} \\
(d)~~&\text{For $r \geq 2s$,}~~\inf_{S_{r}} f_t(y)\mu(y) \leq \frac{2\epsilon}{r}.\label{m4d}
\end{align}

\noindent
{\it Proof of the Claim:}\\
($a$) and ($b$) are consequences of Lemma \ref{m1} (ii) and (iii).
Theorem \ref{t33} implies ($c$). Now consider $S_r$ with $r \geq 2s$ and we have
\begin{eqnarray*}
\inf_{y \in S_{r}}f_t(y) \leq \frac 2 {r } \int_{ S_{r}\setminus S_{r/2}}f_t(y)\mu(y)  \leq 
\frac{2\epsilon}{r} .
\end{eqnarray*}
The Claim is proved.

Using the above claim, we have the following.
\begin{eqnarray}
 \frac {-\frac{\partial}{\partial t} f_{2t,x}(x)}{f_{2t,x} (x)}
&=& \frac{2 \int_\Omega |\nabla f_t(e)|^2 \mu_e}{\int_\Omega f_t^2(y) \mu(y)}.\label{nn}
\end{eqnarray}
For a graph $G$, we can define the operator $\bar \nabla $ for $u,v$ with $A(u,v) \not = 0$, by
\[ \bar \nabla f(u,v)= f(u) + f(v)= 2 f(u)- \nabla f(u,v). \]
The operator $\bar \nabla$ can be extended to the graph limit $\Omega$, satisfying:
\begin{eqnarray}
 |\nabla f(e) \bar \nabla f(e)|&= &|\nabla f^2(e)|,\label{nn1}\\
 \int_\Omega |\bar \nabla f(e)|^2 \mu_e &=& 2 \int_\Omega f^2(y) \mu(y)- \int_\Omega | \nabla f(e)|^2 \mu_e  \nonumber\\
& \leq& 2 \int_\Omega f^2(y) \mu(y). \label{nn2}
 \end{eqnarray}
Returning to (\ref{nn}) and using (\ref{nn1}) and (\ref{nn2}), we have
\begin{eqnarray*}
 \frac {-\frac{\partial}{\partial t} f_{2t}(x)}{f_{2t} (x)} &=&  \frac{2 \int_\Omega |\nabla f_t(e)|^2 \mu_e\int_\Omega |\bar \nabla f_t(e)|^2 \mu_e
}{\int_\Omega f_t^2(y) \mu(y)\int_\Omega |\bar \nabla f_t(e)|^2 \mu_e}\\
&\geq&\left(\frac{\int_\Omega |\nabla f_t^2(e)| \mu_e
}{\int_\Omega f_t^2(y) \mu(y)}\right)^2\\
&\geq&\left(\frac{\int_0^{2s}dr\int_{\partial(S_r)} |\nabla f_t^2(e)| \mu_e
}{\int_\Omega f_t^2(y) \mu(y)}\right)^2\\
&=&\left(\frac{\int_0^{2s}\mu(\partial(S_r))~ d f_t^2(r)
}{\int_\Omega f_t^2(y) \mu(y)}\right)^2\\
&\geq&\left(\frac{\int_0^{2s}\kappa_{t,x,2s} \mu(S_r))~  d f_t^2(r)
}{\int_\Omega f_t^2(y) \mu(y)}\right)^2\\
&\geq&\kappa_{t,x,2s}^2 \left(\frac{\int_0^{2s}f_t^2(r)~ d \mu(S_r)) 
}{\int_\Omega f_t^2(x) \mu(x)}\right)^2\\
&=&\kappa_{t,x,2s}^2 \left(\frac{\int_{S_{2s}} \big(f_t^2(y) -f_t^2[r]) \mu(y) 
}{\int_\Omega f_t^2(y) \mu(y)}\right)^2\\
\end{eqnarray*}
We now use (\ref{m4c} ) and (\ref{m4d})  to obtain:
\begin{eqnarray*}
f^2[2s] 2s &\leq& \frac{2 \epsilon^2} s\\
&\leq& \frac{2 \epsilon^2} {(1-\epsilon)^2 s}\bigg( \int_{S_{2s}} f_t(y) \mu(y) \bigg)^2\\
&\leq& \frac{4 \epsilon^2} {(1-\epsilon)^2 } \int_{S_{2s}} f_t^2(y) \mu(y) 
\end{eqnarray*}
by using H\"{o}lder's inequality.
Therefore we have
\begin{eqnarray*}
 \frac {-\frac{\partial}{\partial t} f_{2t}(x)}{f_{2t} (x)}
 &\geq&\kappa_{t,x,2s}^2 \left(\frac{\int_{S_{2s}} f_t^2(y)\mu(y) -f_t^2[2s] 2s
}{\int_\Omega f_t^2(y) \mu(y)}\right)^2\\
&\geq &((1-\epsilon)^2-4\epsilon^2) \kappa_{t,x,2s}^2\\
&\geq &(1-\delta) \kappa_{t,x,2s}^2
\end{eqnarray*}
 by our assumption on $\delta$.
Theorem \ref{tchee1} is proved.
\qed

\begin{theorem}
\label{tchee2}
For a subset $S$ with $\mu(S)=s$ and $h(S)=h$, there is a subset $T$ of $S$  with $\mu(T) \geq s/2$ so that
for   $\delta = 1-e^{-4th} \leq 1/4$, the following holds:
\begin{align*}
(i) &~ \text{For $x \in T$},~~
 e^{-2th} \leq \int_{S} H_{t,x}(y)  \leq \sqrt{\frac {s}{\mu(x)} }e^{-{t\kappa_{t,x,2s}^2(1-\delta)}/2}\\
 &\text{where where $\kappa=\kappa_{t,x,2s} $ denotes the minimum $2s-$local Cheeger ratio  using }\\
 &\text{ segments $S_r$, $r \leq 2s$, determined by $H_{t,x}$.}
\\
(ii) &~~
\frac 1 2  e^{-4th} \leq \int_T \frac{\mu(x)}s \Big( \int_{S_s} H_{t,x}(y)\Big)^2 \leq \frac{1}{\mu(T)} \int_T H_{2t,x}(x) \mu(x)  \leq e^{-t\kappa_{t,2s}^2(1-\delta)}\\
 &\text{where where $\kappa=\kappa_{t,2s} $ denotes the minimum $2s-$local Cheeger ratio  using }\\
 &\text{ segments $S_r$, $r \leq 2s$, determined by $H_{t,x}$ over all $x$ in $T$.}
 \end{align*}
\end{theorem}

We note that in ($i$), there is a term $\mu(S)/\mu(x)$ which contributes a term of  $O(\log n)$ when we compare $ht$ and $\kappa_{t,x,2s}$ by taking logarithms
of both sides of the inequality. This works for the discrete cases of finite graphs but is rather restrictive for the general case.
In order to getting rid of this untidy factor, we consider the expected case in ($ii$) by integrating over all $x$ in $T$.

\noindent
{\it Proof of Theorem \ref{tchee2}:}\\
By solving the inequalities in (\ref{echee1}), we have
\begin{eqnarray}\label{chee3}
\hkr_{2t,x}(x) \leq   e^{-(1-\delta)t \kappa_{t,x,2s}^2}.
\end{eqnarray}

To prove the upper bound, we use  Lemma \ref{m1} ($iii$):
\begin{eqnarray*}
H_{2t,x}(x)&\geq&  \mu(x) \int_\Omega \frac{H_{t,x}(y)^2}{\mu(y)}dy\\
&\geq&\mu(x) \int_{ S_s}\big(H_{t,x}(y)/\mu(y)\big)^2\mu(y)dy\\
&\geq& \frac{\mu(x)} s \bigg( \int_{ S_s}H_{t,x}(y)dy\bigg)^2
\end{eqnarray*}

Together with (\ref{chee3}),  we have
\begin{eqnarray*}
 e^{-t\kappa_{t,x,s}^2(1-\delta)}& \geq& H_{2t,x}(x) \geq 
\frac{\mu(x)} s\Big( \int_{S_s} H_{t,x}(y) \Big)^2 \\
&\geq& \frac{ \mu(x)e^{-4th}}{ s}. 
\end{eqnarray*}
where $\epsilon$ is as defined  in the proof of Theorem \ref{tchee1}.
This implies
\[ e^{-2th} \leq \int_{S_s} H_{t,x}(y)dy  \leq \sqrt{\frac {s}{\mu(x)} }e^{-{t\kappa_{t,x,2s}^2(1-\delta)}/2 }\]
and ($i$) is proved.
To prove ($ii$), we consider
\[ F(t)=  \frac{1}{\mu(T)} \int_T H_{2t,x}(x) \mu(x) . \]
By Theorem \ref{tchee1}, we have
\begin{eqnarray*}
-\frac{\partial }{\partial t} \log F(t) &=&
\frac{ -\frac{1}{\mu(T)} \int_T \frac{\partial}{\partial t} H_{2t,x}(x) \mu(x)}{F(t)} \\
 &\geq&\frac{\frac{1}{\mu(T)} \int_T (1-\delta) \kappa^2_{t,x,2s} H_{2t,x}(x) \mu(x)}{F(t)} \\
  &\geq&\frac{\frac{1}{\mu(T)} \int_T (1-\delta) \kappa^2_{t,,2s} H_{2t,x}(x) \mu(x)}{F(t)} \\
   &\geq&(1-\delta) \kappa^2_{t,,2s}.
\end{eqnarray*}
By solving the inequality with  $F(0) =1$,
we have
\[ F(t) \leq e^{-(1-\delta) \kappa^2_{t,2s}t}. \]
In the other direction, we have
\begin{eqnarray*}
F(t) &\geq&\frac{1}{\mu(T)} \int_T \Big( \frac{H_{2t,x}(x)}{\mu(x)}\Big) \mu(x)^2\\
&=&\frac{1}{\mu(T)} \int_T \mu(x)^2\Big( \int_{\Omega} \frac{H_{t,x}(y)^2}{\mu(y)}\Big) \\
&\geq&\frac{1}{\mu(T)} \int_T \mu(x)^2\Big( \int_{S_s} \frac{H_{t,x}(y)^2}{\mu(y)}\Big) \\
&\geq&\frac{1}{\mu(T)} \int_T \mu(x)^2\frac{1}{s} \Big( \int_{S_s} H_{t,x}(y)\Big)^2 \\
&\geq&\frac{1}{2} \int_T\Big( \frac{\mu(x)}{\mu(T)}\Big)^2 e^{-4th}\\
&\geq&\frac{e^{-4th}}{2}\Big( \int_T \frac{\mu(x)}{\mu(T)}\Big)^2 \\
&\geq&\frac{e^{-4th}}{2}
\end{eqnarray*}
as desired.
\qed

We remark that Theorems \ref{tchee1} and \ref{tchee2}  provide an algorithm for finding a local cut together with a performance guarantee within a quadratic factor of the optimum.
The above inequality has a similar flavor as the Cheeger inequality  which is the base for the  spectral partitioning algorithm using eigenfunctions. 
However  eigenfunctions do not specify the size of the parts to be separated.
}
}
\section{Examples of graphlets }
\label{ex}

We here consider several examples of   graphlets $\mathcal{G}(\Omega, \Delta)$ which are formed from graph sequences  $G_n, $ for $  n = 1, 2, \ldots $. We will illustrate that
the eigenfunctions of $\Delta$ can be used to serve as a universal basis for  graphs $G_n$. The discretized adaptation  of graphlets will be called ``{\it lifted} graphlets''  for $G_n$, which
 are good approximations for the actual eigenfunctions in $G_n$ as $n $ approaches infinity. In some cases,
the lifted graphlets using  $\Delta$  are fewer 
than  the number of eigenfunctions  in $G_n$ and  in other cases,  there are more eigenfunctions of $\Delta$ than those of $G_n$.   We will  describe a universal basis for $G_n$,
as the union of two parts,  including the {\it primary} series (which are  the lifted graphlets)  and {\it complementary} series  (which are
  orthogonal to the primary series). In a way,  we will see that the primary series  captures the main structures of the graphs while
the complementary series  reflect the ``noise'' toward the convergence.  Before we proceed, some clarifications are in order.
\begin{itemize}
\item
The notion of orthogonality  refers to the usual inner product unless we specify other modified inner products such as the $\mu$-product  $\langle \cdot, \cdot \rangle_\mu$ or  the 
$\mu_n$-product. Sometimes, it is more elegant to use eigenfunctions that are orthogonal under the $\mu$-norm. However,  when
we are dealing with a finite graph  $G_n$  in  a graph sequence, we sometimes wish to  use only what we know about the finite graph $G_n$ and perhaps
the existence
of the limit without
the knowledge of the behavior of the limit (such as $\mu$). In such cases, we will use the usual inner product.
\item
The universal bases  are for approximating the eigenfunctions of the normalized Laplacian of $G_n$. 
In a graph $G_n$, its  Laplace operator $\Delta_n=I - D_n^{-1} A_n$ is not symmetric in general since
the left and right eigenfunctions are not necessarily the same. The universal basis is used for  approximating
the eigenfunctions of the normalized Lapalcian 
\[ {\mathcal L}_n = I- D_n^{-1/2} A_nD_n^{-1/2}, \]
which is equivalent to $\Delta_n$ and is  symmetric. Thus, ${\mathcal L}$ has orthogonal eigenfunctions.
\end{itemize}

\subsection{Dense graphlets}
\label{dense}
Suppose we have a sequence of dense graphs $G_n, $ for $  n = 1, 2, \ldots $,  with $\vol(G_n)=2|E(G_n)|=c_n n^2$ where the 
$c_n $ converge to a constant $c > 0$. In this case,  the
$\mu$-norm is equivalent to   other norms such as the cut-norm and subgraph-norm  in \cite{borgs}. 
By using the regularity lemma, the graphlets $\mathcal{G}(\Omega, \Delta)$ of a  dense graph sequence is  of a finite type. In other words, there is a graph $H$ on $h$ vertices
where $h$ is a constant (independent of $n$) such that $\Omega = \Omega_H$ is taken to be a partition of $[0,1]$ into $h$ intervals of the same length. 
Let $\varphi_1, \ldots, \varphi_h$ denote the eigenfunctions of $H$.

For $n=hm$ and $m \in {\mathbb Z}$, we will describe a basis   for a graph $G_{n}$. 
The primary eigenfunctions  can be written as
\begin{eqnarray*}
\phi_j^{(n)}(v)&=& \varphi_j( \lceil v/m \rceil)~~~~ \text{where $v \in \{1, 2, \ldots, n\}$ and $j=1, \ldots, h$},
\end{eqnarray*}
while the complementary eigenfunctions consist of $n-h=(m-1)h$ eigenfunctions as follows: For $1 \leq a \leq h, 1 \leq b \leq m-1$, 
\begin{eqnarray*}
\phi^{(n)}_{a,b}(a'm+b')&=&\begin{cases}
e^{2\pi i bb'/m} & \text{ if $a'+1=a$,}\\
0& \text{otherwise.}
\end{cases}
\end{eqnarray*}
\ignore{
\subsection{Affine graphlets for families of paths, cycles
 and their cartesian products}
Suppose we consider a sequence of paths  $P_n$ on $n$ vertices. It is not hard to see that the sequence converges to 
the line segment $\Omega=[0,1]$, which is in fact a  manifold.
The eigenfunctions for the Laplace operator $\Delta$ of $\Omega$ are just $\phi_x$ for $x \in [0,1]$,  defined by $\phi_x(y)=e^{\pi i xy}$.  For any integer $n$, the lifted graphlets  for $P_n$  are just $\phi_j$, $j=0, 1, \ldots, n-1$, such that
$\eta_j(k)=\phi_{j/n}(k/n)$ which coincide with the eigenfucntions of $P_n$.  We remark that
in this case,
$\Omega$ has far more eigenfunctions than those for $P_n$.

Consider  a sequence of cycles $C_n$ on $n$ vertices. Again the sequence converges to 
the line segment $\Omega=[0,1]$, where the two endpoints are only allowed to take the same value.
The eigenfunctions for $\Omega$ are just $\phi_x$ for $x \in [0,1]$,  defined by $\phi_x(y)=e^{2\pi i xy}$ . 
For any integer $n$, the lifted graphlets  for $C_n$  are just $\phi_j$, $j=0, 1, \ldots, n-1$ such that
$\eta_j(k)=\phi_{j/n}(k/n)$ which are the same as the eigenfucntions of $C_n$. 

Suppose we consider the cartesian products of cycles, $C_n \times C_n$, the $2$-dimensional torus
lattices. The usual
labeling by mapping vertices into $[0,1]$ no longer describes the geometry of $\Omega$. The geometry of $\Omega$ relies on  the gradiant operator as well as the Laplace
operator as discussed in Section \ref{geo}. The graph limit $\Omega$ here is better described as a $2$-dimensional metric measure space (with $L_1$-distance). 
}
\subsection{Quasi-random graphlets}

Originally, quasi-randomness is an equivalent class of graph properties that are shared by random graphs (see \cite{CGW89}). In the language of graph limits,
 quasi-random graph properties with edge density $1/2$ can be described as a graph sequence $G_n, $ for $  n = 1, 2, \ldots $, converging to graphlets $\mathcal{G}(\Omega, \Delta)$ where  $\Omega=[0,1]$ and 
 $\Delta(x,y)=1/2$, for $x \not = y$ and  $\mu(x)=\mu(y)$ for all $x,y$.
 Compared with the original equivalent quasi-random properties for $G_n$ (included in parentheses), the quasi-random graphlets with edge density $1/2$ satisfies 
 the following equivalent statements for the graph sequence $G_n$ where $n = 1, 2, \ldots$. 
\begin{description}
\item[(1)]
The graph sequence $G_n$  converges to graphlets $\mathcal{G}(\Omega, \Delta)$ in the spectral distance.\\
 ({\it \bf The eigenvalue property}: {\it The adjacency matrix of $G_n$ on $n$ vertices  has one eigenvalue $n/2 + o(n)$ with all other eigenvalues $o(n)$. })
 
\item[(2)]
The graph sequence  $G_n $  converges to graphlets $\mathcal{G}(\Omega, \Delta)$ in the cut-distance.\\
({\bf The discrepancy property}: {\it For any two subsets $S$ and $T$ of the vertex set of $ G_n$, there are $|S| \cdot |T|/2 + o(n^2)$  ordered pairs $(u,v)$ with $u \in S, v \in T$ and $\{u,v\}$ being an edge
of $G_n$}. )
\item[(3)]
The graph sequence $G_n$ converges to graphlets $\mathcal{G}(\Omega, \Delta)$ in the $C_4$-count-distance.\\
({\it \bf The co-degree property}: {\it For all but $o(n^2)$ pairs of vertices $u$ and $v$  in $G_n$,  $u$ and $v$ have $n/4 +o(n)$ common neighbors.)\\
({\bf The trace property}: The trace of the adjacency matrix to the $4$th power is $n^4/16 + o(n^4)$}.)
\item[(4)]
The graph sequence $G_n $  converges to graphlets $\mathcal{G}(\Omega, \Delta)$ in the subgraph-count-distance.\\
({\it  \bf The subgraph-property}: {\it For fixed $k \geq 4$ and for any $H$ on $k$ vertices and $l$ edges, the number of occurrence of $H$  as subgraphs in $G_n$  is $n^k/2^l+o(n^k)$}.)
\item[(5)]
The graph sequence $G_n$ converges to graphlets $\mathcal{G}(\Omega, \Delta)$ in the homomorphism-distance.\\
({\it  \bf The induced-subgraph-property}: {\it For fixed $k \geq 4$ and for any $H$ on $k$ vertices, the number of occurrence of $H$  as induced subgraphs in $G_n$  is $n^k/2^{\binom k 2}+o(n^k)$.} )

 \end{description}

For a quasi-random graph sequence, the primary graphlets for $G_n$ consists of the all $1$'s  vector $\mathbf 1$ and the complementary ones are  irrelevant  in the sense that they  can be any arbitrarily chosen orthogonal functions since all eigenvalues except for one approach zero. In other words, $\Delta$ as the limit of $G_n$ only has one nontrivial eigenfunction.

The generalization of quasi-randomness to sparse graphs and to graphs with general degree distributions \cite{spar, quasid} can also be
described in the framework of graphlets.
  In the previous work on  quasi-random graphs  with given degree distributions,  the results are not as strong since additional conditions are required in order to
  overcome various difficulties \cite{ spar,quasid}. By using graph limits, such obstacles and additional conditions can be removed.
In Section \ref{quasi}, we will give a complete characterization for quasi-random graphlets with any  given general degree distribution which 
include
the sparse cases.
\subsection{Bipartite quasi-random  graphlets}

A bipartite quasi-random graphlets $\mathcal{G}(\Omega, \Delta)$  can be described as follows: $\Omega$ is be partitioned into two parts
$A$ and $B$ while $W(x,y) $ is equal to some constant $\rho$ if ($x \in A, y \in B$) or ($x \in B, y \in A$), and $0$ otherwise. There are two nontrivial eigenvalues of $I - \Delta$, namely,
$1$ and $-1$. The eigenfuction $\phi_0$ associated with eigenvalue $1$  assumes the value $\phi_0(x)=1/\sqrt{\mu(A)}$ for $ x \in A$ and 
$\phi_0(y) = 1/\sqrt{\mu(B)}$ for $y \in B$. The eigenfunction $\phi_1$ associated with eigenvalue $-1$ is defined by $\phi_1(x)=1/\sqrt{\mu(A)}$ for $ x \in A$ and
$\phi_1(y) = -1/\sqrt{\mu(B)}$ for $y \in B$.

The bipartite version of quasi-random graphs is useful in the proof of the regularity lemma \cite{reg}.
Bipartite quasi-random graphlets, as well as  quasi-random graphlets, serve as the basic building blocks  for general types of graphlets. More on
this will be given in Sections \ref{r2} and \ref{rk}.
\subsection{ Graphlets of  bounded rank}
A quasi-random sequence is a graph sequence which converges to a graphlets of rank $1$ as we will see in this section.  We will further consider the generalization of
graph sequences which converge to a graphlets of rank $k$. This will be further examined in  Sections \ref{r2} and  \ref{rk}.
\section{The spectral distance and the discrepancy distance}

\label{2norms}

\subsection{The cut distance and the discrepancy distance}
In previous studies of graph limits,  a so-called cut metric that is often used  for which the distance of two graphs $G$ and $H$ which share the same set of vertices $V$ is measured by the following (see \cite{ borgs, fk}).
\begin{eqnarray}
\label{cut}
\cut(G , H)= \frac 1 {|V|^2}
\sup_{S, T \subseteq V} \left| E_{G}(S,T) - E_{H}(S,T) \right|
\end{eqnarray}
where $E_G(S,T)$ denotes the number of ordered pairs $(u,v)$ where $u $ is in $ S$, $v$ is in $ T$ and $\{u,v\}$ is an edge in $G$.

We will   define a  discrepancy distance  which is similar to but different from the above cut distance.  For two graphs $G$ and $H$ on the same vertex
set $V$,
the discrepancy distance, denoted by $\disc(G,H)$ is defined
as follows:
\begin{eqnarray}
\label{disc}
\disc(G, H)
=\sup_{S, T \subseteq V}\left| \frac{E_{G}(S,T)}{\sqrt{\vol_G(S)\vol_G(T)}}-\frac{E_{H}(S,T)}{\sqrt{\vol_H(S)\vol_H(T)}} \right|. \end{eqnarray}
We remark that the only difference between  the cut distance and the discrepancy distance is in the normalizing factor which will be useful
in the proof later.

For two graphs $G_m$ and $G_n$ with $m$ and $n$ vertices respectively, we use the labeling maps $\theta_n$ and $\eta_n$ to map $[0,1]$ to the vertices of $G_m$ and $G_n$, respectively.  We define the measures $\mu_m$ and $\mu_n$ on $[0,1]$ using the degree sequences of $G_m$ and $G_n$ repectively, as in
Section \ref{subdeg}.   From the definitions and substitutions, we can write:
\begin{eqnarray}
\label{enst}
 E_{G_n}(S,T)=\vol(G_n) \langle \chi_S, (I-\Delta_n) \chi_T \rangle_{\mu_n,\theta_n}.
 \end{eqnarray}
Therefore the  discrepancy distance in (\ref{disc}) can be written in the following general format:
\begin{align}
&\disc(G_m,G_n)\nonumber\\
&=\inf_{\theta_m \in {\F}_m, \eta_n \in \F_n} \sup_{S, T \subseteq [0,1]} \left| \frac{\langle \chi_{S}, (I-\Delta_m) \chi_{T} \rangle_{\mu_m,\theta_n}} 
{\sqrt{\mu_m(S)
 \mu_m(T)}
 }
-\frac{\langle \chi_{S}, (I-\Delta_n) \chi_{T} \rangle_{\mu_n,\eta_n}}{\sqrt{\mu_n(S)
 \mu_n(T)}
 } \right| \nonumber\\
\label{gm}
\end{align}
where $S, T$ range over all integrable subsets of $[0,1]$.
We can rewrite (\ref{enst}) as follows.
\begin{eqnarray}
\label{aa}
E_{G_n}(S,T)&=&\vol(G_n) \int_{x \in \Omega} \chi_S(x) \big((I-\Delta_n) \chi_T\big)(x) \mu_n(x).
\end{eqnarray}
Alternatively,  $E_{G_n}(S,T)$ was  previously expressed   (see   \cite{lsz}) as follows:
\begin{eqnarray}
\label{previous}
E_{G_n}(S,T)&=& n^2 \int_{x \in S}\int_{y \in T} { W}(x,y)~ ds~ dt
\end{eqnarray}

The two formulations (\ref{aa}) and (\ref{previous})  look quite different but are of the same form when the  graphs involved are regular.
However,  the format in (\ref{previous}) seems hard to extend to   general graph sequences with smaller edge density.

Although the above definition in (\ref{gm}) seems complicated, it can be simplified when the degree sequences converge. Then,  $\mu_m$ and $\mu_n$ are to be approximated
by the measure $\mu$ of the graph limit. In such cases, we define
\begin{align}
&~~\disc_\mu(G_m,G_n)\nonumber\\
 &=\inf_{\theta_m \in {\F}_m, \eta_n \in \F_n} \sup_{S, T \subseteq [0,1]} \left| \frac{\langle \chi_{S}, (I-\Delta_m) \chi_{T} \rangle_{\mu,\theta_m}} 
{\sqrt{\mu(S)
 \mu(T)}
 }
-\frac{\langle \chi_{S}, (I-\Delta_n) \chi_{T} \rangle_{\mu,\eta_n}}{\sqrt{\mu(S)
 \mu(T)}
 } \right|\nonumber\\
 &=\sup_{S, T \subseteq [0,1]} \frac 1 {\sqrt{\mu(S)\mu(T)}} \left|\langle \chi_{S}, (\Delta_m-\Delta_n) \chi_{T} \rangle_{\mu} \right|.
 \label{disc2}
\end{align}
where $S, T$ range over all integrable subsets of $[0,1]$ and we suppress the labelings $\theta,\eta$ which achieve the infininum.

We will show  that the convergence using the spectral distance defined under the $\mu$-norm  is equivalent to the convergence using  the discrepancy distance   in Section \ref{2norms}.  

\subsection{The equivalence of convergence using spectral distance and the discrepancy distance}

We will prove the following theorem concerning the equivalence of the convergences under the spectral distance (as in (\ref{eq10})) and the discrepancy distance (as in (\ref{gm})). The result holds without any density restriction on the graph sequence. The proof extends  similar techniques in Bilu and Linial \cite{BL} and  \cite{BN, butler} for regular or random-like graphs to graph sequences of general
degree distributions. 
\begin{theorem}
\label{2norm}
Suppose  the degree distributions $\mu_n$,  of a graph sequence $G_n, $ for $  n = 1, 2, \ldots $,  converges to $\mu$.
The following statements are equivalent:

\noindent
(1) The graph sequence  $G_n$ converges under the spectral distance.

\noindent
(2) The graph sequence $G_n$  converges under the $\disc$-distance.
\end{theorem}
\proof
Suppose that for a given $\epsilon > 0$, there exists an $N> 1/\epsilon$ such that for $n > N$, we have
\begin{eqnarray*}
 \| \mu_n-\mu \|_1 < \epsilon. \end{eqnarray*}
The proof for $ (1)$  
 $\Rightarrow$ 
 $(2)$ is rather straightforward and can be shown as follows:

Suppose (1) holds and we have, for $m, n > N$, $\|\mu_m-\mu\|_1 < \epsilon$, $\|\mu_n-\mu\|_1 < \epsilon$ and
 $\|\Delta_m - \Delta_n\|_\mu < \epsilon$.  (Here we omit the labeling maps $\theta_m, \theta_n$ to simplify the notation.) Then,
\begin{align*}
\disc(G_m,G_n)
&=\sup_{S, T \subseteq [0,1]} \left| \frac{\langle \chi_{S}, (I-\Delta_m) \chi_{T} \rangle_{\mu_m}} 
{\sqrt{\mu_m(S)
 \mu_m(T)}
 }
-\frac{\langle \chi_{S}, (I-\Delta_n) \chi_{T} \rangle_{\mu_n}}{\sqrt{\mu_n(S)
 \mu_n(T)}
 } \right|\\
 &\leq
\sup_{S, T \subseteq [0,1]}\frac{1}{\sqrt{\mu(S)\mu(T)}} \left| \langle \chi_{S}, (\Delta_m-\Delta_n) \chi_{T} \rangle_{\mu}
 \right| + 4 \epsilon
\\
&=\sup_{S, T \subseteq [0,1]}\frac{1}{\|\chi_S\|_\mu\|\chi_T\|_\mu} \left| \langle \chi_{S}, (\Delta_m-\Delta_n) \chi_{T} \rangle_{\mu}\right|
+ 4 \epsilon \\
&\leq \|\Delta_m-\Delta_n\|_\mu + 4 \epsilon\\
&\leq 5\epsilon.
\end{align*}

To prove (2) $\Rightarrow$ (1), we assume that  for $M=\Delta_n-\Delta_m$
  \begin{eqnarray}
  \label{assume}
\left|  \langle \chi_S, M\chi_T \rangle_\mu  \right| \leq
\epsilon \sqrt{\mu(S) \mu(T)} 
\end{eqnarray}
for some  $\epsilon >0$ for  any two integrable subsets  $S, T \subseteq [0,1]$. It is enough to show that for any two integrable functions $f,g: [0,1] \rightarrow {\mathbb R}$, we have
\begin{eqnarray}
\label{eq35}
|\langle f,  M g \rangle_\mu | \leq 20 \epsilon  \log (1/\epsilon ) \| f\|_\mu \|g \|_\mu
\end{eqnarray}
provided $\epsilon < .02$.

\ignore{
\begin{lemma}
\label{ma22}

  Assume  an operator $M: [0,1] \times[0,1] \rightarrow \mathbb R$
  is the difference of two positive operators,
  $M=A-B$ where $A(x,y) \geq 0, B(x,y) \geq 0$ for all $x,y \in [0,1]$. Suppose   $A{\mathbf 1} = B{\mathbf 1} = \mathbf 1$ and suppose $M$ satisfies
  \begin{eqnarray}
  \label{assume}
\left|  \langle \chi_S, M\chi_T \rangle_\mu  \right| \leq
\gamma \sqrt{\mu(S) \mu(T)} 
\end{eqnarray}
for some  $\gamma >0$ for  any two integrable subsets  $S, T \subseteq [0,1]$. Then for any two integrable functions $f,g: [0,1] \rightarrow {\mathbb R}$, we have
\[ 
\langle f,  M g \rangle_\mu \leq 20 \gamma  \log (1/\gamma ) \| f\|_\mu \|g \|_\mu
\]
provided $\gamma < .02$.
\end{lemma}

To prove this, we  use ideas in  \cite{BL, BN, butler}. 
First we prove the following  fact:
}

The proof of  (\ref{eq35}) follows a sequence of claims.

\noindent
{\it Claim 1:}
For  an integrable function $f$ defined on ${[0,1]}$ with $\|f\|_\mu=1$, for any $\epsilon > 0$, there exists an $N(\epsilon)$ such that for any $n > N(\epsilon)$ there is a function $h$ defined on $[0,1]$ satisfying :\\
(1) $ \| h\|_\mu \leq 1$,\\
(2) $\|f-h\|_\mu \leq 1/4+\epsilon$,\\
(3) The  value $h(y)$ in the interval $( (j-1)/n,j/n]$ is a constant $h_j$ and $h_j$ is 
 of the form $(\frac 4 5)^j$ for integers $j$.\\
{\it Proof of Claim 1:}
Since $f$ is integrable, for a given $\epsilon$, we can approximate $\|f\|^2_\mu$ by
 a function $\bar f$, with $\bar f (x)=f_j$ in $((j-1)/n,j/n]$, such that 
 \[ \left| \int_0^1 (f-\bar f)^2(x) \mu(x) \right| < \epsilon. \]
For   $\bar f=(f_j)_{1 \leq j \leq mn}$, we define  $h=(h_j)_{1 \leq j \leq mn} $ as follows. If $f_j=0$, we set $h_j=0$. Suppose $f_j \not = 0$, there is a unique integer $k$ so that $(4/5)^{k}< |f_j| \leq (4/5)^{k-1}$. We set $h_j=\sign(f) (\frac 4 5)^k $ where
$\sign(f_j) =1$ if $f_j$ is positive and $-1$ otherwise. Then 
\[ 0 < |f_j-h_j| \leq  (\frac 4 5)^{k-1} -(\frac 4 5)^k = \frac 1 4 (\frac 4 5 )^k < \frac 1 4 |f_j|,\]
which implies 
$\|f-h\|_\mu^2 \leq \epsilon + \sum_j \int_0^1|f_j-h_j|^2 {\mu}(x) \leq  \epsilon +\frac 1 {16} \sum_t |f_j|^2 \mu(x) = \frac 1 {16}+\epsilon$.
Claim 1 is proved.

\noindent
{\it Claim 2:}
Suppose there are  functions $f',g'$ satifying 
$
\|M\|_\mu$ $=|\langle f',Mg'\rangle_\mu|$ and $\|f'\|_\mu=\|g'\|_\mu=1$. If $f,g$ are functions such that $\|f\|_\mu, \|g\|_\mu \leq 1$ and $\|f'-f\|_\mu\leq 1/4+\epsilon$,$\|g'-g\|_\mu \leq 1/4+\epsilon$, then 
\begin{eqnarray}
\label{eq3f}
\|M\|_\mu \leq (2+4 \epsilon) |\langle f, Mg \rangle_\mu|.
\end{eqnarray}

Claim 2 can be proved by using Claim 1 as follows:
\begin{eqnarray*}
\|M\|_\mu &=&|\langle f',Mg'\rangle_\mu|\\
&\leq&|\langle f,Mg\rangle_\mu|
+|\langle f'-f,Mg\rangle_\mu|+|\langle f',M(g'-g)\rangle_\mu|\\
&\leq& |\langle f,Mg\rangle_\mu|+\big(\frac 2 4+2 \epsilon\big) \|M\|_\mu.
\end{eqnarray*}
This implies
$
\|M\|_\mu \leq (2+4 \epsilon) |\langle f, Mg \rangle_\mu|$, as desired.

From Claims 1 and 2, we can upper bound $\|M\|_\mu$ to within a multiplicative factor of $2+4\epsilon$ by
bounding of $ |\langle f, Mg \rangle_\mu|$ with   $f,g$ of the following form: Namely, $
f =\sum_{t} (\frac 4 5)^t f^{(t)}$, where the  $f^{(t)}$ denotes the indicator function of  $\{x : \bar f(x)= (\frac 4 5 )^t\}$.   Similarly we write
 $g = \sum_t (\frac 4 5)^t g^{(t)}$, where the  $g^{(t)}$ denotes the indicator function of  $\{y : \bar g(y)= (\frac 4 5 )^t\}$. 
Now  we choose $\kappa= \log_{4/5} \epsilon$ and we consider
\begin{align*}
\left| \langle f, M g \rangle_\mu
\right| &\leq \sum_{s,t} (\frac 4 5 )^{s+t} \left|\langle f^{(s)}, Mg^{(t)} \rangle_\mu \right|\\
&\leq \sum_{|s-t|\leq \kappa} (\frac 4 5 )^{s+t} \left|\langle f^{(s)}, Mg^{(t)} \rangle_\mu \right| \\
&~~~+ \sum_{s} (\frac 4 5 )^{2s+\kappa}\sum_t \left|\langle f^{(s)}, Mg^{(t)} \rangle_\mu \right| \\
&~~~+ \sum_{t} (\frac 4 5 )^{2t+\kappa}\sum_s \left|\langle f^{(s)}, Mg^{(t)} \rangle_\mu \right| \\
&= X + Y + Z.
\end{align*}
We now bound the three terms separately.  For a function $f$, we denote  $\mu(f)=\mu(\supp(f))$ to be the measure
 of the support of $f$.
 Using the assumption (\ref{assume}) for $(0,1)$-vectors and the fact that $f^{(s)}$'s are orthogonal (as well as the $g^{(t)}$'s), we have
\begin{align*}
X&=\sum_{|s-t|\leq \kappa} (\frac 4 5 )^{s+t} \left|\langle f^{(s)},Mg^{(t)} \rangle_\mu \right|\\
&\leq  \epsilon  \sum_{|s-t|\leq \kappa} (\frac 4 5 )^{s+t}\sqrt{\mu(f^{(s)})\mu( g^{(t)})}\\
&\leq \frac{ \epsilon } 2 \sum_{|s-t|\leq \kappa}\big( (\frac 4 5 )^{2s}\mu(f^{(s)})+(\frac 4 5 )^{2t}\mu( g^{(t)})\big)\\
&\leq \frac{ \epsilon (2 \kappa+1)} 2 \big(\sum_{s} (\frac 4 5 )^{2s}\mu(f^{(s)})+\sum_{t} (\frac 4 5 )^{2t}\mu( g^{(t)})\big)\\
&\leq  \epsilon (2 \kappa+1),
\end{align*}
since each term can appear at most $2 \kappa +1$ times.
For the second term we have, by using Lemmas \ref{ma1}, the following:
\begin{align*}
Y&\leq \sum_{s} (\frac 4 5 )^{2s+\kappa}\sum_t \left|\langle f^{(s)}, Mg^{(t)} \rangle_\mu \right|\\
&\leq (\frac 4 5 )^{\kappa} \sum_{s} (\frac 4 5 )^{2s} \langle f^{(s)}, |(\Delta_m-\Delta_n)\sum_t g^{(t)}| \rangle_\mu \\
&\leq(\frac 4 5 )^{\kappa} \sum_{s} (\frac 4 5 )^{2s} \langle f^{(s)}, (\Delta_m+\Delta_n){\mathbf 1} \rangle_\mu \\
&\leq 2(\frac 4 5 )^{\kappa} \sum_{s} (\frac 4 5 )^{2s} \langle f^{(s)}, {\mathbf 1} \rangle_\mu \\
&\leq 2 (\frac 4 5 )^{\kappa}
\sum_s\mu(f^{(s)})\\
&\leq 2 (\frac 4 5 )^{\kappa}
\end{align*}
The third term can be bounded in a similar way.  Together, we have
\begin{align*}
\|M\|_\mu & \leq (2 +4\epsilon)\big( \epsilon( 2 \kappa +1)+ 4 (\frac 4 5)^{\kappa})\\
& \leq (2+4\epsilon) \big( \epsilon( 2\frac{ \log(1/\epsilon)}{\log 5/4}+1)+ 4 \epsilon \big)\\
&\leq \frac{4+8\epsilon }{\log(5/4)} \epsilon \log(1/\epsilon) + 8\epsilon
\\
&\leq 20 \epsilon \log (1/\epsilon)
\end{align*}
 since $\frac {4} {\log 5/4} \approx 17.93$ and $\epsilon < .02$.
This completes the proof of the theorem. 
\qed

\section{Quasi-random graphlets with general degree distributions -- graphlets of rank $1$  }
\label{quasi}

We consider a graph sequence that consists of quasi-random graphs with degree distributions converging to some general degree distribution.
We will  give characterizations for  a quasi-random graph sequence by stating a number of equivalent properties. Although the proof is
mainly  by summarizing previous known facts, the format of graph limits helps in  simplifying the previous various statements for quasi-random graphs with general degree distributions  including the cases for sparse graphs.

\noindent
\begin{theorem}
The following statements  are equivalent for a graph sequence $G_n,$ where $ n=1,2, \dots$.
\label{t2}
\begin{description}
\item[($i$)]
 $G_n $'s  form a quasi-random sequence with degree distribution converging to $\mu$.
\item[($ii$)]  The graph sequence  $G_n=(V_n, \Delta_n)$  converges to the graphlets ${\mathcal G}=(\Omega, \Delta)$ where
$\Omega$ is a measure space with measure $\mu$ and  $I -\Delta$ is of rank $1$, i.e., $I- \Delta$ has one
nontrivial eigenvalue $1$. (Equivalently, for each $n$, $I_n- \Delta_n$ has all  eigenvalue $o(1)$ with the exception of  one eigenvalue $1$.)  
\item[($iii$)]   The graph sequence $G_n=(V_n, \Delta_n) $  converges to the graphlets ${\mathcal G}=(\Omega, \Delta)$ where
$\Omega$ is a measure space with measure $\mu$ and    the Laplace operator $\Delta$ on $\Omega$ satisfies
\begin{eqnarray*}
\int_{x \in \Omega} f(x)\big((I-\Delta)g\big)(x) \mu(x)=
\int_{x \in \Omega} f(x) \mu(x) \int_{x \in \Omega} g(x) \mu(x) 
\end{eqnarray*}
for any integrable $f, g: \Omega \rightarrow {\mathbb R}$.
\item[($iv$)] 
The degree distribution $\mu_n$ of $G_n$ converges to $\mu$ and  
\begin{eqnarray*}
\| D_n^{-1/2}\big(A_n-\frac{D_nJD_n}{\vol(G_n)}\big)D_n^{-1/2}  \|
= o(1) \end{eqnarray*}
where  $A_n$ and $D_n$ denote the adjacency matrix and diagonal degree matrix of $G_n$, respectively.
Here $\| \cdot \|$ denotes the usual spectral norm (in $L_2$) and $J$ denotes the all $1$'s matrix.
\item[($v$)] There exists a sequence  $\epsilon_n$ which approaches $0$ as $n$ goes to infinity such that
$G_n$ satisfies   the property $P(\epsilon_n)$, namely,  that the degree distribution $\mu_n$ converges to $\mu$ and  for all $S, T \subseteq V_n$
\begin{eqnarray}
\label{qr}
P(\epsilon_n):~~~~~~~~~~~~~\left| E(S,T) - \frac{\vol(S)\vol(T)}{\vol(G_n)}
\right| \leq \epsilon_n \sqrt{ \vol(S)\vol(T)} 
\end{eqnarray}
where $E(S,T)= \sum_{s \in S, t \in T} A(s,t)$.
\end{description}
\end{theorem}

\begin{remark}{\rm
Before proceeding to prove Theorem \ref{t2}, we note that  a sequence of random graphs with  degree distribution  $\mu_n$ converging to $\mu$ is an 
example satisfying the above properties almost surely.  Here we use  random graph model $G_{\textbf{d}}$  for a given degree sequence $\textbf{d}= (d_v)_{v \in G}$  defined by
choosing $\{u,v\}$ as an edge with probability 
  $d_ud_v/\sum_s d_s$
for any two vertices $u$ and $v$,  (see \cite{cl}).}
\end{remark}

\begin{remark}{\rm 
The above list of   equivalent properties  does not include the  measurement  of  counting  subgraphs. 
Indeed, the problem of enumerating subgraphs in a sparse graph can be  inherently  difficult   because, for example, a random  graph  $G(n,p)$  with $p = o(n^{-1/2})$ 
 contains
very few four cycles. 
Consequently, the error
bounds  could  be proportionally quite large. 
}
\end{remark}

 Instead of counting $C_4$,  we can consider an even cycle $C_{2k}$ or the trace of $(2k)$th power, leading to the following condition:
 
\noindent
{\it  (vi)~~~ For some constant $k$ (depending only on the degree sequence),  a graph sequence $G_n$ satisfies
\[ \left|  \mbox{Trace} (I-\Delta_n)^k -1 \right| =o(1). \]
}
\begin{remark} {\rm Suppose that    in a graph $G_n$,   all eigenvalues of $I -\Delta_n $ except for eigenvalue $1$ are strictly smaller than $1$. Then as $k$ goes to infinity, the 
 trace of the $k$th power of $I-\Delta$  approaches $1$.  How should ($vi$) be modified in a way that it can be an equivalent property to ($i$) through ($v$) ?
 We will leave this as an intriguing  question.}
 \end{remark}
\begin{question}{\rm
Is (vi) equivalent to (i) through (v) for some constant $k$ depending only on $\Omega$?}
\end{question}

\begin{remark}{\rm
It is easily checked that  ($vi$) implies ($ii$).  For the case of dense graphs,  the reverse direction  holds \cite{CGW89}.  For general graphs, to prove  ($ii$) $\rightarrow $ ($vi$) involves the spectral distribution.  
For example, for a regular graph on $n$ vertices and degree $d$,  a necessary condition for ($vi$) to hold is
 that $n d^{k/2} \leq \epsilon_n$.  In particular,  if the spectrum of the  graph satisfies the semi-circle law, then this necessarily condition is also sufficient.
 For a general graph,  the necessary condition should be replaced by  $n \bar d^{k/2} \leq \epsilon_n$  where $\bar d$ is the second order average degree, namely,
 $\bar d = \sum_{v} d^2_v/\sum_v d_v$.  Nevertheless, there are  quasi-random graphs that satisfy  ($ii$) but require $k$ much larger than $2\log n /\log \bar d$.
 For example, we can take the product of a quasi-random graph $G_p$  and a complete graph $K_q$ which is formed  by replacing each vertex of $G_p$  with a copy of $K_q$ and
 replacing each edge in $G_p$  by a complete bipartite graph $K_{q,q}$. 
 
}
 \end{remark}
 
 \begin{question}
{\rm
A subgraph $F$ is said to be  {\it forcing}  if  when the number of occurrence of  $F$ in a graph $G$
is close to (say, within a multiplicative factor of $1+\epsilon$) what is expected in a random graph   with the same degree sequence, then all subgraphs with a bounded number  $k$ of vertices (where $\epsilon$ depends on $k$)  occur in $G$
close to the expected values in a random graph with the same degree sequence. A natural problem is to determine subgraphs which are forcing for quasi-random graphs with general degree sequences.}
\end{question}

\noindent
{\it Proof of Theorem \ref{t2}:}
We will show $ (i) \Rightarrow (v) \Rightarrow (iv) \Rightarrow (iii) \Rightarrow (ii) \Rightarrow (i)$.

We note that $(i) \Rightarrow (v)$ follows from the implications of quasi-randomness for graphs with general degree distributions \cite{quasid}.
Also, $(v) \Rightarrow (iv)$ follows from the fact that $(iv)$ is one of the equivalent quasi-random properties.

To see $(iv) \Leftrightarrow (iii)$, we note that the Laplace operator $\Delta_n$ of $G_n$ satisfies, for any $f, g: V(G_n) \rightarrow {\mathbb R}$,
\begin{align*}
&\left|\int_x f(x)(I-\Delta_n)g(x) \mu_n(x)-
\int_{x} f(x) \mu_n(x) \int_{x } g(x) \mu_n(x) \right|\\
&\\
=& ~~\left| \langle f, (I-\Delta_n)g \rangle_{\mu_n} - \langle f, {\mathbf 1}\rangle_{\mu_n}  \langle g, {\mathbf 1}\rangle_{\mu_n} \right|\\
&\\
=&~~\left| \sum_{u \in V_n} \frac{f(u) A_n g(u)}{\vol(G_n)} - \sum_{u \in V_n} f(u) \mu_n(u) \sum_{v \in V_n} g(v) \mu_n(v) \right|\\
=&~~ \left| f' D_n^{-1/2} \bigg( A_n  - \frac{D_nJD_n}{\vol(G_n)} \bigg)D_n^{-1/2}  g' \right|\\
\end{align*}
where $f'=D_n^{1/2}f/\vol(G_n)$ and $g'= D_n^{1/2} g/\vol(G_n)$.
To prove $(iii) \Rightarrow (iv)$, we have  from ({\it iii}),
\begin{align*}
&\left|\int_x f(x)(I-\Delta)g(x) \mu_n(x)-
\int_{x} f(x) \mu_n(x) \int_{x } g(x) \mu_n(x) \right|\\
\leq&~~\| D^{-1/2} \bigg( A_n  - \frac{D_nJD_n}{\vol(G_n)} \bigg)D^{-1/2} \| \cdot \|f'\| \cdot \|g'\| \\
\leq&~~ \epsilon_n \|f'\| \|g'\|\\
=& ~~\epsilon_n \sqrt{\int f^2(x) \mu_n(x) \int g^2(x) \mu_n(x).
} 
\end{align*}
Since $\mu_n$ converges to $\mu$ and $\epsilon_n$ goes to $0$ as $n$ approaches infinity, $(iv) \Rightarrow (iii)$ is proved.
The other direction can be proved in a similar way.

$(iii) \Rightarrow (ii) $ follows from the fact that $I - \Delta$ is of rank $1$. All adjacency matrices $A_n$ are close to a rank $1$ matrix and therefore
 $\Omega$ is of rank $1$.

To prove that $(ii) \Rightarrow (i) $, we use the fact that for any graph the Laplace operator is a sum of projections of eigenspaces. If $\Delta$ is
of rank $1$, there is only one main eigenspace of dimension $1$  (associated with the Perron vector)  for the normalized adjacency matrix.
\qed

\section{Bipartite quasi-random graphlets with general degree distributions }
\label{bquasi}

We consider the  graph limit of a graph sequence consisting of  bipartite quasi-random graphs with degree distributions converging to some general degree distribution.
The characterizations for  a bipartite quasi-random graph sequence  are similar but different from those of quasi-random graphs.
Because of  the role that  bipartite quasi-random graphlets plays in  general graphlets, we will state a number of equivalent properties. The proof is
quite similar to that for Theorem \ref{t2} and will be omitted. 
\noindent
\begin{theorem}
The following statements are equivalent  for a graph sequence $G_n,$ where $  n = 1, 2, \ldots $.
\label{t22}
\begin{description}
\item[(i)]
 $G_n$'s  form a  bipartite quasi-random sequence with degree distribution converging to $\mu$.
\item[(ii)]   The graph sequence $G_n=(V_n, \Delta_n), $ converges to the graphlets ${\mathcal G}=(\Omega, \Delta)$ where
$\Omega$ is a measure space with measure $\mu$ and  and  $I-\Delta$ has two nontrivial eigenvalues $1$ and $-1$.
 Namely, for each $n$, $I - \Delta_n$ has all eigenvalues $o(1) $ with exceptions of two eigenvalues $1$ and $-1$.
 \item[(iii)]  The graph sequence  $G_n$ converges to the graphlets ${\mathcal G}=(\Omega, \Delta)$ where
$\Omega$ is a measure space with measure $\mu$.  For some $X \subset \Omega$, 
   the Laplace operator $\Delta$  satisfies
\begin{align*}
&\int_{x \in \Omega} f(x)(I-\Delta)g(x) \mu(x)\\
&=
\int_{x \in X} f(x) \mu(x) \int_{x \in\bar X} g(x) \mu (x) 
+\int_{x \in \bar X} f(x) \mu(x) \int_{x \in X} g(x) \mu (x) 
\end{align*}
for any $f, g: \Omega \rightarrow {\mathbb R} $ where $\bar X$ denotes the complement of $X$.
\item[(iv)] The degree distribution $\mu_n$of $G_n$ converges to $\mu$ and 
\begin{eqnarray*}
\| D_n^{-1/2}\big(A_n-\frac{D_n(J_{X,\bar X}+ J_{\bar X,X})D_n}{\vol(G_n)}\big)D_n^{-1/2}  \|
=o(1) \end{eqnarray*}
where  $J_{X,\bar X}(x,y)=1$ if ($x \in X$ and $y \in \bar X$) and $0$ otherwise.
\item[(v)] There exist $X \subset \Omega$ and   a sequence  $\epsilon_n$ which approaches $0$ as $n$ goes to infinity such that
the bipartite graphs $G_n$ satisfies   the property   that the degree distribution $\mu_n$ converges to $\mu$ and  for all $S, T \subseteq V_n$
\begin{align*}
&\left| E(S,T) - \frac{\big(\vol(S\cap X)\vol(T\cap \bar X)+\vol(S\cap \bar X)\vol(T\cap X)\big)}{\vol(G_n)}
\right| \\
&\leq \epsilon_n \sqrt{ \vol(S)\vol(T)} 
\end{align*}
where $E(S,T)= \sum_{s \in S, t \in T} A(s,t)$.
\end{description}
\end{theorem}

\section{Graphlets with rank $2$}
\label{r2}
It is quite natural to  generalize rank $1$ graphlets to  graphlets of higher ranks. The case of rank $2$ graphlets is  particularly of interest, for example,
in the sense for identifying two `communities' in one massive graph.
For two graphs with the same vertex set, the union of two graphs $G_1=(V, E_1)$ and $G_2=(V, E_2)$ has the edge set $E=E_1 \cup E_2$
 and with edge weight $w(u,v)=w_1(u,v)+w_2(u,v)$ if $w_i$ denotes the edge weights in $G_i$.  We will prove the following
 theorem for graphlets of rank $2$.

\begin{theorem}
\label{t31}
The following statements  are equivalent  for a graph sequence $G_n,  $ where $  n = 1, 2, \ldots $. Here we assume that  all $G_n$'s are connected.

\noindent
(i) The graph sequence $G_n=(V_n, \Delta_n) $ converges to  graphlets $\mathcal{G}(\Omega, \Delta)$ and 
 $I-\Delta$  has two nontrivial eigenvalues $1$ and  $\rho \in (0,1)$.
 Namely, or each $n$, $I - \Delta_n$ has all eigenvalues $o(1)$ with the exception of  two eigenvalue $1$ and $\rho_n$ where
 $\rho_n$ converges to $\rho$.
  
\noindent
(ii)  The graph sequence $G_n$ converges to the graphlets ${\mathcal G}=(\Omega, \Delta)$ which is
the union of two quasi-random graphlets (of rank $1$).
 
 \noindent
(iii)   The graph sequence $G_n$  converges to the graphlets ${\mathcal G}=(\Omega, \Delta)$ where
$\Omega$ is a measure space with measure $\mu$  where $\mu = \alpha \mu_1 + (1-\alpha) \mu_2$ for some $\alpha \in [0,1]$ and   the Laplace operator $\Delta$ on $\Omega$ satisfies
\begin{align*}
&\int_{x } f(x)(I-\Delta)g(x) \mu(x)\\
=&
~\alpha \int_{\Omega} f(x) \mu_1(x) \int_{\Omega} g(x) \mu_1(x) +
(1-\alpha) \int_{\Omega} f(x) \mu_2(x) \int_{\Omega } g(x) \mu_2(x) 
\end{align*}
for any $f, g: \Omega \rightarrow {\mathbb R}$.

\ignore{
\noindent
(iii) There exists a sequence  $\epsilon_n$ which approaches $0$ as $n$ goes to infinity such that
the Laplace operator $\Delta_n$ of $G_n$ in a graph sequence satisfies   the property  $I-\Delta_{G_n}$ has two positive eigenvalues $1$ and
$\rho_{G_n}$ with $|\rho_{G_n} -\rho| < \epsilon_n$ and all other eigenvalues have absolute values less than $\epsilon_n$. 
}

\noindent
(iv)
The degree sequence $(d_v)_{v \in V}$ of $G_n$ can be decomposed  as $d_v = d'_v + d''_v$ with $d'_v \geq 0$ and $d''_v \geq 0$. 
The adjacency matrix $A_n$ of $G_n$ satisfies:
\begin{eqnarray*}
\| D_n^{-1/2}\big(A_n-\frac{D'_nJD'_n}{\vol(G'_n)} -\frac{D''_nJD''_n}{\vol(G''_n)}\big)D_n^{-1/2} \|
=o(1) 
 \end{eqnarray*}
where $\vol(G'_n)=\sum_v d'_v$ and $\vol(G''_n)=\sum_v d''_v$.

\noindent
(v) There exists a sequence  $\epsilon_n$ which approaches $0$ as $n$ goes to infinity such that
the degree sequence $(d_v)_{v \in V}$ of $G_n$ can be decomposed  as $d_v = d'_v + d''_v$ with $d'_v \geq 0$ and $d''_v \geq 0$.
Furthermore,  for all $S, T \subseteq V_n$
\begin{eqnarray*}
\left| E_n(S,T) - \frac{\vol'(S)\vol'(T)}{\vol(G'_n)}-\frac{\vol''(S)\vol''(T)}{\vol(G''_n)}
\right| \leq \epsilon_n \sqrt{ \vol(S)\vol(T)}.
\end{eqnarray*}
\end{theorem}

Before we proceed to prove Theorem \ref{t31}, we first prove several key facts that will be used in the proof.
\ignore{
If we require that the graph sequences are not too sparse, then there is more further equivalent statement as stated in Theorem \ref{t4}.
The proof of Theorem \ref{t3}  and Theorem \ref{t4}  will be included in the next section.

\begin{theorem}
\label{t4}
If the graph sequence satisfies the additional condition that $G_n$  on $n$ vertices has the minimum degree at least $c \log n$ for some appropriate
constant $c$, then the statements $(i) \sim (v)$ in Theorem \ref{t3} are equivalent to the following statement:

\noindent
(v) There exists a sequence  $\epsilon_n$ which approaches $0$ as $n$ goes to infinity such that
the adjacency matrix $A_n$ of $G_n$ satisfies   the property  that
the degree sequence $(d_v)_{v \in V}$ of $G$ can be decomposed  as $d_v = d'_v + d''_v$ with $d'_v \geq 0$ and $d''_v \geq 0$ and  $G_n$ is the union of two  graphs $G'_n$ and $G''_n$ with degree sequences $(d'_v)$ and $(d''_v)$, respectively.  Furthermore $G'_n$ and $G''_n$ form quasi-random graph sequences satisfying property $P(\epsilon_n)$.

\end{theorem}

}

\begin{lemma}
\label{mxx}
Suppose that  integers $d_v, d'_v$ and $d''_v$, for $v$ in $V$ satisfy $d_v = d'_v + d''_v$ and $d'_v, d''_v \geq 0 $. Let
$D$, $D'$ and $D''$ denote the diagonal matrices with diagonal entries $d_v$, $d'_v $ and $d''_v$, respectively. Then the matrix $X$ defined by
\[
X= D^{-1/2}\left(\frac{D'JD'}{\vol(G')}+\frac{D''JD''}{\vol(G'')}    \right) D^{-1/2}
\]
has two nonzero eigenvalues $1$ and $\eta$  satisfying
\[
\eta=1-\left(\sum_v \frac{d'_vd''_v}{d_v}\right) \left(\frac {\vol(G)} {\vol(G')\vol(G'')}\right).
\]
The eigenvector $\xi$ which is associated with eigenvalue $\eta$ can be written as
\[
\xi = D^{-1/2} \left(\frac{D'}{\vol(G')} -\frac{D''}{\vol(G'')} \right) {\mathbf 1}.
\]
\end{lemma}
\begin{proof}
The lemma will follow from the following two claims.

\noindent
{\it Claim 1:}~  $\phi_0=D^{1/2} {\mathbf 1}/\sqrt{\vol(G)}$ is an eigenvector of $X$ and $M=D^{-1/2}A D^{-1/2}$.\\
{\it Proof of Claim 1:}
Following the definition of $M$, $\phi_0$ is an eigenvector of $M$.  We can directly verify that  $\phi_0$ is also an eigenvector of $X$ as follows:
\begin{align*}
X \phi_0&= D^{-1/2}\left(  \frac{D'J D'}{\vol(G')}+\frac{D''JD''}{\vol(G'')}    \right) \frac {\mathbf 1}{\sqrt{\vol(G)}}\\
&= D^{-1/2} \left( D' + D''\right) \frac{\mathbf 1}{\sqrt{\vol(G)}}\\
&= D^{-1/2} \frac{D {\mathbf 1}}{\sqrt{\vol(G)}}\\
&= \frac{D^{1/2}{\mathbf 1}}{\sqrt{\vol(G)}}.
\end{align*}

\noindent
{\it Claim 2:}~ $\eta$ is an eigenvalue of $X$ with the associated eigenvector $\xi$.

\noindent
{\it Proof of Claim 2:} We consider
\begin{align*}
X \xi &= D^{-1/2}\left(  \frac{D'J D'}{\vol(G')}+\frac{D''J D''}{\vol(G'')}    \right) D^{-1}  \left( \frac{D'{\mathbf 1}}{\vol(G')} -\frac{D''{\mathbf 1}}{\vol(G'') }\right) 
\\
&= D^{-1/2}\left( \frac{D'{\mathbf 1} }{\vol(G')} \cdot \frac {{\mathbf 1}^*D' D^{-1} D'{\mathbf 1} }{\vol(G')}
- \frac{D'{\mathbf 1} }{\vol(G')} \cdot \frac {{\mathbf 1}^*D' D^{-1} D"{\mathbf 1} }{\vol(G'')} \right)\\
&~~~ + D^{-1/2}\left( \frac{D''{\mathbf 1} }{\vol(G'')} \cdot \frac {{\mathbf 1}^*D'' D^{-1} D'{\mathbf 1} }{\vol(G')}
- \frac{D''{\mathbf 1} }{\vol(G'')} \cdot \frac {{\mathbf 1}^*D'' D^{-1} D''{\mathbf 1} }{\vol(G'')}\right)\\
&= D^{-1/2} \frac{D'{\mathbf 1} }{\vol(G')} \left( \frac {{\mathbf 1}^*D' D^{-1} (D-D''){\mathbf 1} }{\vol(G')}-\frac {{\mathbf 1}^*D' D^{-1} D''{\mathbf 1} }{\vol(G'')}\right) 
\\
&~~~+ D^{-1/2} \frac{D''{\mathbf 1} }{\vol(G'')}  \left( \frac {{\mathbf 1}^*D'' D^{-1} D'{\mathbf 1} }{\vol(G')}-\frac {{\mathbf 1}^*D'' D^{-1} (D-D'){\mathbf 1} }{\vol(G'')}\right)\\
&= D^{-1/2} \left( \frac{D'{\mathbf 1} }{\vol(G')}-\frac{D''{\mathbf 1} }{\vol(G'')} \right)  \left(1- {\mathbf 1}^*D' D^{-1} D''{\mathbf 1} \bigg(\frac {1}{\vol(G')}+ \frac 1 {\vol(G'')} \bigg) \right)\\
&= \eta~ \xi
\end{align*}
as claimed. 

Since $X$ has  rank 2 (i.e., it is the sum of two rank one matrices),   and we have shown that $X$ has eigenvalues  $1$, $\eta$, then the rest  of the eigenvalues are $0$.
\end{proof}

We now apply Lemma \ref{mxx}  using the fact that the normalized adjacency matrix $M=D^{-1/2}A D^{-1/2}$ has eigenvalues $1$ and $\rho=1-\lambda_1$.  Together with Theorem \ref{t2}, we have the following:
\ignore{To establish the implication $(vi) \Rightarrow (iii)$, we will prove the following somewhat stronger statement.}

\begin{theorem}\label{tt1}
Suppose $G$ is the union of two graphs  $G'$ and $G''$with degree sequences $(d'_v)$ and $(d''_v)$ respectively.  Assume both $G'$ and $G''$ satisfy the quasi-random property
$P(\epsilon/2)$ (where $P$ is  one of the equivalent quasi-random properties in Theorem \ref{t2}). Suppose the normalized Laplacian of $G$ has  eigenvalues $\lambda_i=1-\rho_i$, for $i= 0, 1, \ldots, n-1$ with associated orthonormal eigenvectors $\phi_i$.
Then we have:
\begin{enumerate}
\item $\rho_0=1$,
\item  $ \rho_1$ satisfies
\[ 
- \epsilon <1-\rho_1 - \left(\sum_v \frac{d'_vd''_v}{d_v}\right) \left(\frac {\vol(G)} {\vol(G')\vol(G'')}\right) <\epsilon 
\]
\item $|\rho_i| \leq \epsilon$ for $i > 1$.
\item The eigenvector $\phi_1$ associated with $\lambda_1$ can be written as
\[
\phi_1 = D^{-1/2} \left( \frac{D'{\mathbf 1}}{\vol(G')} - \frac{D''{\mathbf 1}}{\vol(G'')}\right) + r
\]
with $\|r\| \leq \epsilon$, where $D'$ and $D''$ denote the diagonal degree matrices of $G'$ and $G''$, respectively.
\end{enumerate}
\end{theorem}
\ignore{
\begin{proof}
Suppose $G$ is the union of two graphs $G'$ and $G''$ . Then the edge sets $E'$ and $E''$ of $G'$ and $G''$, respectively, are disjoint. Thus, the adjacency matrix $A_G$ of $G$ is the sum:
\[
A_G=A_{G'}+A_{G''}.
\]
From the proof of Theorem \ref{t2}, we can write
\begin{align*}
A_{G'} &= \frac{D'JD'}{\vol(G')} + D'^{1/2} R' D'^{1/2}\\
A_{G''} &= \frac{D''J D''}{\vol(G'')} + D''^{1/2} R'' D''^{1/2} 
\end{align*}
where $\|R'\|, \|R''\| \leq \epsilon/2$. 
   
Then the normalized adjacency matrix $M$ of $G$ can be written as
\begin{align*}
M&=D^{-1/2} A D^{-1/2}\\
&=D^{-1/2}\left( \frac{D'JD'}{\vol(G')}+\frac{D''JD''}{\vol(G'')}    \right) D^{-1/2}\\
&~~~+ D^{-1/2} \left(  D'^{1/2} R' D^{1/2}+ D''^{1/2} R'' D^{1/2}   \right) D^{-1/2}\\
&= X + Y
\end{align*}
where $Y=D^{-1/2} \left( D'^{1/2} R' D'^{1/2}+ D''^{1/2} R'' D''^{1/2}   \right) D''^{-1/2}$  has its spectral radius  $\|Y\| <  \epsilon$.

We now use Lemma \ref{mxx}.  Since $M$ differs from $X$ by $Y$,  the eigenvalues $\rho_i$ of $M$ satisfy that $|\rho_1-\eta| \leq  \epsilon$ and $|\rho_i | \leq \epsilon$ for $i \geq 2$.

It remains to show that $z=\phi_1 - \xi$ satisfies $\|z \| \leq 2 \epsilon$. We consider
\begin{align*}
\rho_1 \phi_1 &= M \phi_1 = X \phi_1 + Y \phi_1\\
&= \big(\phi_0\phi_0^*+\eta\xi\xi^*\big)\phi_1+Y\phi_1 \\
&= \eta\langle \phi_1, \xi \rangle \xi + Y \phi_1\\
&=  \eta\langle \phi_1, \xi  \rangle \xi + Y \phi_1
\end{align*}
since $\langle \phi_1, \phi_0 \rangle =0$.
Thus $\langle \phi_1, \xi \rangle^2 \geq 1- \epsilon$ and therefore we can choose the sign for $\xi$ such
that $z=\phi_1 - \xi$  satisfies $\|z\| \leq  \epsilon$.
The proof of Theorem~\ref{tt1} is complete.
\end{proof}
}
\begin{theorem}
\label{mxx1}
Suppose a graphlets  ${\mathcal G}=(\Omega, \Delta)$ is the union of two graphlets ${\mathcal G} ={\mathcal G}_1 \cup {\mathcal G}_2 $ and ${\mathcal G}_i$ are quasi-random graphlets.
Then $I-\Delta$ has two nontrivial eigenvalues $1$ and $\eta$ where $0 < \eta < 1$  satisfies
\[ 1-\eta=\int_{\Omega} \frac{\mu_1(x)\mu_2(x)}{\mu(x)}=\langle \frac{\mu_1}{\mu}, \frac{\mu_2}{\mu} \rangle_\mu,\]
where $\mu_i$ denotes the measure on $\Omega_i$. 
\end{theorem}
\proof
The proof follows immediately from Lemma \ref{mxx} by substituting  $\mu_1(v)=d'(v)/\vol(G')$ and $\mu_2(v)=d''(v)/\vol(G'')$ in Lemma \ref{mxx} and Theorem 
\ref{tt1} before taking limit as $n$ goes to infinity.
\qed

 In the other direction, we prove the following:

\begin{theorem}
\label{q13}
Suppose that the normalized adjacency matrix of  a graph $G$ has  two nontrivial positive eigenvalues   $1$ and $\rho$ and the other eigenvalues satisfy $|\rho_i | \leq \epsilon$ for $2 \leq i \leq n-1$. Then for each vertex $v$, the degree $d_v$ can be written as $d_v = d'_v + d''_v$, with $d'_v, d''_v \geq 0$, so that for any subset $S$ of vertices, the number $E(S)$ of ordered pairs $(u,v)$, with $u,v \in S$ and $\{u,v\} \in E$,  satisfies
\[ 
\left| E(S)- \frac{\vol'(S)^2}{\vol'(G)}-\frac{\vol''(S)^2}{\vol''(G)} \right| \leq 2 \epsilon \vol(S)
\]
where $\vol'(S)=\sum_{v \in S} d'_v$ and $\vol''(S)=\sum_{v \in S} d''_v$.
\end{theorem}

\begin{proof}
Let $\phi_i$,  $0 \leq i \leq n-1$,  denote the eigenvectors  of the normalized adjacency matrix of $G$. Let $\phi_0$ and $\phi_1$ denote the eigenfunctions associated with $\rho_0=1$ and $\rho_1 $.

Since $G$ is connected, 
the eigenvector  $\phi_0$ associated with eigenvalue $\rho_0=1$  of $M_G$ can be written as $\phi_0 = D^{1/2} {\mathbf 1}/\sqrt{\vol (G)}$ as seen in
\cite{ch0}. The second largest eigenvalue $\rho_1$ is strictly between $0$ and $1$ because of the connectivity of $G$. Before we proceed to analyze the eigenvector $\phi_1$ associated with $\rho_1$, we consider the following two vectors which depend on a value $\alpha$ to be specified later.
\begin{align}
f_1&= \alpha D {\mathbf 1} - D^{1/2} \phi_1 \sqrt{\rho_1 \alpha(1-\alpha) \vol(G)} \nonumber\\
f_2&= (1-\alpha) D {\mathbf 1}  + D^{1/2} \phi_1 \sqrt{\rho_1 \alpha(1-\alpha) \vol(G)}\label{eq3}
\end{align}
It is easy to verify that $f_1$ and $f_2$ satisfy the following:
\begin{align}
f_1+f_2& = D {\mathbf 1} \label{fsum}\\
{\mathbf 1} &\perp   \left( \frac {f_1}{\alpha}-\frac{f_2}{1-\alpha}\right) \label{perp}\\
\sum_v f_1(v)&=\alpha \vol(G), \nonumber\\
\sum_v f_2(v)&=(1-\alpha ) \vol(G).\nonumber
\end{align}
In particular,  by considering $\langle f_1, D^{-1} f_2 \rangle$, we see that  $\alpha$ satisfies
\begin{equation}
1-\rho_1= \frac{1}{\alpha (1-\alpha)\vol(G)} \sum_v \frac{f_1(v)f_2(v)}{d_v} .\label{frow}
\end{equation}
and we have
\[
\phi_1= \sqrt{\frac{\alpha (1-\alpha)}{\rho_1 \vol(G)}}D^{-1/2} \left( \frac {f_1}{\alpha}-\frac{f_2}{1-\alpha}\right) .
\]
\noindent
{\it Claim A:}
\[ 
\phi_0 \phi^*_0 + \rho_1 \phi_1 \phi^*_1 = \frac{D^{-1/2} f_1  f^*_1 D^{-1/2}}{\alpha \vol(G)} + \frac{D^{-1/2} f_2 f^*_2 D^{-1/2}}{(1-\alpha) \vol(G)}.
\]
{\it Proof of Claim A:}\\
From \eqref{eq3}, we have
\[
\frac{D^{-1/2} f_1  f^*_1 D^{-1/2}}{\alpha \vol(G)} = \alpha  \frac{D^{1/2}J D^{1/2}}{ \vol(G)} + (1-\alpha) \rho_1 \phi_1 \phi^*_1.
\]

Similarly, we have
\[
\frac{D^{-1/2} f_2  f^*_2 D^{-1/2}}{(1-\alpha)\vol(G)}=(1-\alpha)\frac{D^{1/2}J D^{1/2}}{\vol(G)} + {\alpha} \rho_1 \phi_1 \phi^*_1.
\]
Combining the above two equalities, Claim A is proved.

Now, we define two subsets $X$ and $Y$ satisfying
\begin{align*}
X &= \{ x:  f_1(x) < 0\}= {\bigg \{}x: d_x^{1/2} \leq \phi_1(x)\sqrt{\frac{(1-\alpha)\vol(G)}{\alpha}} {\bigg \}} \\
Y &= \{ y:  f_2(y) < 0\}= {\bigg \{}y: d_y^{1/2} <  -  \phi_1(y)\sqrt{\frac{\alpha \vol(G)}{1-\alpha }}{\bigg \}}.
\end{align*}
Clearly $X$ and $Y$ are disjoint.

Note that  when $\alpha$ decreases, the volume of  $X$ decreases and the volume of $Y$ increases. If $\alpha=1$, $X$ consists of all $v$ with $\phi_1(v) \geq 0$ and $Y$ is empty.  For $\alpha=0$,  $Y$ consists of all $u$ with $\phi_1(u) < 0$ and $X$ is empty. We choose  $\alpha$  so that
\begin{equation}
\label{subsum}
\sum_{x \in X} |f_1(x)|=\sum_{y \in Y} |f_2(y)| .
\end{equation}
Here we use the convention that a subset $X'$ of $X$ means that there are values $\gamma_v$ in $\{0,1\}$, associated  each vertex in $X$  with the exception of  one vertex  with a fractional $\gamma_v$ and the size of $X'$ is the sum of all $\gamma_v$s.

Now, for each vertex $v$, we define ${d}'_v$ and ${d}''_v$ as follows:
\begin{equation}
\label{dprime1}
{d}'_v
=\left\{
\begin{array}[c]{ll}
f_1(v) & \mbox{if } v \not \in X \cup Y, \\
0 & \mbox{if } v \in X, \\
d_v & \mbox{if } v \in Y. 
\end{array}
\right.
\end{equation}
Also, we define ${d}''_v=d_v -d'_v$.

\noindent
{\it Claim B:}  
\begin{align*}
\sum_v {d}'_v  &= \alpha \vol(G) \\
\sum_v {d}''_v &= (1-\alpha) \vol(G)
\end{align*}
{\it Proof of Claim B:}
We note that
\begin{align*}
\sum_v {d}'_v -\alpha\vol(G)&=\sum_v {d}'(v) -\sum_v f_1(v)\\
&= \sum_{v \in X \cup Y} (d'_v -f_1(v))\\
&= \sum_{x \in X} |f_1(x)| + \sum_{y \in Y}( d_y - f_1(y))\\
&= \sum_{x \in X} |f_1(x)| + \sum_{y \in Y} f_2(y)\\
&= \sum_{x \in X} |f_1(x)| - \sum_{y \in Y}| f_2(y)|\\
&=0 .
\end{align*}
The second equality can be proved in a similar way that completes the proof of Claim B.

For a subset $S$ of vertices, let $\chi_S$ denote the characteristic function of $S$ defined by $\chi_S(x)=1$ if $x$ in $S$ and $0$ otherwise.  We consider
\begin{align}
0 &\leq \chi^*_X D^{1/2}M D^{1/2} \chi_Y \nonumber\\
&\leq   \chi^*_X  D^{1/2}( \phi_0 \phi^*_0 + \rho_1 \phi_1 \phi^*_1) D^{1/2} \chi_Y+ {\epsilon} \|  D^{1/2} \chi_X\|~ \| D^{1/2}\chi_Y \| \nonumber\\
&= \frac{ \chi^*_X f_1  f^*_1 \chi_Y}{\alpha \vol(G)} + \frac{\chi^*_X f_2  f^*_2 \chi_Y}{(1-\alpha) \vol(G)} +  \epsilon\sqrt{ \vol(X)\vol(Y)} . \label{eq100}
\end{align}
From the definition, we have  $\chi_X^* f_1 < 0$,  $\chi_Y^* f_1 > 0$, $\chi_X^* f_2 > 0$ and $\chi_Y^* f_2 < 0$.  This implies
\begin{align}\label{eq8}
\epsilon \sqrt{ \vol(X)\vol(Y)}&\geq-\frac{ \chi^*_X f_1  f^*_1 \chi_Y}{\alpha\vol(G)} -\frac{\chi^*_X f_2  f^*_2 \chi_Y}{(1-\alpha) \vol(G) }\nonumber \\
&= \left|\frac{ \chi^*_X f_1  f^*_1 \chi_Y}{\alpha \vol(G)} \right|+\left|\frac{\chi^*_X f_2  f^*_2 \chi_Y}{(1-\alpha) \vol(G) }\right|\nonumber \\
&=\frac{ | f_1^* \chi_X| (\vol(Y)- f_2^* \chi_Y)}{\alpha \vol(G)} +\frac{| f_2^* \chi_Y| (\vol(X)- f_1^* \chi_X)}{(1-\alpha) \vol(G) }\nonumber\\
&=\frac{ | f_1^* \chi_X| (\vol(Y)+|f_2^* \chi_Y|)}{\alpha \vol(G)} +\frac{| f_2^* \chi_Y| (\vol(X)+|f_1^* \chi_X|)}{(1-\alpha) \vol(G) }\nonumber\\
&\geq\frac{| f_1^* \chi_X|}{\vol(G)} \big( \frac{\vol(Y)}{\alpha}+\frac{\vol(X)}{1-\alpha} \big)
\end{align}
by using \eqref{eq3} and \eqref{fsum}. Now,  we have
\begin{align}\label{eq9}
\frac{\vol(Y)}{\alpha}+\frac{\vol(X)}{1-\alpha}&=\alpha\bigg(\frac{\sqrt{\vol(Y)}}{\alpha}\bigg)^2+(1-\alpha)\bigg(\frac{\sqrt{\vol(X)}}{1-\alpha}\bigg)^2\nonumber\\
&\geq \big(\sqrt{\vol(X)} + \sqrt{\vol(Y)}\big)^2\nonumber\\
&\geq 4 \sqrt{\vol(X)\vol(Y)}
\end{align}
by using the Cauchy-Schwarz inequality.
Combining \eqref{eq8} and \eqref{eq9}, we have
\begin{equation}
|f_1^* \chi_X|=|f_2^* \chi_Y| \leq \frac \epsilon 4\vol(G). \label{eqb}
\end{equation}

Now we consider
\[
R = A- \frac{D'J D'}{\sum_v d'_v} - \frac{D'' J D''}{\sum_v d''_v}.
\]
Then,  for $f = \chi_S$, the characteristic function of the subset $S$, we have
\begin{align*}
\langle f, Rf \rangle &= f^*D^{1/2}M D^{1/2}f - \frac{f^*D'J D'f}{\sum_v d'_v} - \frac{f^*D'' J D''f}{\sum_v d''_v}   \nonumber\\
&\leq f^*   D^{1/2}( \phi_0 \phi^*_0 + \rho_1 \phi_1 \phi^*_1) D^{1/2}f \\
&~~~- \frac{f^*D'J D'f}{\sum_v d'_v} - \frac{f^*D'' J D''f}{\sum_v d''_v} + 2 {\epsilon} \|  D^{1/2}f\|^2 \nonumber\\
&\leq  \frac{ f^*f_1  f^*_1f}{\alpha \vol(G)} + \frac{f^* f_2  f^*_2 f}{(1-\alpha) \vol(G)}  - \frac{f^*D'J D'f}{\sum_v d'_v} - \frac{f^*D'' J D''f}{\sum_v d''_v} + 2 {\epsilon} \vol(S)\\
&\leq  \frac{ (f^*f_1)^2- (f^* {\mathbf d}')^2 }{\alpha \vol(G)} + \frac{(f^* f_2)^2-(  f^*{\mathbf d}'')^2}{(1-\alpha) \vol(G) } + 2{\epsilon} \vol(S).
\end{align*}
where ${\mathbf d}'$ and ${\mathbf d}''$ are the degree vectors with entries $d'_v$ and $d''_v$, respectively.  

Since $f = \chi_S$, we have
\begin{eqnarray*}
\frac{ (f^*f_1)^2- (f^* {\mathbf d}')^2 }{\alpha \vol(G)}
&\leq&
\frac{ 2\sum_{v \in S \cap X} |f_1(v)| \vol'(S) + \sum_{v \in S \cap X} |f_1(v)|^2  }{\alpha \vol(G)}\\
&\leq& 3 \epsilon \vol(S) 
\end{eqnarray*}

Similar inequalities hold for $f_2$ and ${\mathbf d}''$.  Thus, we have
\[
\langle f, Rf \rangle \leq 8 \epsilon \vol(S)
\]
The proof of Theorem~\ref{q13} is complete.
\end{proof}

\begin{theorem}
\label{q131}
Suppose ${\mathcal G}=(\Omega,\Delta)$ is a graphlets and $I-\Delta$ has two nontrivial eigenvalues $1$ and $\rho$ with $0 < \rho < 1$.
Then there is a value $\alpha \in [0,1]$ such that 
\\
(i) $\Omega =\Omega_1 \cup \Omega_2$ where $\mu(\Omega_1)=\alpha$ and $\mu(\Omega_2)=1-\alpha$,\\
(ii) $\Omega_i$ has  a measure $\mu_i$ satisfying
\begin{eqnarray*}
\mu_1(x)&=&\mu(x)+\sqrt{\frac {\alpha \rho}{1-\alpha}}\mu(x)\varphi_1(x),\\
\mu_2(x)&=&\mu(x)-\sqrt{\frac {(1-\alpha) \rho}{\alpha}}\mu(x)\varphi_1(x),
\end{eqnarray*} 
where  $\varphi_1$  is the   eigenvector, with $\|\varphi_1\|_{\mu}=1$,  associated with
$\rho$.
\end{theorem}
 The proof of Theorem \ref{q131} follows from the proof in Theorem \ref{q13} and Lemma \ref{mxx}.  Thus, we have $(i) \Leftrightarrow (ii)$.
 
\noindent
{\it Proof of Theorem \ref{t31}:}
We note that in the statement of Theorem \ref{t31}, the implications
$(ii) \Leftrightarrow (iv) \Leftrightarrow (v)$  follow from the definitions and Lemma \ref{mxx}.
It suffices to prove 
 $(i) \Leftrightarrow (ii) $ and $(iii) \Leftrightarrow (iv)$. 

The implication $(ii) \Rightarrow (i)$ is proved in Theorem \ref{tt1},  and Theorems \ref{q13} and \ref{q131} implies $(i) \Rightarrow (ii)$.

To see that $(iii) \Leftrightarrow (iv)$, we note that
if in  a graph $G_n$ in the graph sequence, 
the degree sequence $d_x$ can be written as $d_v=d'_v+d''_v$ for all $v \in V(G_n)$ where $d'_v, d''_v \geq 0$,  then by
defining $\mu^{(n)}_1(v) = d'_v/\sum_v d'_v$,  $\mu^{(n)}_2(v) = d''_v/\sum_v d''_v$ and $\alpha= \sum_vd'_v /\sum_v d_v$, we have
$\mu_n= \mu^{(n)}_1+\mu^{(n)}_2$. Furthermore, we can use the fact that 
\begin{align*}
&\int_x f(x)(I-\Delta_n)g(x) \mu_n(x)&= &~~\frac{1}{\vol(G_n)}\langle f, (I-\Delta_n)g \rangle_{\mu_n} \\
\text{and}~& ~~~~~~~~~~~~\langle f, {\mathbf 1}\rangle_{\mu^{(n)}_1}  \langle g, {\mathbf 1}\rangle_{\mu^{(n)}_1} &=&~~\sum_{u \in V_n} f(u) \mu^{(n)}_1(u) \sum_{v \in V_n} g(v) \mu^{(n)}_2(v) \\
&&=& ~~\frac{fD'_nJD''_ng}{\vol(G_n)^2} .\\
\end{align*}
The equivalence of $(iii)$ and $(iv)$ follows from substitutions using the above two equations and applying Theorem       \ref{t2}.
Theorem \ref{t31} is proved. \qed

For graphlets of rank $2$, there can be  a negative eigenvalue  $-\rho$ of $I - \Delta$ in addition to the eigenvalue $1$. For example,
 bipartite quasi-random graphlets have eigenvalues $1$ and $-1$ for $I-\Delta$.  In general, can
such graphlets  be characterized as the union of a quasi-random graphlet and a bipartite quasi-random graphlets?
To this question, the answer is negative.  It is not hard to construct examples of  a graphlets  having three nontrivial eigenvalues  which
is the union of a quasi-random graphlets and a bipartite quasi-random graphlets. With additional restrictions on degree distributions
and edge density, the three eigenvalues can collapse into two eigenvalues. It is possible to apply similar methods as in the proof of 
Theorem \ref{q13} to derive the necessary and sufficient conditions for such cases but we will not delve into the details here.
\section{Graphlets of rank $k$}
\label{rk}

In this section, we examine graphlets of rank $k$ for some given positive integer $k$. It would be desirable to derive some
general characterizations for graphlets of rank $k$, for example, similar to Theorem \ref{t31}.  However, for $k \geq 3$, the situation is more complicated. Some of the methods for the case of $k=2$ can be extended but some techniques in the proof of Theorem \ref{t31} do not. Here we state a few useful facts about graphlets of rank $k$ and leave some discussion in the last section.

\ignore{
\begin{theorem}
\label{t55}
The following statements  are equivalent for a graph sequence $G_n, n=1,2, \dots$ where all $G_j$'s are connected:

\noindent
(i) $G_n, n=1,2, \dots$ converges to a graphlets $(\Omega, \Delta) $ and  $I-\Delta$  has $k$ nontrivial eigenvalue $ \rho_j \geq 0$, for $j=0, 1, \ldots, k-1$.

\noindent
(ii) $G_n, n=1,2, \dots$ converges to a measure space $\Omega$ and
 $\Omega=\Omega_1 \cup \Omega_2 \cup \ldots \Omega_k$ where $\Omega_j$ has rank $1$ for $j=1,2, \ldots, k$.
 
  \noindent
(iii) $G_n, n=1,2, \dots $ converge to a measure space $ \Omega$ with measure $\mu $ where $\mu = \sum_{i=1}^k\alpha_i \mu_i $ for some $\alpha_i > 0, \sum_i \alpha_i=1 $ and   the Laplace operator $\Delta$ on $\Omega$ satisfies
\begin{align*}
&\int_{x} f(x)(I-\Delta)g(x) \mu(x)\\
&=
\alpha_i \int_{x} f(x) \mu_1(x) \int_{x} g(x) \mu_1(x) + \ldots
+ \alpha_k \int_{x} f(x) \mu_k(x) \int_{x } g(x) \mu_k(x) 
\end{align*}
for any $f, g: \Omega \rightarrow {\mathbb R}$.
 
\noindent
(iii) There exists a sequence  $\epsilon_n$ which approaches $0$ as $n$ goes to infinity such that
the Laplace operator $\Delta_n$ of $G_n$ in a graph sequence satisfies   the property  $\Delta_n$ has a $k$  positive eigenvalue
$\rho_{n,j}<1$ with $|\rho_{n,j} -\rho_j| < \epsilon_n$ and all other eigenvalues have absolute values less than $\epsilon_n$.

\noindent
(iv)
There exists a sequence  $\epsilon_n$ which approaches $0$ as $n$ goes to infinity such that
the degree sequence $(d_v)_{v \in V}$ of $G_n$ can be decomposed  as $d_v = \sum_{j=1}^kd_v^{(j)} $ with $d_v^{(j)} \geq 0$. 
The adjacency matrix $A_n$ of $G_n$ satisfies:
\begin{eqnarray*}
\| D_n^{-1/2}\big(A_n-\sum_{j=1}^{k}\frac{D_n^{(j)}JD_n^{(j)}}{\vol^{(j)}(G_n)} \big)D_n^{-1/2} \|
 \leq \epsilon_n 
 \end{eqnarray*}
 where $D_n^{(j)}$ denotes the diagonal degree matrix with  $(v,v)$th entry $d_n^{(j)}(v)$ and $\vol^{(j)}_n(G_n) =\sum_v d_v^{(j)}$.

\noindent
(v) There exists a sequence  $\epsilon_n$ which approaches $0$ as $n$ goes to infinity such that
the degree sequence $(d_v)_{v \in V}$ of $G_n$ can be decomposed  as $d_v = \sum_{j=1}^k d_v^{(j)}$ with $d'_v \geq 0$ and $d''_v \geq 0$.  Furthermore,
  for all $S, T \subseteq V_n$
\begin{eqnarray*}
\left| E_n(S,T) - \sum_{j=1}^k \frac{\vol^{(j)}(S)\vol^{(j)}(T)}{\vol^{(j)}(G_n)}
\right| \leq \epsilon_n \sqrt{ \vol(S)\vol(T)}.
\end{eqnarray*}
\end{theorem}

}

\ignore{

\begin{theorem}
\label{q12}
Suppose that the normalized adjacency matrix of  a graph $G=(V,G)$ has two nontrivial positive values and all other eigenvalues are smaller than $\epsilon$.  Also, we assume that the minimum degree in $G$ is at least $c\log n$. Then,  the degree sequence $(d_v)$ of  $G$ can be partitioned  such that $d_v = d'_v + d''_v$, with $d'_v, d''_v \geq 0$ such that $G$ is the union of  two edge-disjoint subgraphs $G'$ and $G''$ of $G$  each of which is  quasi-random satisfying $P(12 \sqrt{\epsilon})$ as in (\ref{qr}).
\end{theorem}

\begin{proof}
We will use and follow the notation in the proof of Theorem \ref{q13}. The degrees $d'_v$ and $d''_v$ are chosen by defining $f_1, f_2$ in \eqref{eq3} and in \eqref{fsum} as in the proof of Theorem~\ref{q13}.

We will construct $G'$ and $G''$ as follows:
For each (unordered) pair of vertices $u$ and $v$, we assign $\{u,v\}$ to $G'$ with probability
\[ 
p(u,v)=\frac{\frac{d'_ud'_v}{ \sum_w d'_w}}{ \frac{d'_ud'_v}{\sum_w d'_w}+\frac{d''_ud''_v}{\sum_w d''_w}}, 
\]
and assigned $\{u,v\}$ to $G''$ otherwise. In other words, the adjacency matrix $A'$ of $G'$ has entries as random indicator variables $X(u,v)$ for $\{u,v\} \in E$ such that $X(u,v) =1$ with probability $p(u,v)$ and $0$ otherwise. We define $X(u,v)$ if $\{u,v\}$ is not an edge in $G$.
 
From Claim A in the proof of Theorem \ref{q13}, we have
\begin{equation}
\label{AA1}
A = \frac{D'JD' }{ \vol'(G)} + \frac{D''JD''}{\vol''(G)} +R 
\end{equation}
where $|\langle f, Rf \rangle |\leq 2\epsilon \langle f, D f \rangle$ or, equivalently, $\| D^{-1/2} R D^{-1/2} \| \leq  2 \epsilon$.

The adjacency matrix $A'$ of $G'$ has entries $A'(u,v)=A(u,v)X(u,v)$. The expected value $\E(A'(u,v))$ of $A'(u,v)$ is 
\begin{eqnarray}
\label{ea}
\E(A'(u,v))= A(u,v)p(u,v).
\end{eqnarray}
Let $\bar A'$ denote the matrix with entries $\E(A'(u,v))$ and $\bar M'$ denote the matrix $D'^{-1/2}\bar A' D'^{-1/2} $. The matrices $\bar A''$ and $\bar M''$ are constructed in a similar way for $G''$.  It remains to show the following two claims:

\noindent
{\it Claim (i)}:  With high probability, the normalized adjacency matrix $M'$ of $G'$  has eigenvalues close to those of $\bar M'$.  The same holds for $G''$.

\noindent
{\it Claim (ii)}:  $\bar M'$ has eigenvalues less than $20 \epsilon^{1/4}$ except for one eigenvalue $1$. The same is true for $\bar M''$.

The proof of Claim (i) follows immediately from the matrix concentration inequalities and the resulting bounds on eigenvalues of
a positive matrix and the expected matrix having entries of expected values  using the condition that minimum degrees are bounded below
by $c \log n$ in graphs on $n$ vertices \cite{mcon}.

Before we proceed to prove (ii), we note that the volume $\vol'(G)$, which is  the sum of $X(u,v)$, satisfies
\begin{align*}
\E(\vol(G')) &= \sum_{u,v} \E( A'(u,v))\\
&= \sum_w d'_w + \sum_{u,v} R(u,v)p(u,v).
\end{align*}
where
\[
\bigg|\sum_{u,v} R(u,v)\bigg|\leq {\epsilon} \|D^{1/2} {\mathbf 1}\|^2  \leq {\epsilon} \vol(G).
\]
This also implies that 
\[
\big| \E(\vol(G'')) - \sum_w d''_w \big| \leq  {\epsilon} \vol(G). 
\]

By using the Chernoff inequalities (see \cite{ chern}), we have
\begin{equation}
\label{cher}
\mbox{Prob} \left( |\vol(G') - \E(\vol(G'))|  \geq c \sqrt {\E(\vol(G'))} \right)  < e^{-{c^2}/{3}}.
\end{equation}

To prove Claim (i),  we consider 
\[
\bar A'(u,v)= \frac{  \frac{d'_u d'_v}{\vol' G}}{ \frac{d'_u d'_v}{\vol' G}+ \frac{d''_u d''_v}{\vol''G}} A(u,v)
\]
 
From \eqref{ea}, we have
\begin{align*}
\bar A'(u,v)&= \frac{ \frac{d'_u d'_v}{\vol' G}}{ \frac{d'_u d'_v}{\vol' G}+ \frac{d''_u d''_v}{\vol''G}} \bigg(\frac{D'JD' }{ \vol'(G)} + \frac{D''JD''}{\vol''(G)} +R \bigg)(u,v)\\
&=\frac{d'_u d'_v}{\vol' G}+ \frac{ \frac{d'_u d'_v}{\vol' G}}{ \frac{d'_u d'_v}{\vol' G}+ \frac{d''_u d''_v}{\vol''G}}R(u,v)\\
&=\bigg(\frac{D'(J+\bar R )D'}{\vol'(G)}\bigg)(u,v)
\end{align*}
where
\begin{equation}
\label{AA2}
\bar R(u,v)=\frac{R(u,v)}{\frac{d'_u d'_v}{\vol' G}+ \frac{d''_u d''_v}{\vol''G}}.
\end{equation}
To prove (i), it is enough to show that for any  subset $S $ of vertices, we have
\begin{eqnarray}
\label{rr}
\left|\langle \chi_S, \frac{D'\bar R D'}{\vol'(G)} \chi_S\rangle \right| \leq c \epsilon \vol'(S).
\end{eqnarray}

From \eqref{AA1}, we have
\begin{align}
\label{hy}
2   \epsilon\vol(S) &\geq  |\langle \chi_S,  R \chi_S\rangle | \nonumber\\
&\geq \left|\langle \chi_S, \frac{D'\bar R D'}{\vol'(G)} \chi_S\rangle \right|
\end{align}

To simplify the rest of the proof for Claim (ii), we will write $\vol(G')=\alpha \vol(G)$ and  $\vol(G'')=(1-\alpha) \vol(G)$ with the understanding that the subsequent error terms, expressed as $\epsilon_i$ can still be bounded using \eqref{cher} within a constant factor.

Now, we can write
\begin{align*}
\bigg(\frac{d'_u d'_v }{\vol' G}+ \frac{d''_u d''_v}{\vol''G}\bigg) \vol(G)&=(d_u'+d''_u)(d_v'+d_v'') \\
&~~~+ \alpha \bar \alpha \big( \frac{d'_u}{ \alpha}-\frac{d''_u}{1-\alpha} \big)\big( \frac{d'_v}{\alpha}
-\frac{d''_v}{1-\alpha} \big)
\end{align*}
Therefore
\begin{align*}
R &= \frac {D \bar R D}{\vol(G)} + \frac{\alpha (1- \alpha) \big(\frac{D'}{\alpha}-\frac{D''}{1-\alpha }\big)\bar R \big(\frac{D'}{ \alpha }-\frac{D''}{1-\alpha}\big)}{ \vol(G)} \\
&= D^{1/2} F D^{1/2} + D^{1/2}\bar D F \bar D D^{1/2},
\end{align*}
where
\begin{equation}
\label{barD}
\bar D = \sqrt{\alpha (1- \alpha)} \big( \frac{D'}{ 
\alpha }-\frac{D''}{1-\alpha} \big)D^{-1}.
\end{equation}
Since $2\epsilon\geq  \| D^{-1/2} R D^{-1/2} \| $, we have
\begin{equation}
\big|\|F\| - \| \bar D F \bar D \|\big| \leq 2 \epsilon,
~~~ \mbox{and}~~~
\big| \|\bar D F \bar D \| - \| \bar D^2 F \bar D^2 \| \big| \leq 4 \epsilon. \label{eq20}
\end{equation}

Let  $\theta$ denote a unit (row) eigenvector of $F$ with eigenvalue $\kappa=\|F\|$. Without loss of generality, we also assume that  $\eta$ is positive. We have  the following:
\begin{align}
2 \epsilon &\geq |\kappa +\theta \bar D F \bar D\theta^* |\nonumber \\
&\geq \kappa + \eta (\theta \bar D \theta^*)^2  -\left|\theta' \bar D F \bar D\theta'^* \right|  \nonumber\\ 
&\geq \eta (\theta \bar D \theta^*)^2, \label{barD1}
\end{align}
 where $\theta'=\theta \bar D \mathcal{P}_1 \bar D^{-1}$ where $\mathcal{P}_1$ is the projection $I - \theta^* \theta$. Here $\bar D^{-1}$ denotes the pseudo inverse of $\bar D$. That is, it is the inverse for  the space orthogonal to the zero space of $\bar D$. From  \eqref{barD} and \eqref{barD1}, we have
\begin{align*}
| \theta D' D^{-1}\theta^* - \alpha| &\leq \sqrt{\frac{2 \epsilon \bar \alpha }{\kappa \alpha}}\\
\mbox{and~~~} | \theta D'' D^{-1}\theta^* -(1-\alpha)| &\leq \sqrt{\frac{2 \epsilon \alpha }{\kappa \bar \alpha}}.
\end{align*}
From \eqref{barD1}, we have 
\begin{equation} \label{barD2}
\kappa -4 \epsilon \leq |\theta \bar D F \bar D\theta^*| \leq \kappa \|\theta \bar D\|.
\end{equation}
On the other hand, we have from \eqref{eq20}
\begin{align}
4 \epsilon + \kappa &\geq \|\bar D^2 F \bar D^2 \| \nonumber\\
&\geq \theta \bar D^2 F \bar D^2 \theta^* \nonumber\\
&\geq \kappa (\theta \bar D^2 \theta)^2 + |\theta' \bar D^2 F \bar D^2\theta''| \nonumber\\
&\geq \kappa (\theta \bar D^2 \theta)^2 \label{barD3}
\end{align}
where $\theta''=\theta \mathcal{P}_2 F (\bar D^2)^{-1}$ where $\mathcal{P}_2$ is the projection $\theta^*\theta$.  Combining \eqref{barD2} and \eqref{barD3}, we have
\begin{align*}
\frac{4 \epsilon}{\kappa} + \sqrt{\frac{8 \epsilon}{\kappa}} &\geq |1- \theta \bar D^2 \theta^* | \\
&= \left| 1 -  \frac {1-\alpha } \alpha+ 2 \frac{\theta D''D^{-1}\theta}{\alpha} - \frac{\theta D''^2D^{-2} \theta^*}{1-\alpha \bar \alpha} \right|.
\end{align*}
Therefore we have
\begin{align}
|\theta D'^2D^{-2} \theta^*-\alpha| &\leq  \epsilon_1 \nonumber\\
|\theta (D''D^{-1} -D''^2D^{-2}) \theta^*| &\leq \epsilon_2 \label{d'd''}
\end{align}
where $\epsilon_1 =\frac{4 \epsilon \alpha \bar \alpha}{\kappa} + 4\sqrt{\frac{2 \epsilon \alpha \bar \alpha}{\kappa}}  $ and $\epsilon_2=\epsilon_1+\sqrt{\frac{6 \epsilon  \bar \alpha}{\kappa \alpha}}$.  We define 
\[ 
T=\{u~:~ d'_u \geq d_u (1- \sqrt{\epsilon_1}) ~\mbox{or}~ d'_u \leq \sqrt{\epsilon_2} d_u \}. 
\]
From \eqref{d'd''}, we have
\begin{align*}
\sqrt{\epsilon_2} \sum_{u \not \in T} (\theta(u))^2 & \leq \sum_{u \not \in T} |d''_u-{d''}_u^2| (\theta(u))^2 \\
&\leq \epsilon_2.
\end{align*}
Note that $T= T_1 \cup T_2$ where $T_1=\{u~:~ d'_u \geq d_u (1- \sqrt{\epsilon_2})\}$.  Let $\mathcal{P}_3$ denote the projection as a diagonal matrix with $\mathcal{P}(u,u)=1$ if $u \in T_1$ and $0$ otherwise. We have
\begin{align*}
| \theta \mathcal{P}_3 F  \mathcal{P}_3 \theta - \theta D' D^{-1} F D^{-1}D' \theta | &\leq \epsilon_3,\\
\mbox{and}~~~~~| \theta \mathcal{P}_3 D''D^{-1} F  D^{-1}D'' \mathcal{P}_3\theta  | &\leq \epsilon_3,
\end{align*}
where $\epsilon_3 = 2 \sqrt{\epsilon_2}$.  This implies
\begin{align*}
2 \epsilon &\geq  \left| \frac{ \theta  \mathcal{P}_3 D'D^{-1} F D^{-1}D'  \mathcal{P}_3 \theta   }{ \alpha}+ \frac{\theta  \mathcal{P}_3 D''D^{-1} F D^{-1}D'' \mathcal{P}_3 \theta  }{1- \alpha} \right| \\
&\geq \frac{| \theta  \mathcal{P}_3 D'D^{-1} F D^{-1}D'  \mathcal{P}_3 \theta  | }{ \alpha} -\frac{\epsilon_3}{ 1-\alpha}\\
&\geq \frac{| \theta  \mathcal{P}_3  F  \mathcal{P}_3 \theta  | }{\alpha}- \frac{2\epsilon_3}{1- \alpha} \\
&\geq\frac{| \theta D'D^{-1} F D^{-1}D'  \theta  | }{ \alpha}- \frac{3\epsilon_3}{1-\alpha}
\end{align*}

Hence, for $\epsilon_4=2\epsilon+ \frac{2\epsilon_3}{1- \alpha}$ we have
\begin{align*}
\epsilon_4 &\geq\frac{| \theta  \mathcal{P}_3 F \mathcal{P}_3 \theta  | }{ \alpha}\\
&\geq\frac{| \theta  D'^{1/2}D^{-1/2}  F  D^{-1/2}D' \mathcal{P}_3 \theta  | }{\alpha}.
\end{align*}
This complete the proof of (\ref{rr}) and therefore
 Claim (ii) is proved. The proof of Theorem \ref{q12} is complete.
 \end{proof} 
}

\begin{lemma}
\label{ma12}
Suppose $D$ is the diagonal degree matrix of a graph $G$.
Suppose that  for all $v$ in $V$, $d_v =\sum_{i=1}^k d_i(v)$, for $d_i(v)\geq 0$, $1 \leq i \leq k$. Let
$D_i$  denote the diagonal matrices with diagonal entries $D_i(v,v)=d_i(v)$. Then the matrix $X$ defined by
\[
X= D^{-1/2}\left(\sum_{i=1}^k\frac{D_iJD_i}{\vol_i(G)}    \right) D^{-1/2}
\]
has $k$ nonzero eigenvalues $\eta_i$  where $\eta_i$ are eigenvalues of a $k \times k$ matrix $M$ defined by
\[
M(i,j)=\sum_v \frac{d_i(v) d_j(v) }{d_v}.
\]
Furthermore, the eigenvector $\xi_i$ for $X$ which is associated with eigenvalue $\eta_i$ can be written as
\[
\xi_i (v)= \sum_{j=1}^k \psi_i(j) \frac{d_j(v)d_v^{-1/2}}{\vol_j(G)}
\]
where $\psi_i$ are eigenvectors of $M$ associated with eigenvalues $\eta_i$.
\end{lemma}
\proof
The proof of Lemma \ref{ma12} is by straightforward verification. 
Under the assumption that $\varphi_j M = \eta_j \varphi_i$ for $1 \leq i \leq k$, it suffices to check
that 
$\xi_i X = \eta_i \xi$ for $\xi_i$.  The proof is done by direct substitution and  will be omitted.
\qed

\begin{theorem}
\label{t13}
If a graphlets ${\mathcal G}(\Omega, \Delta)$ is the union of $k$ quasi-random graphlets, then
the Laplace operator $\Delta$ satisfies the property that $I-\Delta$ has $k $ nontrivial positive eigenvalues.
\end{theorem}
The proof of Theorem \ref{t13} follows immediately  from Lemma \ref{ma12}.

Several questions follow the above theorem. If $I-\Delta$ has $k$ eigenvalues that are not necessarily positive,
is it possible  to find a decomposition into   a number of  quasi-random graphlets or bipartite quasi-random
graphlets? Under what additional conditions can such decompositions exist? If they exist, are they
  unique? Numerous additional questions can be asked here.

\ignore{

Let  $\varphi_i$  denote  the orthonormal eigenvector  (under $\mu$-norm) associated with
$\rho_i $.  
From (\ref{aa}), it is enough if we can choose  $g_{i} $ satisfing, for $x,y \in \Omega$,
\begin{eqnarray*}
\mu(x) \big(I-\Delta\big)(x,y)&=& \sum_{i=1}^kg_i(x)  g_i (y).
\end{eqnarray*}
and $g_i (x) \geq 0$ for all all $x \in \Omega \setminus S_i$  for a subset $S_i$  with $\mu(S_i)=0$.

We will need to choose $b_{i,j}\geq 0$, where $0 \leq i,j \leq k-1$, so that
\begin{eqnarray*}
g_i(x)&=&\mu(x)\sum_{j=0}^{k-1} b_{i,j} \varphi_j(x).
\end{eqnarray*}  
In particular, for eigenfunction $\phi_0$ associated with eigenvalue $0$ of $\Delta$, we have
  \begin{eqnarray*}
  \int_{x \in \Omega} g_i(x) = b_{1,j} \int_{x \in \Omega} \mu(x) \varphi_0(x) = \sqrt{\mu(\Omega_j)}.
  \end{eqnarray*}
  It suffices to show that for $S_i = \{ x ~:~ g_i(x) < 0\}$, we have $\mu(\cup_{i=1}^k S_i)=0$.
  To simplify the notation, we write $\mu(\Omega_i)=\alpha_i$.We  set 
  \begin{eqnarray*}
 \mu_i(x) = \begin{cases}\frac{g_i(x) }{\sqrt{\alpha_i} }&\text{if $x \not \in S_i$}\\
0& \text{ otherwise.}
\end{cases}
\end{eqnarray*}

  The proof of Theorem \ref{q131} follows from the following two claims:
  
  \noindent
  {\it Claim A}. If $b_{i,j}$'s  satisfy the following conditions, then $\mu(\cup_{i=1}^k S_i)=0$.
   \begin{align*}
(1)~~~~&\sqrt{\alpha_i\alpha_j} \sum_{s=1}^kb_{i,s}b_{j,s} = \begin{cases} \rho_i &\text{if $i=j$}\\
0&\text{otherwise.}
\end{cases}\\
(2)~~~~& g_1(S_1)=\sum_{i \not = 1} g_i(S_2 \setminus S_1), \text{and}~ \frac{g_2(S_1)}{\alpha_2}=\ldots=\frac{g_k( S_1)}{\alpha_k}.\\
(3)~~~~&
g_2(S_2\setminus S_1)=\sum_{i \not = 2} g_i(S_3 \setminus S_1 \setminus S_2),  \text{and}~ \frac{g_3(S_2\setminus S_1)}{\alpha_3}=\ldots=\frac{g_k( S_2\setminus
S_1)}{\alpha_k}.\\
(4)~~~~&\text{For $1 \leq i \leq k-1$},
g_i\big(S_i \setminus (\cup_{j <i} S_j)=\sum_{j > i} g_j\big(S_{i+1} \setminus (\cup_{j \leq i} S_j)\big), \\
& \text{and}~ \frac{g_{i+1} \big(S_i\setminus (\cup_{j< i} S_j)\big)}{\alpha_{i+1}}=\ldots=\frac{g_k \big(S_i\setminus (\cup_{j<i} S_j)\big)}
{\alpha_k}.
\end{align*}

 \noindent
  {\it Claim B}. There exist  $b_{i,j}$'s  satisfy the above  conditions.
  
  Claim A can be verified similar to the proof in Theorem \ref{q13}. The verification is tedious but straightforward
  we will sketch the proof here. To simplify the description, all the $\epsilon$"s involved will be viewed as the limiting value $0$.
    The key part of proof for Theorem \ref{q13} lies in (\ref{subsum}) and (\ref{eq100}). Instead of $X,Y$ in  (\ref{eq100}), we set
  $X_1=S_1, Y_1=S_2 \setminus S_1 $.  As in the proof at (\ref{subsum}) and (\ref{eq100}),
  we then derive $ g_1(X_1)=\sum_{j>1} g_j(Y_1)=0$. This leads to $\mu(X_1)=0$ (if we use the definition that $S_i = \{ x ~:~ g_i(x) < -\epsilon\}$ for some small fixed $\epsilon$).  For general $i$, we set $X_i=S_i \setminus (\cup_{j <i} S_j)$
  and $Y_i=S_{i+1} \setminus \big(\cup_{j \leq i} S_j \big)$, leading to $g_i(X_i)=\sum_{j>i}g_i(Y_i)=0$.  
  Note that $\cup_{i=1}^k S_i=\cup_{i=1}^{k-1} X_i \cup Y_{k-1}$ and therefore we have $\mu(\cup_{i=1}^k S_i)=0$.
  
  To prove Claim B, we note that 
   there are $\binom {k+1} 2$ constraints in (1) and $\binom k 2$ contraints in (2)-(4). Together, there are $k^2$ constraints. In the other
   direction,
   there  are $k^2$ variables $b_{i,j}$'s. 
\qed

The above theorem leads to the following conjecture which will be examined in a subsequent paper.
\begin{theorem}
\label{t14}
The following statements  are equivalent for a graph sequence $G_n, n=1,2, \dots$ where all $G_j$'s are connected:

\noindent
(i) $G_n, n=1,2, \dots$ converges to a graphlets $(\Omega, \Delta)$ and 
 $I-\Delta$  has $k$ nontrivial  eigenvalues including $a$ positive values and $b$ negative values.

\noindent
(ii) $G_n, n=1,2, \dots$ converges to a measure space $\Omega$ and
 $\Omega=\Omega_1 \cup \ldots \cup \Omega_k$ where $k=a+b$, and for $i \leq a$,  $\Omega_i$,   is a quasi-random graphlets  while, for  $j > a$,
 $\Omega_j$  is a quasi-random bipartite graphlets.
 \end{theorem}
}
\ignore{
\section{Graph  embeddings preseving the Cheeger ratio using the heat kernel }
\label{cheeger}
In a graphlets $\Omega$,  each $x$ is associated with its heat kernel function $H_{t,x}$ which can be viewed as an embedding of $\Omega$ to $[0,1]^{[0,1]}$.
The functions $H_{t,x}$ has many strong properties, in particular, in preserving the Cheeger ratios.

For  a funciton  $f~:~\Omega \rightarrow {\mathbb R}$ and  real value $r \in [0,1]$, the {\it segment}  $S_{r}$  of $f$ denotes a subset  of $\Omega$
with $\mu(S_r)=r$ so that $f(x)/\mu(x) \geq f(y)/\mu(y)$ for any $x \in S_{r}$ and $y \not \in S_{r}$.  In other words, if we arrange elements  $x$ of $\Omega$ in the 
decreasing order of $f(x)/\mu(x)$ and select  elements having large values to form a set $S_{r}$  with  $\mu(S_{r})=r$.   For a given $s$, we consider
 the least Cheeger
ratio $h(S_{r})$ over all  segments $S_r$, $r \leq s$, denoted by $\kappa_{r,f}$, called  the $r$-local Cheeger
ratio determined by  $f$.

We will  examine the  local Cheeger ratio determined by   the heat kernel     functions. The idea of the proof is similar to the proof of the Cheeger inequality in \cite{ch0}. A weaker  discrete version of (\ref{echee1}) was given
in \cite{hk1} with a different proof.
\begin{theorem}
\label{tchee1}
For a subset $S \subset \Omega $ with the Cheeger ratio $h(S)=h$ and $\mu(S)=s \leq 1/4$, there is a subset $T$ of $S$ with $\mu(S) \geq s/2$ such that for $x
\in T$,  the heat kernel   $H_{2t,x}$ satisfies:
\begin{eqnarray}
\label{echee1}
-   \frac { \frac{\partial}{\partial t} H_{2t,x}(x)}{H_{2t,x}(x)} \geq 
\frac {(1-\delta)} 2 \kappa_{t,x,2s}^2 
\end{eqnarray}
provided $\delta = 1-e^{-4th} \leq 1/4$ where $\kappa_{t,x,2s} $ is the($2s$)-local Cheeger ratio  determined by $H_{t,x}$.
\end{theorem}

\proof
For $x$ in $T$, we define 
\[  f_t(y)=\frac{H_t(x,y)}{\mu(y)}. \]

\noindent
{\it Claim A:}
For $\epsilon=1-e^{-2th}$,  the heat kernel  rank function  $f_t$ satisfies the following:
\begin{align}
(a)~~~~~~~~~~~~&-\frac{\partial}{\partial t} f_{2t} (x) = 2 \int_{\Omega} |\nabla f_t(e)|^2 \mu_e, \label{m4a}\\
(b)~~~~~~~~~~~~&~~~~~~~f_{2t}(x) = \int_{\Omega} f_t^2(y) \mu(y), \label{m4b}\\
(c)~~~~~~~~~~~~&\text{For $r \geq s$,} ~~\int_{ S_{r}} f_t(y)\mu(y)  = 1-\epsilon, \label{m4c} \\
(d)~~~~~~~~~~~~&\text{For $r \geq 2s$,}~~\inf_{S_{r}} f_t(y)\mu(y) \leq \frac{2\epsilon}{r}.\label{m4d}
\end{align}

\noindent
{\it Proof of the Claim:}\\
($a$) and ($b$) are consequences of Lemma \ref{m1} (ii) and (iii).
Theorem \ref{t333} implies ($c$). Now consider $S_r$ with $r \geq 2s$ and we have
\begin{eqnarray*}
\inf_{y \in S_{r}}f_t(y) \leq \frac 2 {r } \int_{ S_{r}\setminus S_{r/2}}f_t(y)\mu(y)  \leq 
\frac{2\epsilon}{r} .
\end{eqnarray*}
The Claim is proved.

Using the above claim, we have the following.
\begin{eqnarray}
 \frac {-\frac{\partial}{\partial t} f_{2t,x}(x)}{f_{2t,x} (x)}
&=& \frac{2 \int_\Omega |\nabla f_t(e)|^2 \mu_e}{\int_\Omega f_t^2(y) \mu(y)}.\label{nn}
\end{eqnarray}
For a graph $G$, we can define the operator $\bar \nabla $ for $u,v$ with $A(u,v) \not = 0$, by
\[ \bar \nabla f(u,v)= f(u) + f(v)= 2 f(u)- \nabla f(u,v). \]
The operator $\bar \nabla$ can be extended to the graph limit $\Omega$, satisfying:
\begin{eqnarray}
 |\nabla f(e) \bar \nabla f(e)|&= &|\nabla f^2(e)|,\label{nn1}\\
 \int_\Omega |\bar \nabla f(e)|^2 \mu_e &=& 2 \int_\Omega f^2(y) \mu(y)- \int_\Omega | \nabla f(e)|^2 \mu_e  \nonumber\\
& \leq& 2 \int_\Omega f^2(y) \mu(y). \label{nn2}
 \end{eqnarray}
Returning to (\ref{nn}) and using (\ref{nn1}) and (\ref{nn2}), we have
\begin{eqnarray*}
 \frac {-\frac{\partial}{\partial t} f_{2t}(x)}{f_{2t} (x)} &=&  \frac{2 \int_\Omega |\nabla f_t(e)|^2 \mu_e\int_\Omega |\bar \nabla f_t(e)|^2 \mu_e
}{\int_\Omega f_t^2(y) \mu(y)\int_\Omega |\bar \nabla f_t(e)|^2 \mu_e}\\
&\geq&\left(\frac{\int_\Omega |\nabla f_t^2(e)| \mu_e
}{\int_\Omega f_t^2(y) \mu(y)}\right)^2\\
&\geq&\left(\frac{\int_0^{2s}dr\int_{\partial(S_r)} |\nabla f_t^2(e)| \mu_e
}{\int_\Omega f_t^2(y) \mu(y)}\right)^2\\
&=&\left(\frac{\int_0^{2s}\mu(\partial(S_r))~ d f_t^2(r)
}{\int_\Omega f_t^2(y) \mu(y)}\right)^2\\
&\geq&\left(\frac{\int_0^{2s}\kappa_{t,x,2s} \mu(S_r)~  d f_t^2(r)
}{\int_\Omega f_t^2(y) \mu(y)}\right)^2\\
&\geq&\kappa_{t,x,2s}^2 \left(\frac{\int_0^{2s}f_t^2(r)~ d \mu(S_r) 
}{\int_\Omega f_t^2(x) \mu(x)}\right)^2\\
&=&\kappa_{t,x,2s}^2 \left(\frac{\int_{S_{2s}} \big(f_t^2(y) -f_t^2(r)) \mu(y) 
}{\int_\Omega f_t^2(y) \mu(y)}\right)^2\\
\end{eqnarray*}
We now use (\ref{m4c}) and (\ref{m4d})  to obtain:
\begin{eqnarray*}
f^2(2s) 2s &\leq& \frac{2 \epsilon^2} s\\
&\leq& \frac{2 \epsilon^2} {(1-\epsilon)^2 s}\bigg( \int_{S_{2s}} f_t(y) \mu(y) \bigg)^2\\
&\leq& \frac{4 \epsilon^2} {(1-\epsilon)^2 } \int_{S_{2s}} f_t^2(y) \mu(y) 
\end{eqnarray*}
by using H\"{o}lder's inequality.
Therefore we have
\begin{eqnarray*}
 \frac {-\frac{\partial}{\partial t} f_{2t}(x)}{f_{2t} (x)}
 &\geq&\kappa_{t,x,2s}^2 \left(\frac{\int_{S_{2s}} f_t^2(y)\mu(y) -f_t^2(2s) 2s
}{\int_\Omega f_t^2(y) \mu(y)}\right)^2\\
&\geq &((1-\epsilon)^2-4\epsilon^2) \kappa_{t,x,2s}^2\\
&\geq &(1-\delta) \kappa_{t,x,2s}^2
\end{eqnarray*}
 by our assumption on $\delta$.
Theorem \ref{tchee1} is proved.
\qed

\begin{theorem}
\label{tchee2}
For a subset $S$ with $\mu(S)=s$ and $h(S)=h$, there is a subset $T$ of $S$  with $\mu(T) \geq s/2$ so that
for   $\delta = 1-e^{-4th} \leq 1/4$, the following holds:
\begin{align*}
(i) &~ \text{For $x \in T$},~~
 e^{-2th} \leq \int_{S} H_{t,x}(y)  \leq \sqrt{\frac {s}{\mu(x)} }e^{-{t\kappa_{t,x,2s}^2(1-\delta)}/2}\\
 &\text{where where $\kappa=\kappa_{t,x,2s} $ denotes the minimum $2s-$local Cheeger ratio  using }\\
 &\text{ segments $S_r$, $r \leq 2s$, determined by $H_{t,x}$.}
\\
(ii) &~~
\frac 1 2  e^{-4th} \leq \int_T \frac{\mu(x)}s \Big( \int_{S_s} H_{t,x}(y)\Big)^2 \leq \frac{1}{\mu(T)} \int_T H_{2t,x}(x) \mu(x)  \leq e^{-t\kappa_{t,2s}^2(1-\delta)}\\
 &\text{where where $\kappa=\kappa_{t,2s} $ denotes the minimum $2s-$local Cheeger ratio  using }\\
 &\text{ segments $S_r$, $r \leq 2s$, determined by $H_{t,x}$ over all $x$ in $T$.}
 \end{align*}
\end{theorem}

We note that in ($i$), there is a term $\mu(S)/\mu(x)$ which contributes a term of  $O(\log n)$ when we compare $h$ and $\kappa_{t,x,2s}$ by taking logarithms
of both sides of the inequality. This works for the discrete cases of finite graphs but is rather restrictive for the general case.
In order to get rid of this untidy factor, we consider the expected case in ($ii$) by integrating over all $x$ in $T$.

\noindent
{\it Proof of Theorem \ref{tchee2}:}\\
By solving the inequalities in (\ref{echee1}), we have
\begin{eqnarray}\label{chee3}
\hkr_{2t,x}(x) \leq   e^{-(1-\delta)t \kappa_{t,x,2s}^2}.
\end{eqnarray}

To prove the upper bound, we use  Lemma \ref{m1} ($iii$):
\begin{eqnarray*}
H_{2t,x}(x)&\geq&  \mu(x) \int_\Omega \frac{H_{t,x}(y)^2}{\mu(y)}dy\\
&\geq&\mu(x) \int_{ S_s}\big(H_{t,x}(y)/\mu(y)\big)^2\mu(y)dy\\
&\geq& \frac{\mu(x)} s \bigg( \int_{ S_s}H_{t,x}(y)dy\bigg)^2
\end{eqnarray*}

Together with (\ref{chee3}),  we have
\begin{eqnarray*}
 e^{-t\kappa_{t,x,s}^2(1-\delta)}& \geq& H_{2t,x}(x) \geq 
\frac{\mu(x)} s\Big( \int_{S_s} H_{t,x}(y) \Big)^2 \\
&\geq& \frac{ \mu(x)e^{-4th}}{ s} 
\end{eqnarray*}
where $\epsilon$ is as defined  in the proof of Theorem \ref{tchee1}.
This implies
\[ e^{-2th} \leq \int_{S_s} H_{t,x}(y)dy  \leq \sqrt{\frac {s}{\mu(x)} }e^{-{t\kappa_{t,x,2s}^2(1-\delta)}/2 }\]
and ($i$) is proved.
To prove ($ii$), we consider
\[ F(t)=  \frac{1}{\mu(T)} \int_T H_{2t,x}(x) \mu(x) . \]
By Theorem \ref{tchee1}, we have
\begin{eqnarray*}
-\frac{\partial }{\partial t} \log F(t) &=&
\frac{ -\frac{1}{\mu(T)} \int_T \frac{\partial}{\partial t} H_{2t,x}(x) \mu(x)}{F(t)} \\
 &\geq&\frac{\frac{1}{\mu(T)} \int_T (1-\delta) \kappa^2_{t,x,2s} H_{2t,x}(x) \mu(x)}{F(t)} \\
  &\geq&\frac{\frac{1}{\mu(T)} \int_T (1-\delta) \kappa^2_{t,2s} H_{2t,x}(x) \mu(x)}{F(t)} \\
   &\geq&(1-\delta) \kappa^2_{t,2s}.
\end{eqnarray*}
By solving the inequality with  $F(0) =1$,
we have
\[ F(t) \leq e^{-(1-\delta)t \kappa^2_{t,2s}}. \]
In the other direction, we have
\begin{eqnarray*}
F(t) &\geq&\frac{1}{\mu(T)} \int_T \Big( \frac{H_{2t,x}(x)}{\mu(x)}\Big) \mu(x)^2\\
&=&\frac{1}{\mu(T)} \int_T \mu(x)^2\Big( \int_{\Omega} \frac{H_{t,x}(y)^2}{\mu(y)}\Big) \\
&\geq&\frac{1}{\mu(T)} \int_T \mu(x)^2\Big( \int_{S_s} \frac{H_{t,x}(y)^2}{\mu(y)}\Big) \\
&\geq&\frac{1}{\mu(T)} \int_T \mu(x)^2\frac{1}{s} \Big( \int_{S_s} H_{t,x}(y)\Big)^2 \\
&\geq&\frac{1}{2} \int_T\Big( \frac{\mu(x)}{\mu(T)}\Big)^2 e^{-4th}\\
&\geq&\frac{e^{-4th}}{2}\Big( \int_T \frac{\mu(x)}{\mu(T)}\Big)^2 \\
&\geq&\frac{e^{-4th}}{2}
\end{eqnarray*}
as desired.
\qed

We remark that Theorems \ref{tchee1} and \ref{tchee2}  provide an algorithm for finding a local cut together with a performance guarantee within a quadratic factor of the optimum.
The above inequality has a similar flavor as the Cheeger inequality  which is the base for the  spectral partitioning algorithm using eigenfunctions. 
The disadvantage of using   eigenfunctions is the partition algorithm is not local in the sense that one can not specify the size of the parts to be separated. Here by the appopriate
choice of $t$, we can control the support of the heat kernel functions.
}

\section{Concluding remarks}\label{remarks}
In this paper, we have merely scatched the surface of the study of graphlets. Numerous questions remain, some of which we mention here.

\noindent
(1) In this paper, we mainly study quasi-random graphlets and graphlets of finite rank (which are basically `sums' of quasi-random graphlets).
It will be quite essential to understand other families of graph sequences, such as the graph sequences of paths, cycles, trees, grids, planar graphs, etc. In this paper, we define the spectral distance between two graphs as the spectral norm of the  `difference' of the associated Laplacians. 
In a subsequent paper, we consider a generalized version of spectral distance for considering large families of graphlets.

\noindent
(2) We here use $[0,1]$ as the labels for the graphlets and the measure $\mu$ of the graphlets depends on the Lebesgue measure on $[0,1]$.
To fully understand the geometry of graphlets derived from general graph sequences, it seems essential to consider general measurable spaces as labeling spaces.
For example, for graph sequences $C_n \times C_n$, it works better to use $[0,1]\times[0,1]$ as the labeling space, instead.

\noindent
(3) In this paper we relate 
the  spectral distance to the previously studied cut-distance by showing the
equivalence of the two distance measures for graph sequences of any degree distribution. It will be of interest to find  and to relate to other distances.
For example,  will  some nontrivial subgraph count measures be implied by the spectral distance (see the questions and remarks mentioned in Section \ref{quasi})?

\noindent
(4)
In the study of complex graphs  motivated by numerous real-world networks, random graphs are often utilized for analyzing various  models  of  networks. Instead of using the classical Erd\H{o}s-R\'enyi model, for which graphs have the same expected degree for every vertex, the graphs  under consideration usually have prescribed degree distributions, such as a power law degree distribution. For example, for a given expected degree sequence ${\mathbf w}=(d_v)$, for $v \in V$, a random graph $G({\mathbf w})$ has edges between $u$ and $v$ with probability $p d_ud_v$, for some scaling constant (see \cite{cl}).  Such random graphs are basically quasi-random of positive rank one.  Nevertheless,   realistic networks often are clustered or have uneven distributions.  A natural problem of interest is to identify the clusters or `local communities'.  The study of graphlets of  rank two or higher can be regarded as extensions of the previous models. Indeed,  the geometry of the graphlets  can be used to illustrate the limiting behavior of large complex networks. 
In the other direction, network properties that are ubiquitous in many examples of real-world graphs can be a rich source for new directions in graphlets. 

\noindent
(5)
Although we consider undirected graphs here, some of these questions can be extended to directed graphs.  In this paper, we focus on the spectral distance of graphs but  for directed graphs the spectral gaps can be exponentially small and any diffusion process on directed graphs can have very different behavior.
The treatment for directed graphs will need to  take these considerations into account.
Many questions remains.

\end{document}